\documentclass[22pt,epsf]{article}
\usepackage{color}
\usepackage{indentfirst}
\newcommand{\bed}{\begin{displaymath}}
\newcommand{\eed}{\end{displaymath}}
\newcommand{\bea}{\bed\begin{array}{rl}}
\newcommand{\eea}{\end{array}\eed}
\newcommand{\barray}{\begin{array}{ll}}
\newcommand{\earray}{\end{array}}

\usepackage{color,latexsym,amsfonts,amssymb,amsmath,amsthm}
\usepackage{amsfonts,amsmath,epsfig,subfigure,cite}
\usepackage{graphicx}
\usepackage{epsfig}
\usepackage{float}
\usepackage{caption}
\usepackage{mathrsfs}
\usepackage{multirow}
\newtheorem{theorem}{Theorem}[section]
\newtheorem{lemma}[theorem]{Lemma}
\newtheorem{remark}{Remark}[section]

\newtheorem{corollary}[theorem]{Corollary}

\newtheorem{assumption}{Assumption}[section]
\newtheorem{definition}{Definition}[section]

\numberwithin{equation}{section}

\allowdisplaybreaks[4]
\begin{document}

\title{Approximation of Invariant Measures for Stochastic Differential Equations with Piecewise Continuous Arguments via Backward Euler Method\thanks{This work is supported by the NNSFC (NOs.  91630312, 11711530017 and 11871068) and National Postdoctoral Program for Innovative Talents (NO. BX20180347)}}
\author{Chuchu Chen,~~Jialin Hong,~~Yulan Lu \thanks{Corresponding author. Email: chenchuchu@lsec.cc.ac.cn,~hjl@lsec.cc.ac.cn,~yulanlu@lsec.cc.ac.cn
\vspace{6pt}}\\
\footnotesize{LSEC, ICMSEC, Academy of Mathematics and Systems Science, Chinese Academy of Sciences,}\vspace{-6pt}\\
\footnotesize{Beijing 100049, People's Republic of China}\\
\footnotesize{School of Mathematical Sciences, University of Chinese Academy of Sciences, }\vspace{-6pt}\\
\footnotesize{Beijing 100049, People's
Republic of China}}
\maketitle
\begin{abstract}
For the stochastic differential equation (SDE) which has piecewise continuous arguments (PCAs), is driven by multiplicative noises and its drift coefficients are dissipative, we show that the solution at integer time is  a Markov chain and admits a unique invariant measure. In order to inherit numerically the invariant measure of SDE with PCAs, we apply the backward Euler (BE) method to the equation, and prove that the numerical solution at integer time is not only Markovian but also reproduces a unique numerical invariant measure. We present the time-independent weak error analysis for the method  under certain hypothesis. Further, we show that  the numerical invariant measure converges to the original one with order 1. Numerical experiments verify the theoretical analysis.



 \end{abstract}
\section{Introduction}

Differential equations with piecewise continuous arguments (PCAs), which play an important role in biomedicine, physics, neural networks, control theory, etc,  represent a hybrid of continuous and discrete dynamical systems and thus combine properties of both differential and difference equations\cite{Wiener1993}.  A typical differential equation with PCAs is of the form
\begin{equation}\label{DEPCA}
X^\prime (t)=f(t,X(t),X(\alpha(t))),
\end{equation}
where $\alpha(t)$ (e.g., $\alpha(t)=[t]$) has intervals of constancy, and $[\cdot]$ denotes the greatest-integer function.
The discontinuity of $\alpha(t)$ may make \eqref{DEPCA} exhibit complex and extraordinary dynamical behavior, such as stability, oscillation, ergodicity, periodicity and chaos.  In practice, stochastic factors like environment noise or accidental events may greatly influence a system. Thus, stochastic differential equations (SDEs) with PCAs attract lots of attention. For the well-posedness, mean-square stability and almost sure stability of SDEs with PCAs, we refer to \cite{zhangling,Lu2017,Mao2016} and references therein.


There have been many works on the numerical approximations of SDEs with PCAs (see e.g., \cite{zhangling,Lu2017,Xie2019,mm,Lu2019}).
Under the local Lipschitz and the general Khasminskii-type conditions, authors in \cite{zhangling} prove the convergence of explicit Euler method in probability for SDEs with PCAs. Recently, the strong convergence of the split-step $\theta$ method is proved under a coupled monotone condition in \cite{Lu2017}, and the convergence rate is obtained with some polynomial growth conditions in \cite{Lu2019}.  For the time-dependent SDEs with PCAs, the strong convergence of the one-leg $\theta$ method is investigated in \cite {Xie2019}. In terms of stability, the split-step $\theta$ method and the one-leg $\theta$ method are proved to inherit the mean-square stability of the original  equations under some dissipative conditions in \cite{Lu2017} and \cite{Xie2019}, respetively.
Based on the convergence of the Euler method in finite time intervals, the equivalence between the mean-square exponential stability of a kind of retarded SDEs with PCAs and that of its Euler method with sufficiently small step-size is established in \cite{mm}.

As far as we know, besides stability, the invariant measure  also plays significant roles in describing  the long-time behavior of a dynamical system. For SDEs and stochastic partial differential equations (SPDEs), there have been plenty of works on invariant measure. Da Prato \cite{Prato1}  provides two common approaches which ensure that the semigroup generated by the solutions of SDEs or SPDEs admits a unique invariant measure, i.e., it is ergodic.  One is the  Krylov-Bogoliubov theorem together with Doob's theorem, which is efficient to deal with equations with non-degenerate noises. Another one is the remote start or dissipative method which usually deals with equations with general noises.  As for approximations of invariant measures,  one important topic is to construct numerical schemes to inherit the invariant measure and to give the error estimates between the numerical invariant measure and the original one, see, e.g., \cite{Talay1990,Chen2016} by Kolmogrov equation, \cite{Mattingly2010} via Poisson equation, \cite{Brehier2014,Cui2018} by Malliavin calculus, \cite{Hong2017} by generating functions.
For the error analysis of the invariant measures, we also mention that \cite{Vilmart2014} provides new sufficient conditions for a numerical method to approximate the invariant measure of an ergodic SDE with high order of accuracy, independently of the weak order of the method.

However,  to our best knowledge, there is no result on the invariant measure of both SDEs with PCAs and their numerical approximations. In this paper, our aim is to  make a contribution on that of the following SDE with PCAs
\begin{eqnarray}\label{eq}
\begin{cases}
dX(t)=f(X(t),X([t]))dt+g(X(t),X([t]))dB(t)\qquad t\geq 0,\\
 X(0)=x
\end{cases}
\end{eqnarray}
and its backward Euler (BE) method. Here $x\in \mathbb{R}^d $ is the initial value, $f: \mathbb{R}^d\times \mathbb{R}^d\rightarrow \mathbb{R}^d$, $g: \mathbb{R}^d\times \mathbb{R}^d\rightarrow \mathbb{R}^{d\times r}$  and $B(t)$ is an $r$-dimensional Brownian motion defined on a filtered complete  probability space   $(\Omega,\mathcal{F},\{\mathcal{F}_t\}_{t\geq 0}, \mathbb{P})$.
In this paper, we  concern with the following questions.
\begin{itemize}
  \item [(I)] Does the solution of Eq. \eqref{eq} admit an invariant measure? If it does, is it unique?
  \item [(II)] If Eq. \eqref{eq} admits a unique invariant measure, does the BE method reproduce a unique numerical invariant measure?
  \item [(III)] Does the numerical invariant measure, if it exists, converge to the original one ?
\end{itemize}

In order to answer the questions above, we assume that the drift coefficient $f$ is dissipative and the diffusion coefficient $g$ satisfies the global Lipschitz condition.
For stochastic functional differential equations (SFDEs) with either continuous or discrete delay arguments, it is well known that their solutions are non-Markovian because of the dependence on their history. However, the segment process of SFDEs with continuous arguments is proved to be Markovian \cite{Mohammed1984}. The invariant measure of SFDEs with continuous arguments has been studied extensively (see \cite{Bao2014,Prato1} and references therein). For \eqref{eq}, we prove that the solution $\{X(k)\}_{k\in\mathbb{N}}$ at integer time is a time-homogeneous Markov chain  under above assumptions. This reveals the influence of the discrete arguments and reflects the characteristic of  difference dynamics of \eqref{eq}. By proving the exponential convergence  of $\{X(k)\}_{k\in\mathbb{N}}$ in distribution and the continuous dependence on initial values of $\{X(k)\}_{k\in\mathbb{N}}$ under the dissipative condition,  we then obtain that the Markov chain $\{X(k)\}_{k\in\mathbb{N}}$ is exponentially ergodic with a unique invariant measure $\pi$.

Taking the divergence of explicit Euler method without the linear growth condition on drift coefficients into consideration, we apply the implicit BE method to discretize Eq. \eqref{eq}. Denoting $Y_k$ the BE approximation of $X(k)$ (i.e. $X(t)$ at integer time $t=k$), we show that $\{Y_k\}_{k\in\mathbb{N}}$ also possesses the time-homogenous Markov property. Then we prove that $\{Y_k\}_{k\in\mathbb{N}}$ is uniformly bounded in mean square sense and continuously dependent on initial values, which guarantees the existence and uniqueness of the numerical invariant measure $\pi^{\delta}$. Moreover, the transition probability measure of  $\{Y_k\}_{k\in\mathbb{N}}$  converges exponentially to $\pi^{\delta}$ as $k$ tends to infinity, i.e. the BE method preserves the exponential ergodicity of Eq. \eqref{eq}. The error between $\pi$ and $\pi^{\delta}$ is estimated via deducing the weak error between $X(k)$ and $Y_k$, which is required not only to be independent of $k$ but also to decay exponentially. The main difficulty  is to deriving several uniform priori estimations via Malliavin calculus.  Based on the weak error analysis, we show that $\pi^{\delta}$ converges to $\pi$ with order 1 which coincides with the weak convergence order of the BE method.

This paper is organized as follows. In Section 2, some notations are introduced and the solution of Eq. \eqref{eq} at integer time is proved to be a time-homogeneous Markov chain as well as  exponentially ergodic with a unique invariant measure. In Section 3, we apply the BE method to Eq. \eqref{eq} and prove that  the BE approximation at integer time preserves the exponential ergodicity with a unique numerical invariant measure. The time-independent weak error
of the solutions together with the error between invariant measures are given in Section 4. In Section 5, numerical experiments are presented to verify the theoretical results.

\section{Notations and Invariant Measures of the Solution}
To begin with, we introduce some notations. 
 Let $(\mathbb{R}^d,\langle \cdot,\cdot\rangle,\|\cdot\|)$ be a $d$-dimensional real Euclidean space. Given a matrix $A\in\mathbb{R}^{d\times r}$, its trace norm is defined as $\|A\|:=\sqrt{\text{trace}(A^TA)}$.  
Assume that $C^{1,2}(\mathbb{R}_+\times\mathbb{R}^d;\mathbb{R}_+)$ denotes the family of all real-valued functions $V(t,x)$ defined on $\mathbb{R}_+\times\mathbb{R}^d$ such that they are continuously twice differentiable in $x$ and once in $t$.
$B(x,r)$ denotes the open ball in $\mathbb{R}^d$ with center $x$ and radius $r>0$. Denote by $C_b(\mathbb{R}^d)$ (resp. $B_b(\mathbb{R}^d)$) the Banach space of all uniformly continuous and bounded mappings (resp. Borel bounded mappings) $\varphi: \mathbb{R}^d\rightarrow\mathbb{R}$
endowed with the norm $\|\varphi\|_0=\sup_{x\in\mathbb{R}^d}|\varphi(x)|$. For any $k\in \mathbb{N}$, $C_b^k(\mathbb{R}^d)$ is the subspace of $C_b(\mathbb{R}^d)$ consisting of all functions with bounded partial derivatives $D_x^{i}\varphi(x)$ for $1\leq i\leq k$ and with the norm $\|\varphi\|_k=\|\varphi\|_0+\sum_{i=1}^k\sup_{x\in\mathbb{R}^d}\|D_x^{i}\varphi(x)\|$. The notation $\mathscr{P}(\mathbb{R}^d)$ denotes the family of all probability measures on $(\mathbb{R}^d,\mathcal{B}(\mathbb{R}^d))$.
 For $a,b\in\mathbb{R}$, we denote $\max(a,b)$ and $\min(a,b)$ by $a\vee b$ and $a\wedge b$, respectively. We define $\inf\emptyset=\infty$ and denote by  $\mathbf{1}_{D}$ the indicative function of a set $D$.

Now, we make the following assumptions on the drift and diffusion coefficients.
   \begin{assumption}\label{assumption1}
      For any $R>0$, there exists a positive constant $K_R$ such that
        \begin{equation}\label{local}
          \left\| f(x_1,y_1)-f(x_2,y_2)\right\| ^2\vee\left\| g(x_1,y_1)-g(x_2,y_2)\right\| ^2\leq K_R(\|x_1-x_2\|^2+\|y_1-y_2\|^2)
        \end{equation}
        for any $x_1 ,y_1,x_2$, $y_2\in \mathbb{R}^d$ with $\left\| x_1\right\| \vee\left\| y_1\right\| \vee\left\| x_2\right\| \vee\left\| y_2\right\| \leq R$.
   \end{assumption}

\begin{assumption}\label{assumption2}
     There exist $\lambda_1, \lambda_2, \lambda_3>0$ such that for any $x_1 ,y_1,x_2$, $y_2\in \mathbb{R}^d$,
        \begin{equation}\label{monotone}
          \left\langle x_1-x_2,f(x_1,y)-f(x_2,y)\right\rangle \leq-\lambda_1\|x_1-x_2\|^2,
        \end{equation}
         \begin{equation}\label{muy}
       \left\| f(x,y_1)-f(x,y_2)\right\| ^2\leq \lambda_2\left\| y_1-y_2\right\| ^2
        \end{equation}
and
       \begin{equation}\label{sigma}
       \left\| g(x_1,y_1)-g(x_2,y_2)\right\| ^2\leq \lambda_3(\left\| x_1-x_2\right\| ^2+\left\|  y_1-y_2\right\| ^2).
      \end{equation}
   \end{assumption}
From Assumption \ref{assumption2}, for any $x,y\in\mathbb{R}^d$, we have
\begin{equation}\label{f-1}
\begin{split}
2\left\langle x,f(x,y)\right\rangle
=&2\left\langle  x-0,f(x,y)-f(0,y)\right\rangle +2\left\langle  x,f(0,y)\right\rangle
\\
\leq&-2\lambda_1\left\| x\right\| ^2+\|x\|^2+\|f(0,y)-f(0,0)+f(0,0)\|^2
\\
\leq&-(2\lambda_1-1)\|x\|^2+2\lambda_2\|y\|^2+2\|f(0,0)\|^2
\end{split}
\end{equation}
and
\begin{equation}\label{g-1}
\begin{split}
\left\| g(x,y)\right\| ^2
\leq&2\left\| g(x,y)-g(0,0)\right\| ^2+2\left\| g(0,0)\right\| ^2
\\
\leq&2\lambda_3\|x\|^2+2\lambda_3\|y\|^2+2\|g(0,0)\|^2.
\end{split}
\end{equation}

  Under Assumptions \ref{assumption1} and \ref{assumption2},  Eq. \eqref{eq}  admits a unique global solution $X(t)$ (see \cite[Theorem 3.1]{zhangling}). To demonstrate the dynamics of $\{X(k)\}_{k\in\mathbb{N}}$, for any $x\in\mathbb{R}^d$ and $B\in\mathcal{B}(\mathbb{R}^d)$, we define 
  $$P(x,B)=\mathbb{P}\{X(1)\in B|X(0)=x\}\quad \mathrm{and} \quad P_k(x,B)=\mathbb{P}\{X(k)\in B|X(0)=x\}.$$
Unless otherwise specified, we write $X^{k,x}(t)$ in lieu of  $X(t)$ to highlight the initial value $X(k)=x$.  Let us first verify that  $\{X(k)\}_{k\in\mathbb{N}}$ is indeed a Markov chain.

\begin{theorem}\label{Markov Property}
Suppose that Assumptions \ref{assumption1} and \ref{assumption2} hold. Then $\{X(k)\}_{k\in\mathbb{N}}$ is a time-homogeneous Markov chain with the transition probability kernel $P(x,B)$.
\end{theorem}
\begin{proof}
We divide this proof into two parts.

\textbf{(i) Time-homogeneity.} For $k,k'\in\mathbb{N}$, if $X(k')=x$, then 
\begin{equation*}\label{time-homogeneous1}
\begin{split}
X^{k',x}(k+k')=&x+\int_{k'}^{k+k'}f(X^{k',x}(s),X^{k',x}([s]))ds+\int_{k'}^{k+k'}g(X^{k',x}(s),X^{k',x}([s]))dB(s)
\\
=&x+\int_0^{k}\!\!f(X^{k',x}(u+k'),X^{k',x}([u]+k'))du+\int_{0}^{k}\!\!g(X^{k',x}(u+k'),X^{k',x}([u]+k'))d\widetilde{B}(u),
\end{split}
\end{equation*}
where $\widetilde{B}(u)=B(u+k')-B(k')$, $u\geq0$. In addition, if $X(0)=x$, then
$$X^{0,x}(k)=x+\!\int_0^{k}\!f(X^{0,x}(u),X^{0,x}([u]))du+\!\int_{0}^{k}\!g(X^{0,x}(u),X^{0,x}([u]))dB(u).$$
Since $\widetilde{B}(u)$ and $B(u)$ have the same distribution, by the weak uniqueness of the solution for Eq. \eqref{eq}, we obtain that
$X^{k',x}(k+k')$ and $X^{0,x}(k)$ are identical in probability law. Hence
$$\mathbb{P}\{X(k+k')\in B|X(k')=x\}=\mathbb{P}\{X(k)\in B|X(0)=x\}$$
for any $B\in\mathcal{B}(\mathbb{R}^d)$ and $x\in\mathbb{R}^d$, which means that $\{X^{0,x}(k)\}_{k\in\mathbb{N}}$ is time-homogeneous.

\textbf{(ii) Markov property.}
Define $\mathcal{G}_{t,s}=\sigma\{B(u)-B(s),~s\leq u\leq t\}\cup\mathcal{N}$, where $s,t>0$ and $\mathcal{N}$ denotes the collection of all $\mathbb{P}$-null sets in $\mathcal{F}$. The property of Brownian motion yields that $\mathcal{F}_s$ is independent of $\mathcal{G}_{t,s}$. For $k\in\mathbb{N}$, let $B^k(t):=B(t)-B(k)$,  $t\geq k$. Then $B^k(t)$ is $\mathcal{F}_t\cap\mathcal{G}_{t,k}$-measurable. Replacing $B(t)$ by $B^k(t)$ in Eq. \eqref{eq}, we get the unique solution $\{X^{k,y}(t)\}_{t\geq k}$, which is adapted to $\{\mathcal{F}_t\cap\mathcal{G}_{t,k}\}_{t\geq k}$. Thus,  $X^{k,y}(k+k')$ is independent of $\mathcal{F}_k$ for any $k,k'\in\mathbb{N}$ and $y\in\mathbb{R}^d$.

For any fixed $k,k'\in\mathbb{N}$, $y\in\mathbb{R}^d$, define $\Psi: \mathbb{R}^d\times \Omega\rightarrow\mathbb{R}^d$, $(y,\omega)\mapsto X^{k,y}(k+k',\omega)$. We claim that $\Psi$ is $\mathcal{B}(\mathbb{R}^d)\otimes\mathcal{G}_{k+k',k}$-measurable. Applying It\^{o}'s formula to $\mathbb{E}\left\| X^{k,z}(t)-X^{k,y}(t)\right\|^2$, $t\geq k$, we obtain
\begin{align*}
\mathbb{E}\left\| X^{k,z}(t)-X^{k,y}(t)\right\|^2
=&\left\| z-y\right\|^2
+\mathbb{E}\int_{k}^{t} \left\| g(X^{k,z}(s),X^{k,z}([s]))-g(X^{k,y}(s),X^{k,y}([s]))\right\| ^2ds\\
&+2\mathbb{E}\int_{k}^{t}\left\langle X^{k,z}(s)-X^{k,y}(s), f(X^{k,z}(s),X^{k,z}([s]))-f(X^{k,y}(s),X^{k,z}([s]))\right\rangle  ds\\
&+2\mathbb{E}\int_{k}^{t}\left\langle X^{k,z}(s)-X^{k,y}(s), f(X^{k,y}(s),X^{k,z}([s]))-f(X^{k,y}(s),X^{k,y}([s]))\right\rangle  ds.
\end{align*}
According to Assumption \ref{assumption2}, the equation above yields
\begin{equation*}
\begin{split}
\mathbb{E}\left\| X^{k,z}(t)-X^{k,y}(t)\right\|^2
\leq&\left\| z-y\right\|^2
-(2\lambda_1-1-\lambda_3)\mathbb{E}\int_{k}^{t}\left\|  X^{k,z}(s)-X^{k,y}(s)\right\|^2ds\\
&+(\lambda_2+\lambda_3)\mathbb{E}\int_{k}^{t}\left\|  X^{k,z}([s])-X^{k,y}([s])\right\|^2ds\\
\leq&\left\| z-y\right\|^2
+\lambda\int_{k}^{t}\sup_{k\leq u\leq s}\mathbb{E}\left\|  X^{k,z}(u)-X^{k,y}(u)\right\|^2ds,
\end{split}
\end{equation*}
where $\lambda:=|2\lambda_1-1-\lambda_3|+\lambda_2+\lambda_3$. By Gronwall's inequality, we have
\begin{equation}\label{bounded in probability}
\mathbb{E}\left\| X^{k,z}(t)-X^{k,y}(t)\right\|^2\leq e^{\lambda (t-k)}\left\| z-y\right\|^2,
\end{equation}
which implies that $\Psi$ is continuous in probability with respect to $y$, i.e., for any $\varepsilon>0$
$$\mathbb{P}\left\lbrace \omega\in\Omega: \left\| X^{k,z}(t,\omega)-X^{k,y}(t,\omega)\right\| >\varepsilon\right\rbrace \rightarrow 0,\qquad \mathrm{as}~z\rightarrow y.$$ Theorem 3.1 in \cite{Gusak2010Theory} implies that there is a modification $\widetilde{\Psi}$ of  $\Psi$ that is $\mathcal{B}(\mathbb{R}^d)\otimes\mathcal{G}_{k+k',k}$-measurable.
Therefore, $\varphi(X^{0,x}(k+k'))$  is $\mathcal{B}(\mathbb{R}^d)\otimes\mathcal{G}_{k+k',k}$-measurable for any $\varphi\in B_b(\mathbb{R}^d)$, where we used the uniqueness of the solution to Eq. \eqref{eq}, i.e.,
\begin{equation*}
X^{0,x}(k+k')=X^{k,X^{0,x}(k)}(k+k')=\widetilde{\Psi}(X^{0,x}(k)),\qquad a.s.
\end{equation*}

Combining the fact that $X^{0,x}(k)$ is $\mathcal{F}_k$-measurable, we have
$$ \mathbb{E}\left[ \varphi(X^{k,X^{0,x}(k)}(k+k'))|\mathcal{F}_k\right] =\mathbb{E}\left[ \varphi(X^{k,y}(k+k'))|\mathcal{F}_k\right]\Big|_{y=X^{0,x}(k)}$$
and
 \begin{equation*}
 \begin{split}
\mathbb{E}\left[ \varphi(X(k+k'))|\mathcal{F}_k\right]
  = \mathbb{E}\left[ \varphi(X^{k,y}(k+k'))\right]\Big|_{y=X(k)}
=\mathbb{E}\left[ \varphi(X(k+k'))|X(k)\right].
\end{split}
\end{equation*}
The proof is completed.
\end{proof}

\begin{theorem}\label{IM of eq}
  Under Assumptions \ref{assumption1} and \ref{assumption2}, if $\lambda_1-\lambda_2-2\lambda_3-1>0$, then the Markov chain $\{X(k)\}_{k\in\mathbb{N}}$ admits a unique invariant measure $\pi$ and there exist $C_1,\nu>0$ independent of $k$ and $x$ such that
  \begin{equation}\label{ exp-con-eq}
\left|\mathbb{E}\varphi(X^{0,x}(k))-\int_{\mathbb{R}^d}\varphi(x)\pi(dx)\right|\leq C_1e^{-\nu k}(1+\|x\|^2),\quad \forall ~\varphi \in C^1_b(\mathbb{R}^d).
\end{equation}
   \end{theorem}
\begin{proof}
{\bf{(i) Existence of invariant measures.}} Let $\alpha:=2\lambda_1-2\lambda_3-1$, $\beta:=2(\lambda_2+\lambda_3)$ and $\gamma:=2(\|f(0,0)\|^2+\|g(0,0)\|^2)$, then $\alpha>0$ and $\frac{\beta}{\alpha}<1$ since $\lambda_1-\lambda_2-2\lambda_3-1=\frac{1}{2}(\alpha-\beta-1)>0$.
Let $\{\tilde{B}(t)\}_{t\geq0}$ be another Brownian motion, independent of $\{B(t)\}_{t\geq0}$, defined on $(\Omega,\mathcal{F},\mathbb{P})$ and define
$$\bar{B}(t)=\begin{cases}B(t),\qquad t\geq0,\\
\tilde{B}(-t),~\quad t<0
\end{cases}$$
with the filtration $\bar{\mathcal{F}}_t:=\sigma{\{\bar{B}(s), s\leq t\}}$, $t\in\mathbb{R}$. For any $k\in\mathbb{N}$ and $x\in\mathbb{R}^d$, we consider the following equation
\begin{equation}\label{eq1}
\begin{cases}
dX(t)=f(X(t),X([t]))dt+g(X(t),X([t]))d\bar{B}(t),\quad t\geq -k,\\
X(-k)=x.
\end{cases}
\end{equation}
It can be verified that \eqref{eq1} admits a unique solution under Assumptions \ref{assumption1} and \ref{assumption2}. In what follows, we show the existence of invariant measure through three steps.

\textbf{Step 1. }  A priori estimate

For any $k\in\mathbb{N}$ and $t>-k$, applying It\^o's formula to $e^{\alpha t}\|X^{-k,x}(t)\|^2$, \eqref{f-1}-\eqref{g-1} lead to
\begin{equation*}\label{ito}
\begin{split}
e^{\alpha t}\mathbb{E}\|X^{-k,x}(t)\|^2=&e^{\alpha [t]}\mathbb{E}\|X^{-k,x}([t])\|^2+\alpha\mathbb{E}\int_{[t]}^te^{\alpha s}\|X^{-k,x}(s)\|^2ds
\\
&+\mathbb{E}\int_{[t]}^te^{\alpha s}\left(2\left\langle X^{-k,x}(s),f(X^{-k,x}(s),X^{-k,x}([s]))\right\rangle
+\|g(X^{-k,x}(s),X^{-k,x}([s]))\|^2\right)ds
\\
\leq&\left(e^{\alpha [t]}+\frac{\beta}{\alpha}\left(e^{\alpha t}-e^{\alpha [t]}\right)\right)\mathbb{E}\|X^{-k,x}([t])\|^2+\frac{\gamma}{\alpha}\left(e^{\alpha t}-e^{\alpha [t]}\right).
\end{split}
\end{equation*}
Hence
\begin{equation}\label{ito}
\begin{split}
\mathbb{E}\|X^{-k,x}(t)\|^2\leq&\left(\frac{\beta}{\alpha}+\left(1-\frac{\beta}{\alpha}\right)e^{-\alpha\{t\}}\right)\mathbb{E}\|X^{-k,x}([t])\|^2+\frac{\gamma}{\alpha}\left(1-e^{-\alpha \{t\}}\right),
\end{split}
\end{equation}
where $\{t\}=t-[t]$. Let $r(\{t\})=\frac{\beta}{\alpha}+\left(1-\frac{\beta}{\alpha}\right)e^{-\alpha\{t\}}$ and $F=\frac{\gamma}{\alpha}$, then $0<r(\{t\})<1$ and
\begin{equation*}
\begin{split}
\mathbb{E}\|X^{-k,x}(t)\|^2\leq& r(\{t\})\mathbb{E}\|X^{-k,x}([t])\|^2+F
\\
\leq&r(\{t\})r(1)\mathbb{E}\|X^{-k,x}([t]-1)\|^2+r(\{t\})F+F
\\
\leq&\cdots
\\
\leq&r(\{t\})e^{([t]+k)\log r(1)}\|x\|^2+\frac{1-r(1)^{[t]+k}}{1-r(1)}r(\{t\})F+F
\\
\leq&\frac{1}{r(1)}e^{(t+k)\log r(1)}\|x\|^2+\frac{1}{1-r(1)}F+F.
\end{split}
\end{equation*}
Since $\log r(1)<0$, there exists a positive constant $C$ independent of $k$ and $t$ such that
\begin{equation}\label{ito2}
\sup_{k\in\mathbb{N}}\mathbb{E}\|X^{-k,x}(t)\|^2\leq \frac{1}{r(1)}\|x\|^2+\frac{1}{1-r(1)}F+F\leq C(1+\|x\|^2).
\end{equation}

\textbf{Step 2.} For any $k_1,k_2\in\mathbb{N}$, $-k_1<-k_2\leq t<\infty$, let $Z(t)=X^{-k_1,x}(t)-X^{-k_2,x}(t)$, then
\begin{equation*}
\begin{split}
Z(t)=&Z(-k_2)
+\int_{-k_2}^t\left(f(X^{-k_1,x}(s),X^{-k_1,x}([s]))-f(X^{-k_2,x}(s),X^{-k_2,x}([s]))\right)ds\\
&+\int_{-k_2}^t\left(g(X^{-k_1,x}(s),X^{-k_1,x}([s]))-g(X^{-k_2,x}(s),X^{-k_2,x}([s]))\right)dB(s).
\end{split}
\end{equation*}
Similarly to Step 1, applying It\^o's formula to $e^{\alpha t}\mathbb{E}\|Z(t)\|^2$,  Assumption \ref{assumption2} leads to
\begin{equation*}
\begin{split}
\mathbb{E}\|Z(t)\|^2\leq\bar{r}(\{t\})\mathbb{E}\left\|Z([t])\right\|^2,
\end{split}
\end{equation*}
where $\bar{r}(\{t\})=\frac{\beta}{2\alpha}+\left(1-\frac{\beta}{2\alpha}\right)e^{-\alpha \{t\}}\in(0,1)$.  Furthermore, we derive
\begin{equation*}
\begin{split}
\mathbb{E}\|Z(t)\|^2\leq&\frac{1}{\bar{r}(1)}e^{(t+k_2)\log\bar{r}(1)}\mathbb{E}\left\|X^{-k_1,x}(-k_2)-x\right\|^2\leq Ce^{(t+k_2)\log\bar{r}(1)}\left(1+\|x\|^2\right).\end{split}
\end{equation*}
In particular,
\begin{equation*}\label{Cauchy}
\mathbb{E}\left\|X^{-k_1,x}(0)-X^{-k_2,x}(0)\right\|^2\leq Ce^{k_2\log\bar{r}(1)}\left(1+\|x\|^2\right),
\end{equation*}
which implies that $\{X^{-k,x}(0)\}_{k\in\mathbb{N}}$ is a Cauchy sequence in $L^2(\Omega,\mathcal{F},\mathbb{P};\mathbb{R}^d)$. Therefore, there exists $\eta^x\in L^2(\Omega,\mathcal{F},\mathbb{P};\mathbb{R}^d)$ such that
\begin{equation}\label{Feller3}
\lim_{k\rightarrow +\infty}\mathbb{E}\left\|X^{-k,x}(0)-\eta^x \right\|^2=0.
\end{equation}
Moreover, following the similar procedure, we obtain
\begin{equation}\label{independence on initial value}
\mathbb{E}\left\|X^{-k,x}(0)-X^{-k,y}(0)\right\|^2\leq\frac{1}{\bar{r}(1)}e^{k\log\bar{r}(1)}\|x-y\|^2.
\end{equation}
Combining \eqref{Feller3} and \eqref{independence on initial value}, we get
\begin{equation*}
\begin{split}
\mathbb{E}\left\|\eta^x-\eta^y\right\|^2=&\mathbb{E}\left\|\eta^x-X^{-k,x}(0)+X^{-k,x}(0)-X^{-k,y}(0)+X^{-k,y}(0)-\eta^y\right\|^2\\
\leq&3\lim_{k\rightarrow\infty}\left(\mathbb{E}\left\|\eta^x-X^{-k,x}(0)\right\|^2+\mathbb{E}\left\|X^{-k,x}(0)-X^{-k,y}(0)\right\|^2+\mathbb{E}\left\|X^{-k,y}(0)-\eta^y\right\|^2\right)\\
=&0.
\end{split}
\end{equation*}
This means that $\eta^x$ is independent of the initial value $x$, which is thus denoted by $\eta$. Furthermore,
\begin{equation}\label{Cauchy}
\mathbb{E}\left\|X^{-k_2,x}(0)-\eta\right\|^2=\lim_{k_1\rightarrow+\infty}\mathbb{E}\left\|X^{-k_2,x}(0)-X^{-k_1,x}(0)\right\|^2\leq Ce^{k_2\log\bar{r}(1)}\left(1+\|x\|^2\right),
\end{equation}
which indicates that $X^{-k,x}(0)$ converges to $\eta$ in distribution as $k\rightarrow \infty$. Since $X^{-k,x}(0)$ and $X^{0,x}(k)$ possess the same distribution, by the definition of convergence in distribution, the transition probabilities $P_k(x,\cdot)=\mathbb{P}\{X(k)\in \cdot |X(0)=x\}$ weakly converges to $\mathbb{P}\circ\eta^{-1}(\cdot)$ as $k\rightarrow \infty$.

\textbf{Step 3.} Denoting by $\pi:=\mathbb{P}\circ\eta^{-1}$ the probability measure induced by $\eta$,  we claim that $\pi$ is an invariant measure. In fact, for any $B\in\mathcal{B}(\mathbb{R}^d)$, the Chapman-Kolmogorov equation  leads to
\begin{equation*}
\begin{split}
\pi(B)=&\int_{\mathbb{R}^d}\mathbf{1}_{B}(y)\pi(dy)=\lim_{k\rightarrow+\infty}\int_{\mathbb{R}^d}\mathbf{1}_{B}(y)P_{k+1}(x,dy)\\
=&\lim_{k\rightarrow+\infty}\int_{\mathbb{R}^d}\int_{\mathbb{R}^d}\mathbf{1}_{B}(y)P(z,dy)P_{k}(x,dz)\\
=&\lim_{k\rightarrow+\infty}\int_{\mathbb{R}^d}P(z,B)P_{k}(x,dz)=\int_{\mathbb{R}^d}P(z,B)\pi(dz).
\end{split}
\end{equation*}

{\bf{(ii) Uniqueness of the invariant measure.}} Since $P_k(x,\cdot)$ weakly converges to $\pi$ as $k\rightarrow \infty$, for any $B\in\mathcal{B}(\mathbb{R}^d)$, $x\in\mathbb{R}^d$ and $k\in\mathbb{N}$, we get
$$\pi(B)=\int_{\mathbb{R}^d}\mathbf{1}_{B}(y)\pi(dy)=\lim_{k\rightarrow+\infty}\int_{\mathbb{R}^d}\mathbf{1}_{B}(y)P_k(x,dy)=\lim_{k\rightarrow+\infty}P_k(x,B).$$
Assume that $\tilde{\pi}\in\mathscr{P}(\mathbb{R}^d)$ is another invariant measure of $\{X(k)\}_{k\in\mathbb{N}}$, then for any $B\in\mathcal{B}(\mathbb{R}^d)$ and $k\in\mathbb{N}$, 
$$\tilde{\pi}(B)=\int_{\mathbb{R}^d}P_k(x,B)\tilde{\pi}(dx).$$
Letting $k\rightarrow \infty$, we obtain
$$\tilde{\pi}(B)=\lim_{k\rightarrow\infty}\int_{\mathbb{R}^d}P_k(x,B)\tilde{\pi}(dx)=\pi(B),$$
which implies that $\pi$ is the unique invariant measure of $\{X(k)\}_{k\in\mathbb{N}}$.

{\bf{(iii) }} For any $\varphi\in C_b^1(\mathbb{R}^d)$, \eqref{Cauchy} and $\pi=\mathbb{P}\circ\eta^{-1}$ lead to
 \begin{equation*}
 \begin{split}
\left|\mathbb{E}\varphi(X^{0,x}(k))-\int_{\mathbb{R}^d}\varphi(x)\pi(dx)\right|
\leq&\mathbb{E}\left|\varphi(X^{0,x}(k))-\varphi(\eta)\right|
\leq\|\varphi\|_1\cdot\mathbb{E}\left\|X^{0,x}(k)-\eta\right\|
\\
\leq&C_1\mathrm{e}^{-\nu k}(1+\|x\|^2),
\end{split}
\end{equation*}
where $C_1=\frac{\|\varphi\|_1}{\sqrt{\bar{r}(1)}}$ and $\nu=-\frac{1}{2}\log\bar{r}(1)$. The proof is completed.
\end{proof}

Besides a priori estimate in Theorem \ref{IM of eq}, we also present the uniform boundedness of $X(t)$ in $p$th ($p\geq1$) moment, which is crucial
to estimating the time-independent weak error of numerical methods.

\begin{lemma}\label{uniform-bounded-2p}
 Let Assumptions \ref{assumption1}-\ref{assumption2} hold and $p\geq 1$. If $\lambda_1-\lambda_2-2\lambda_3-1>4\lambda_3(p-1)$, then there exists a positive constant $C:=C(x,\lambda_1,\lambda_2,\lambda_3,p)>0$ independent of $t$ such that
 \begin{equation}\label{uniformly bounded-3}
 \mathbb{E}\|X(t)\|^{2p}\leq C.
 \end{equation}
  \end{lemma}
  \begin{proof}
 Theorem \ref{IM of eq} (i) implies that the assertion \eqref{uniformly bounded-3} holds for the case $p=1$.  Thus, it suffices to consider the case $p>1$, which is proved by the  induction.

 We assume that  there exists $C>0$ independent of $t$ such that \eqref{uniformly bounded-3} holds for all $p'=1,2,\cdots,p-1$, then we show $\mathbb{E}\|X(t)\|^{2p}\leq C$.
  Applying It\^o's formula to $\mathbb{E}\|X(t)\|^{2p}$,  \eqref{f-1} and \eqref{g-1} lead to
   \begin{equation}\label{uniformly bounded-4}
   \begin{split}
 \mathbb{E}\|X(t)\|^{2p}\leq&\|x\|^{2p}+2p\mathbb{E}\int_0^t\|X(s)\|^{2p-2}\langle X(s),f(X(s),X([s]))\rangle ds\\
 &+p(2p-1)\mathbb{E}\int_0^t\|X(s)\|^{2p-2}\|g(X(s),X([s]))\|^2ds\\
 \leq&\|x\|^{2p}-p(2\lambda_1-2\lambda_3-1-4\lambda_3(p-1))\mathbb{E}\int_0^t\|X(s)\|^{2p}ds\\
 &+2p\left(\lambda_2+\lambda_3(2p-1)\right)\mathbb{E}\int_0^t\|X(s)\|^{2p-2}\|X([s])\|^{2}ds\\
 &+2p\left(\|f(0,0)\|^2+(2p-1)\|g(0,0)\|^2\right)\mathbb{E}\int_0^t\|X(s)\|^{2p-2}ds.
  \end{split}
 \end{equation}
Using Young's inequality and the assumption $\mathbb{E}\|X(t)\|^{2(p-1)}\leq C$, we obtain
 \begin{equation*}\label{uniformly bounded-5}
   \begin{split}
 \mathbb{E}\|X(t)\|^{2p}\leq&\|x\|^{2p}-\alpha_3(p)\int_0^t\mathbb{E}\|X(s)\|^{2p}ds
 +\int_0^t\left(\gamma_3(p)+\beta_3(p)\mathbb{E}\|X([s])\|^{2p}\right)ds,
  \end{split}
 \end{equation*}
 where $\alpha_3(p)=2\lambda_1p-2\lambda_2p+2\lambda_2-p-2\lambda_3(2p-1)^2$,
  $\gamma_3(p)=2p\left(\|f(0,0)\|^2+(2p-1)\|g(0,0)\|^2\right)C$ and $\beta_3(p)=2\lambda_2+2\lambda_3(2p-1)$.  In addition,
 \begin{equation}\label{uniformly bounded-6}
   \begin{split}
 \mathbb{E}\|X(t)\|^{2p}\leq&\|x\|^{2p}-\!\alpha_3(p)\!\int_0^t\mathbb{E}\|X(s)\|^{2p}ds
 +\!\int_0^t\!\left(\gamma_3(p)+\beta_3(p)\sup_{0\leq r\leq s}\mathbb{E}\|X(r)\|^{2p}\right)ds.
  \end{split}
 \end{equation}
According to \cite[Lemma 8.1]{Ito}, we have
 $$ \sup_{0\leq s\leq t}\mathbb{E}\|X(s)\|^{2p}\leq\|x\|^{2p} +\int_0^t e^{-\alpha_3(p)(t-s)}\left(\gamma_3(p)+\beta_3(p)\sup_{0\leq r\leq s}\mathbb{E}\|X(r)\|^{2p}\right)ds.$$
Due to $2\lambda_1-2\lambda_2-4\lambda_3-1>8\lambda_3(p-1)$, it can be verified that $\alpha_3(p)>\beta_3(p)>0$. Thus, \cite[Lemma 8.2]{Ito} leads to
 \begin{equation*}
   \mathbb{E}\|X(t)\|^{2p}\leq\frac{\gamma_3(p)+\alpha_3(p)\|x\|^{2p}}{\alpha_3(p)-\beta_3(p)}=:C.
 \end{equation*}
The proof is completed.
   \end{proof}
\section{Invariant Measures of the Backward Euler Method}
 Let $\delta=\frac{1}{m}$ be the given step-size with integer $m\geq 1$. Grid points $t_n$ are defined as $t_n=n\delta,~n=0,1,\cdots.$ The backward Euler (BE) method for \eqref{eq} is given by	
   \begin{equation*}
      X_{n+1}= X_{n}+ \delta f(X_{n+1},X_{[n\delta]m})+g(X_{n},X_{[n\delta]m})\Delta B_{n},
   \end{equation*}
   where $X_0=x$, $\Delta B_{n}=B(t_{n+1})-B(t_{n})$, $X_n$ is the approximation to $X(t_n)$ and $X_{[n\delta]m}$ is the approximation to $X([t_n])$. Since, for arbitrary $n=0,1,2,\cdots$, there exist $k\in\mathbb{N}$ and $l=0,1,2,...m-1$ such that $n=km+l$, the BE method can be written as
 \begin{equation}\label{Xn}
      X_{km+l+1}= X_{km+l}+ \delta f(X_{km+l+1},X_{km})+g(X_{km+l},X_{km})\Delta B_{km+l}.
   \end{equation}

 Under the condition \eqref{monotone},   the implicit BE method admits a unique solution $\{X_{km+l}: l=0,1,\cdots,m-1,k\in\mathbb{N}\}$ for all step-sizes. 
 Rewrite \eqref{Xn} as
 \begin{equation}\label{Xn1}
      X_{km+l+1}- \delta f(X_{km+l+1},X_{km})= X_{km+l}+g(X_{km+l},X_{km})\Delta B_{km+l}.
   \end{equation}
For any $a\in\mathbb{R}^d$ and $\delta\in(0,1)$, define the mapping $G: \mathbb{R}^d\rightarrow\mathbb{R}^d$, $x\mapsto x-\delta f(x,a)$. Then $G$ admits its inverse function $G^{-1}: \mathbb{R}^d\rightarrow\mathbb{R}^d$. Moreover, the numerical solution $X_{km+l+1}$ satisfies
\begin{equation}\label{Xn2}
      X_{km+l+1}= G^{-1}(X_{km+l}+g(X_{km+l},X_{km})\Delta B_{km+l})
   \end{equation}
for all $k\in\mathbb{N}$ and $l=0,1,2,\cdots,m-1$.

In order to investigate that whether the BE method inherits the Markov property and admits a unique numerical invariant measure, we denote by $Y_k:=X_{km}$ the solution of BE method at $t=k$, $k\in\mathbb{N}$ and define
  $$P^\delta(x,B)=\mathbb{P}\{Y(1)\in B|Y(0)=x\}\quad \mathrm{and}\quad P_k^\delta(x,B)=\mathbb{P}\{Y(k)\in B|Y(0)=x\},$$
where  $x\in\mathbb{R}^d$ and $B\in\mathcal{B}(\mathbb{R}^d)$. Similarly to $X^{0,x}(k)$, we write $Y_k^{0,x}$ in lieu of  $Y_k$ to highlight the initial value $Y_{0}=x$. Now, let us proceed to show the Markov property of $\{Y_k\}_{k\in\mathbb{N}}$.

  \begin{theorem}\label{Markov chain-nm}
Assume that Assumptions \ref{assumption1} and \ref{assumption2} hold. Then  $\{Y_k\}_{k\in\mathbb{N}}$ is a time-homogeneous Markov chain with the transition probability $P^\delta(x,B)$.
   \end{theorem}
 \begin{proof} \textbf{(i) Time-homogeneity.}
For $k\in\mathbb{N}$, if $Y_k=x$, i.e., $X_{km}=x$, then from \eqref{Xn2}, it follows
$$X_{km+1}^{km,x}=G^{-1}(x+g(x,x)\Delta B_{km}).$$
Define the mapping $G_1: \mathbb{R}^d\times\mathbb{R}^m\rightarrow\mathbb{R}^d$, $(y,z)\mapsto G^{-1}(y+g(y,x)z)$, then $X_{km+1}^{km,x}=G_1(x,\Delta B_{km})$.  In addition, if $Y_0=X_0=x$, then
$$X_{1}^{0,x}=G^{-1}(x+g(x,x)\Delta B_{0})=G_1(x,\Delta B_{0}).$$
Since $\Delta B_{km}$ and $\Delta B_{0}$ are identical in probability law, $X_{km+1}^{km,x}$ and $X_1^{0,x}$ possess the same distribution.

Again, by \eqref{Xn2}, we have
\begin{equation*}
\begin{split}
X_{km+2}^{km,x}=&G^{-1}(X_{km+1}^{km,x}+g(X_{km+1}^{km,x},x)\Delta B_{km+1})
\\
=&G^{-1}(G_1(x,\Delta B_{km}))+g(G_1(x,\Delta B_{km}),x)\Delta B_{km+1})
\end{split}
\end{equation*}
and
\begin{equation*}
\begin{split}
X_{2}^{0,x}=&G^{-1}(X_{1}^{0,x}+g(X_{1}^{0,x},x)\Delta B_{1})=G^{-1}(G_1(x,\Delta B_{0}))+g(G_1(x,\Delta B_{0}),x)\Delta B_{1}).
\end{split}
\end{equation*}
Therefore, there exists a function $G_2:\mathbb{R}^d\times\mathbb{R}^m\times\mathbb{R}^m\rightarrow\mathbb{R}^d$ such that
$$X_{km+2}^{km,x}=G_2(x,\Delta B_{km},\Delta B_{km+1})$$
and
$$X_{2}^{0,x}=G_2(x,\Delta B_{0},\Delta B_{1}).$$
Since $(\Delta B_{km}$, $\Delta B_{km+1})$ and $(\Delta B_{0}$, $\Delta B_{1})$ have the same distribution, $X_{km+2}^{km,x}$ and $X_2^{0,x}$  are identical in probability law.

By the same procedure as above,  there exists a function $G_m$ such that
\begin{equation}\label{Y_k+1}
Y_{k+1}^{k,x}=X_{(k+1)m}^{km,x}=G_m(x,\Delta B_{km},\Delta B_{km+1},\cdots,\Delta B_{km+m-1})
\end{equation}
and
$$Y_1^{0,x}=X_{m}^{0,x}=G_m(x,\Delta B_{0},\Delta B_{1},\cdots,\Delta B_{m-1}).$$
Since $(\Delta B_{km}, \Delta B_{km+1}, \cdots, \Delta B_{km+m-1})$ and $(\Delta B_{0},\Delta B_{1},\cdots,\Delta B_{m-1})$ have the same distribution, $Y_{k+1}^{k,x}$ and $Y_1^{0,x}$  are also identical in probability law. Hence
$$\mathbb{P}\{Y_{k+1}\in B|Y_{k}=x\}=\mathbb{P}\{Y_{1}\in B|Y_0=x\}$$
for any $B\in\mathcal{B}(\mathbb{R}^d)$. Further, for any $k,k'\in\mathbb{N}$, we have
$$\mathbb{P}\{Y_{k+k'}\in B|Y_{k'}=x\}=\mathbb{P}\{Y_{k}\in B|Y_0=x\},$$
which implies the time-homogeneous property.

\textbf{(ii) Markov property.}
By the uniqueness of the numerical solution of  \eqref{Xn}, we have
$$Y^{0,x}_{k+1}=X^{0,x}_{(k+1)m}=X^{km,X^{0,x}_{km}}_{(k+1)m}=Y^{k,Y^{0,x}_{k}}_{k+1},\qquad a.s.$$
For $k\in\mathbb{N}$, define $\bar{\mathcal{G}}_{k+1,k}:=\sigma\{\Delta B_{km+l},l=0,1,2,\cdots,m-1\}$. Then $\bar{\mathcal{G}}_{k+1,k}$
is independent of $\mathcal{F}_k$.  From \eqref{Y_k+1}, we know that $Y_{k+1}^{k,y}$ is $\bar{\mathcal{G}}_{k+1,k}$-measurable, and thus is independent of $\mathcal{F}_k$. Using similar techniques as Step 2 in Theorem \ref{Markov Property},  we obtain that $Y^{k,\cdot}_{k+1}$ is $\mathcal{B}(\mathbb{R}^d)\otimes \bar{\mathcal{G}}_{k+1,k}$-measurable. Since $Y^{0,x}_k$ is $\mathcal{F}_k$-measurable,   \begin{equation*}
 \begin{split}
\mathbb{E}\left[ \varphi(Y_{k+1})|\mathcal{F}_k\right]=&\mathbb{E}\left[ \varphi(Y^{k,Y_k}_{k+1})|\mathcal{F}_k\right]
= \mathbb{E}\left[ \varphi(Y^{k,y}_{k+1})\right]\Big|_{y=Y_k}
=\mathbb{E}\left[ \varphi(Y_{k+1})|Y_k\right],
\end{split}
\end{equation*}
 which is the required Markov property. Further, the same procedure yields
\begin{equation*}
 \begin{split}
\mathbb{E}\left[ \varphi(Y_{k+k'})|\mathcal{F}_k\right]=\mathbb{E}\left[ \varphi(Y_{k+k'})|Y_k\right].
\end{split}
\end{equation*}
 The proof is completed.
 \end{proof}

Before we present the existence and uniqueness of the invariant measure of $\{Y_{k}\}_{k\in\mathbb{N}}$, we prepare several lemmas, including the mean square boundedness  and the dependence on initial data of the numerical solution.
\begin{lemma}\label{Inequality 1}
For a nonnegative sequence $Z_{km+l+1}$, if there exist $\alpha>\beta>0$, $\gamma>0$ such that $1-\alpha \delta>0$ and
\begin{equation}\label{zz}
Z_{km+l+1}\leq(1-\alpha \delta)Z_{km+l}+\beta \delta Z_{km}+\gamma\delta
\end{equation}
for $k\in\mathbb{N}$, $l=0,1,\cdots,m-1$, then
\begin{equation}\label{inequality 2}
  Z_{km+l+1}\leq\left(\frac{\beta}{\alpha}+\left(1-\frac{\beta}{\alpha}\right)e^{-\alpha(l+1)\delta}\right)Z_{km}+\frac{\gamma}{\alpha}.
\end{equation}
\end{lemma}
\begin{proof}
From $1-\alpha h>0$ and \eqref{zz}, it follows
\begin{equation}\label{inequality 21}
\begin{split}
Z_{km+l+1}
 \leq&(1-\alpha \delta)^{l+1}Z_{km}+\frac{\beta}{\alpha}\left(1-(1-\alpha\delta)^{l+1}\right)Z_{km}+\frac{\gamma}{\alpha}\left(1-(1-\alpha\delta)^{l+1}\right)
  \\
=&\left(\frac{\beta}{\alpha}+\left(1-\frac{\beta}{\alpha}\right)(1-\alpha \delta)^{l+1}\right)Z_{km}+\frac{\gamma}{\alpha}\left(1-(1-\alpha\delta)^{l+1}\right)
\\
\leq&\left(\frac{\beta}{\alpha}+\left(1-\frac{\beta}{\alpha}\right)e^{\alpha(l+1) \delta}\right)Z_{km}+\frac{\gamma}{\alpha},
 \end{split}
\end{equation}
where in the last step we use  $1-\frac{\beta}{\alpha}>0$ and $1-\alpha\delta>0$.
This completes the proof.
\end{proof}

\begin{lemma}\label{BE-uniform-bounded-2p}
 Let conditions in Lemma \ref{uniform-bounded-2p} hold. Then there exists $C:=C(x,\lambda_1,\lambda_2,\lambda_3,p)>0$ independent of $\delta$, $k$ and $l$ such that
  \begin{equation}\label{BE-uniformly bounded-0}
 \mathbb{E}\|X_{km+l+1}\|^{2p}\leq C
 \end{equation}
for $k\in\mathbb{N}$, $l=0,1,\cdots,m-1$ and $\delta\in(0,\delta_0)$ with $\delta_0$ being sufficiently small.
  \end{lemma}
  \begin{proof}
{\bf{Case 1.}} If $p=1$, then taking the inner product of
 \eqref{Xn} with $X_{km+l+1}$, we get
    \begin{equation}\label{Xn-3}
     \begin{split}
   &\left\|X_{km+l+1}\right\|^2-\left\|X_{km+l}\right\|^2+\left\|X_{km+l+1}-X_{km+l}\right\|^2\\
     =& 2\delta\left\langle X_{km+l+1}, f(X_{km+l+1},X_{km})\right\rangle+2\left\langle X_{km+l+1},g(X_{km+l},X_{km})\Delta B_{km+l}\right\rangle.
      \end{split}
   \end{equation}
From  \eqref{f-1} and \eqref{g-1}, it follows
\begin{equation}\label{boundedness-BE1}
\begin{split}
\mathbb{E}\left\|X_{km+l+1}\right\|^2\leq&\mathbb{E}\left\|X_{km+l}\right\|^2
+\delta\mathbb{E}\left\|g(X_{km+l},X_{km})\right\|^2
+2\delta\mathbb{E} \left\langle X_{km+l+1},f(X_{km+l+1},X_{km})\right\rangle
\\
\leq&(1+2\lambda_3\delta)\mathbb{E}\left\|X_{km+l}\right\|^2-\delta(2\lambda_1-1)\mathbb{E}\left\|X_{km+l+1}\right\|^2
\\
&+2\delta(\lambda_2+\lambda_3)\mathbb{E}\left\|X_{km}\right\|^2+2\delta(\|f(0,0)\|^2+\|g(0,0)\|^2).
\end{split}
\end{equation}
Let $\alpha_1:=\frac{2\lambda_1-2\lambda_3-1}{1+\left(2\lambda_1-1\right)\delta}$, $\beta_1:=\frac{2(\lambda_2+\lambda_3)}{1+\left(2\lambda_1-1\right)\delta}$ and $\gamma_1:=\frac{2(\|f(0,0)\|^2+\|g(0,0)\|^2)}{1+\left(2\lambda_1-1\right)\delta}$. Then
\begin{equation}\label{boundedness-BE2}
\begin{split}
\mathbb{E}\left\|X_{km+l+1}\right\|^2
\leq&(1-\alpha_1\delta)\mathbb{E}\left\|X_{km+l}\right\|^2
+\beta_1\delta\mathbb{E}\left\|X_{km}\right\|^2
+\gamma_1\delta.
\end{split}
\end{equation}
Since $\lambda_1-\lambda_2-2\lambda_3-1>0$, we have $0<\alpha_1\delta<1$ and $\frac{\beta_1}{\alpha_1}<1$ for any $\delta\in(0,1)$. By Lemma \ref{Inequality 1}, \eqref{boundedness-BE2} yields
\begin{equation}\label{boundedness-BE21}
\begin{split}
\mathbb{E}\left\|X_{km+l+1}\right\|^2
\leq&\left(\frac{\beta_1}{\alpha_1}+\left(1-\frac{\beta_1}{\alpha_1}\right)e^{-\alpha_1\delta(l+1)}\right)\mathbb{E}\left\|X_{km}\right\|^2
+\frac{\gamma_1}{\alpha_1}=:r_1(l)\mathbb{E}\left\|X_{km}\right\|^2
+\frac{\gamma_1}{\alpha_1}.
\end{split}
\end{equation}
Here $0<r_1(l)<1$ for $l=0,1,2,\cdots,m-1$. If $l=m-1$, then
$$\mathbb{E}\left\|X_{(k+1)m}\right\|^2\leq r_1(m-1)\mathbb{E}\left\|X_{km}\right\|^2+\frac{\gamma_1}{\alpha_1}.$$
Let $F_1:=\left(\frac{\beta_1}{\alpha_1}+\left(1-\frac{\beta_1}{\alpha_1}\right)e^{-(2\lambda_1-2\lambda_3-1)}\right)^{-1}$ and $F_2:=\left(1-\frac{\beta_1}{\alpha_1}-\left(1-\frac{\beta_1}{\alpha_1}\right)e^{-\frac{2\lambda_1-2\lambda_3-1}{2\lambda_1}}\right)^{-1}$. Then
\begin{equation}\label{boundedness-BE3}
\begin{split}
\mathbb{E}\left\|X_{km+l+1}\right\|^2
\leq&r_1(l)(r_1(m-1))^k\left\|x\right\|^2+\frac{\gamma_1}{\alpha_1}+\frac{\gamma_1}{\alpha_1}r_1(l)\left(1+r_1(m-1)+\cdots+(r(m-1))^k\right)
\\
\leq&\frac{1}{r_1(m-1)}e^{(km+l+1)\delta\log r_1(m-1)}\|x\|^2+\frac{\gamma_1}{\alpha_1}\cdot\frac{2-r_1(m-1)}
{1-r_1(m-1)}
\\
\leq&F_1\|x\|^2+\frac{\gamma_1}{\alpha_1}\left(1+F_2\right)=:C,
\end{split}
\end{equation}
where we use  $\delta\in(0,1)$ and $0<r_1(l)<1$ for $l=0,1,2,\cdots,m-1$.

 {\bf{Case 2.}}  For $p>1$, we show the assertion \eqref{BE-uniformly bounded-0} by  induction. Since $\lambda_1-\lambda_2-2\lambda_3-1>0$,  Case 1 implies that $ \mathbb{E}\|X_{km+l+1}\|^{2p'}\leq C$ holds with $p'=1$. 
   Multiplying \eqref{Xn-3} by $\left\|X_{km+l+1}\right\|^2$ and taking expectation yield $LHS=RHS$, where
     \begin{equation*}
     \begin{split}
   LHS=&\mathbb{E}\left\|X_{km+l+1}\right\|^4-\mathbb{E}\left\|X_{km+l+1}\right\|^2\left\|X_{km+l}\right\|^2+\mathbb{E}\left\|X_{km+l+1}\right\|^2\left\|X_{km+l+1}-X_{km+l}\right\|^2\\
   =&\frac{1}{2}\mathbb{E}\left(\left\|X_{km+l+1}\right\|^4-\left\|X_{km+l}\right\|^4+\left(\left\|X_{km+l+1}\right\|^2-\left\|X_{km+l}\right\|^2\right)^2\right)\\
   &+\mathbb{E}\left\|X_{km+l+1}\right\|^2\left\|X_{km+l+1}-X_{km+l}\right\|^2
      \end{split}
   \end{equation*}
   and
   \begin{align*}
   RHS
     =& 2\delta\mathbb{E}\left\|X_{km+l+1}\right\|^2\left\langle X_{km+l+1}, f(X_{km+l+1},X_{km})\right\rangle\\
    &+2\mathbb{E}\left\|X_{km+l+1}\right\|^2\left\langle X_{km+l+1},g(X_{km+l},X_{km})\Delta B_{km+l}\right\rangle\\
     =& 2\delta\mathbb{E}\left\|X_{km+l+1}\right\|^2\left\langle X_{km+l+1}, f(X_{km+l+1},X_{km})\right\rangle\\
     &+2\mathbb{E}\left\|X_{km+l+1}\right\|^2\left\langle X_{km+l+1}-X_{km+l},g(X_{km+l},X_{km})\Delta B_{km+l}\right\rangle\\
     &+2\mathbb{E}\left\|X_{km+l+1}\right\|^2\left\langle X_{km+l},g(X_{km+l},X_{km})\Delta B_{km+l}\right\rangle\\
     =:&R_1+R_2+R_3.
   \end{align*}
For term $R_1$, by \eqref{f-1} and $2ab\leq a^2+b^2$, we have
 \begin{align*}
R_1
\leq&-\left(2\lambda_1-1\right)\delta\mathbb{E}\left\|X_{km+l+1}\right\|^4+2\lambda_2\delta\mathbb{E}\left\|X_{km+l+1}\right\|^2\left\|X_{km+l}\right\|^2+2\|f(0,0)\|^2\delta\mathbb{E}\left\|X_{km+l+1}\right\|^2\\
\leq&-\left(2\lambda_1-1-\lambda_2\right)\delta\mathbb{E}\left\|X_{km+l+1}\right\|^4+\lambda_2\delta\mathbb{E}\left\|X_{km+l}\right\|^4+2\|f(0,0)\|^2\delta\mathbb{E}\left\|X_{km+l+1}\right\|^2.
   \end{align*}
For terms $R_2$ and $R_3$,   according to \eqref{g-1} and $2ab\leq \epsilon a^2+\frac{1}{\epsilon}b^2$, $\epsilon>0$, we obtain
   \begin{align*}
R_2
\leq&\mathbb{E}\|X_{km+l+1}\|^2\|X_{km+l+1}-X_{km+l}\|^2
+\mathbb{E}\|X_{km+l}\|^2\|g(X_{km+l},X_{km})\Delta B_{km+l}\|^2\\
&+\mathbb{E}\left(\|X_{km+l+1}\|^2-\|X_{km+l}\|^2\right)\|g(X_{km+l},X_{km})\Delta B_{km+l}\|^2\\
\leq&\mathbb{E}\|X_{km+l+1}\|^2\|X_{km+l+1}-X_{km+l}\|^2+\frac{\epsilon_1}{4}\mathbb{E}\left(\|X_{km+l+1}\|^2-\|X_{km+l}\|^2\right)^2\\
&+\frac{1}{\epsilon_1}\mathbb{E}\|g(X_{km+l},X_{km})\Delta B_{km+l}\|^4+\mathbb{E}\|X_{km+l}\|^2\|g(X_{km+l},X_{km})\Delta B_{km+l}\|^2\\
\leq&\mathbb{E}\|X_{km+l+1}\|^2\|X_{km+l+1}-X_{km+l}\|^2+\frac{\epsilon_1}{4}\mathbb{E}\left(\|X_{km+l+1}\|^2-\|X_{km+l}\|^2\right)^2\\
&+\frac{3}{\epsilon_1}\delta^2\mathbb{E}(2\lambda_3\|X_{km+l}\|^2+2\lambda_3\|X_{km}\|^2+2\|g(0,0)\|^2)^2\\
&+2\lambda_3\delta\mathbb{E}\|X_{km+l}\|^4+2\lambda_3\delta\mathbb{E}\|X_{km+l}\|^2\|X_{km}\|^2+2\|g(0,0)\|^2\delta\mathbb{E}\|X_{km+l}\|^2\\
=&\mathbb{E}\|X_{km+l+1}\|^2\|X_{km+l+1}-X_{km+l}\|^2+\frac{\epsilon_1}{4}\mathbb{E}\left(\|X_{km+l+1}\|^2-\|X_{km+l}\|^2\right)^2\\
&+\left(\frac{24\lambda_3^2\delta^2}{\epsilon_1}+3\lambda_3\delta\right)\mathbb{E}\|X_{km+l}\|^4+\left(\frac{24\lambda_3^2h^2}{\epsilon_1}+\lambda_3\delta\right)\mathbb{E}\|X_{km}\|^4+\frac{12\delta^2}{\epsilon_1}\|g(0,0)\|^4\\
&+\frac{24\lambda_3\delta^2}{\epsilon_1}\|g(0,0)\|^2\mathbb{E}(\|X_{km+l}\|^2+\|X_{km}\|^2)+2\|g(0,0)\|^2\delta\mathbb{E}\|X_{km+l}\|^2
\end{align*}
and
\begin{align*}
R_3
=&2\mathbb{E}(\|X_{km+l+1}\|^2-\|X_{km+l}\|^2)\langle X_{km+l},g(X_{km+l},X_{km})\rangle\Delta B_{km+l}\\
\leq&\epsilon_2\mathbb{E}(\|X_{km+l+1}\|^2-\|X_{km+l}\|^2)^2+\frac{1}{\epsilon_2}\mathbb{E}(\langle X_{km+l},g(X_{km+l},X_{km})\rangle\Delta B_{km+l})^2\\
\leq&\epsilon_2\mathbb{E}(\|X_{km+l+1}\|^2-\|X_{km+l}\|^2)^2+\frac{3\lambda_3\delta}{\epsilon_2}\mathbb{E}\|X_{km+l}\|^4\\
&+\frac{\lambda_3\delta}{\epsilon_2}\mathbb{E}\|X_{km}\|^4+\frac{2\|g(0,0)\|^2\delta}{\epsilon_2}\mathbb{E}\|X_{km+l}\|^2.
\end{align*}
Substituting $R_1$, $R_2$ and $R_3$ into $RHS$, we get
\begin{align}
RHS\leq&-(2\lambda_1-1-\lambda_2)\delta\mathbb{E}\|X_{km+l+1}\|^4
+\mathbb{E}\|X_{km+l+1}\|^2\|X_{km+l+1}-X_{km+l}\|^2\notag \\
&+\left(\frac{\epsilon_1}{4}+\epsilon_2\right)\mathbb{E}\left(\|X_{km+l+1}\|^2-\|X_{km+l}\|^2\right)^2+\left(\frac{24\lambda_3^2\delta^2}{\epsilon_1}+3\lambda_3\delta
+\frac{3\lambda_3\delta}{\epsilon_2}\right)\mathbb{E}\|X_{km+l}\|^4\label{R-BE-3}\\
&+\left(\frac{24\lambda_3^2\delta^2}{\epsilon_1}+\lambda_3\delta+\lambda_2\delta+\frac{\lambda_3\delta}{\epsilon_2}\right)\mathbb{E}\|X_{km}\|^4+C\delta\notag ,
\end{align}
where $\mathbb{E}\|X_{km+l+1}\|^2\leq C$, $k\in\mathbb{N}$, $l=0,1,\cdots,m-1$ are used and $C$ is independent of $\delta$, $k$ and $l$.
Let $\epsilon_2+\frac{\epsilon_1}{4}=\frac{1}{2}$, $\alpha_4=\frac{2(2\lambda_1-1-\lambda_2)-\frac{48\lambda_3^2\delta}{\epsilon_1}-6\lambda_3-\frac{6\lambda_3}{\epsilon_2}}{1+2(2\lambda_1-1-\lambda_2)\delta}$ and $\beta_4=\frac{\frac{48\lambda_3^2\delta}{\epsilon_1}+2\lambda_3+2\lambda_2+\frac{2\lambda_3}{\epsilon_2}}{1+2(2\lambda_1-1-\lambda_2)\delta}$. Recall that $LHS=RHS$,
\begin{align*}
(1+2(2\lambda_1-1-\lambda_2)\delta)\mathbb{E}\|X_{km+l+1}\|^4\leq&\left(1+\frac{48\lambda_3^2\delta^2}{\epsilon_1}+6\lambda_3\delta+\frac{6\lambda_3\delta}{\epsilon_2}\right)\mathbb{E}\|X_{km+l}\|^4\\
&+\left(\frac{48\lambda_3^2\delta^2}{\epsilon_1}+2\lambda_3\delta+2\lambda_2\delta+\frac{2\lambda_3\delta}{\epsilon_2}\right)\mathbb{E}\|X_{km}\|^4+C\delta,
\end{align*}
which implies
\begin{equation*}
\begin{split}
\mathbb{E}\|X_{km+l+1}\|^4
\leq&(1-\alpha_4 \delta)\mathbb{E}\|X_{km+l}\|^4+\beta_4 \delta\mathbb{E}\|X_{km}\|^4+C\delta.
\end{split}
\end{equation*}
Up to now, it suffices to show that
$$\alpha_4-\beta_4=\frac{2\left(2\lambda_1-1-2\lambda_2-4\lambda_3\right)-\frac{96\lambda_3^2\delta}{\epsilon_1}-\frac{8\lambda_3}{\epsilon_2}}{1+2(2\lambda_1-1-\lambda_2)\delta}>0.$$
Since $1+2(2\lambda_1-1-\lambda_2)\delta>0$,  we just need
 $$2\lambda_1-1-2\lambda_2-4\lambda_3>\frac{48\lambda_3^2\delta}{\epsilon_1}+\frac{4\lambda_3}{\epsilon_2}.$$
Note that the condition $\lambda_1-1-\lambda_2-2\lambda_3>4(p-1)\lambda_3$ equals to $2\lambda_1-1-2\lambda_2-4\lambda_3>8(p-1)\lambda_3+1$, and $8(p-1)\lambda_3+1>8\lambda_3+1$.  We choose $\epsilon_2=\frac{8\lambda_3}{16\lambda_3+1}$, then
$$\frac{4\lambda_3}{\epsilon_2}=8\lambda_3+\frac{1}{2},$$
 and there exists $\delta_1>0$ such that
 $$\frac{48\lambda_3^2\delta}{\epsilon_1}\leq\frac{1}{2},\qquad \forall ~\delta\in(0,\delta_1),$$
which leads to
$$\frac{48\lambda_3^2\delta}{\epsilon_1}+\frac{4\lambda_3}{\epsilon_2}\leq 8\lambda_3+1<2\lambda_1-1-2\lambda_2-4\lambda_3,$$
i.e., $\alpha_4-\beta_4>0$. Therefore, there exists $C>0$ independent of $\delta$, $k$ and $l$ such that
 $$\mathbb{E}\|X_{km+l+1}\|^{4}\leq C$$
 for all $k\in\mathbb{N}$ and $l=0,1,\cdots,m-1$. This implies that $\mathbb{E}\|X_{km+l+1}\|^{2p'}\leq C$ holds with $p'=2$.

 By repeating the same procedure as the case $p'=2$, there exist $\delta_0>0$ and $C>0$ independent of $\delta$, $k$ and $l$ such that $\mathbb{E}\|X_{km+l+1}\|^{2p'}\leq C$ for $p'=3,4,\cdots, p$, which completes the proof.
  \end{proof}
\begin{corollary}\label{Boundedness of BE method}
Under Assumptions \ref{assumption1} and \ref{assumption2}, if $\lambda_1-\lambda_2-2\lambda_3-1>0$, then, for any $\delta\in(0,1)$, there exists a constant $C_2>0$ independent of $\delta$ and $k$ such that $\{Y^{0,x}_{k}\}_{k\in\mathbb{N}}$ satisfies
\begin{equation}\label{boundedness-BE}
\begin{split}
\sup_{k\in\mathbb{N}}\mathbb{E}\left\|Y^{0,x}_{k+1}\right\|^2\leq C_2(1+\|x\|^2).
\end{split}
\end{equation}
\end{corollary}

Denote $\alpha_2:=\frac{2\lambda_1-\lambda_3-1}{1+(2\lambda_1-1)\delta}$,
$\beta_2:=\frac{\lambda_2+\lambda_3}{1+(2\lambda_1-1)\delta}$ and $\bar{r}_1(l):=\frac{\beta_2}{\alpha_2}+\left(1-\frac{\beta_2}{\alpha_2}\right)e^{-\alpha_2(l+1)\delta}$, $l=0,1,2,\cdots, m-1$. By $\lambda_1-\lambda_2-2\lambda_3-1>0$, we get $0<\alpha_2\delta<1$, $\frac{\beta_2}{\alpha_2}<1$ and thus $0<\bar{r}_1(l)<1$ for $l=0,1,2,\cdots, m-1$ and $\delta\in(0,1)$. In what follows, we show the continuous dependence on initial data of $\{Y_k\}_{k\in\mathbb{N}}$.
\begin{lemma}\label{Dependence on  initial value of BE method}
Let Assumptions \ref{assumption1} and \ref{assumption2} hold. If $\lambda_1-\lambda_2-2\lambda_3-1>0$, then, for any $\delta\in(0,1)$ and any two initial values $x,y\in\mathbb{R}^d$ with $x\neq y$, the solutions generated by the BE method satisfy
\begin{equation}\label{Dependence on  initial value-BE}
\begin{split}
\mathbb{E}\left\|Y^{0,x}_{k+1}-Y^{0,y}_{k+1}\right\|^2
\leq\frac{1}{\bar{r}_1(m-1)}e^{(k+1)\log\bar{r}_1(m-1)}\|x-y\|^2.
\end{split}
\end{equation}
\end{lemma}
\begin{proof}
From the BE method \eqref{Xn1}, for any initial data $x,y\in\mathbb{R}^d$, it follows
\begin{equation*}
\begin{split}
&X^{0,x}_{km+l+1}-X^{0,y}_{km+l+1}-\delta\left(f(X^{0,x}_{km+l+1},X^{0,x}_{km})-f(X^{0,y}_{km+l+1},X^{0,y}_{km})\right)
\\
=&X^{0,x}_{km+l}-X^{0,y}_{km+l}-\left(g(X^{0,x}_{km+l},X^{0,x}_{km})-g(X^{0,y}_{km+l},X^{0,y}_{km})\right)\Delta B_{km+l}.
\end{split}
\end{equation*}
Using Assumption \ref{assumption2}, we obtain
\begin{align*}
\mathbb{E}\left\|X^{0,x}_{km+l+1}-X^{0,y}_{km+l+1}\right\|^2=&\mathbb{E}\left\|X^{0,x}_{km+l}-X^{0,y}_{km+l}\right\|^2-\delta^2\mathbb{E}\left\|f(X^{0,x}_{km+l+1},X^{0,x}_{km})-f(X^{0,y}_{km+l+1},X^{0,y}_{km})\right\|^2
\\
&+2\delta\mathbb{E}\left\langle X^{0,x}_{km+l+1}-X^{0,y}_{km+l+1},f(X^{0,x}_{km+l+1},X^{0,x}_{km})-f(X^{0,y}_{km+l+1},X^{0,y}_{km})\right\rangle
\\
&+\mathbb{E}\left\|\left(g(X^{0,x}_{km+l},X^{0,x}_{km})-g(X^{0,y}_{km+l},X^{0,y}_{km})\right)\Delta B_{km+l}\right\|^2
\\
\leq&\mathbb{E}\left\|X^{0,x}_{km+l}-X^{0,y}_{km+l}\right\|^2
+\delta\mathbb{E}\left\|g(X^{0,x}_{km+l},X^{0,x}_{km})-g(X^{0,y}_{km+l},X^{0,y}_{km})\right\|^2
\\
&+2\delta\mathbb{E}\left\langle X^{0,x}_{km+l+1}-X^{0,y}_{km+l+1},f(X^{0,x}_{km+l+1},X^{0,x}_{km})-f(X^{0,y}_{km+l+1},X^{0,y}_{km})\right\rangle
\\
\leq&(1+\lambda_3\delta)\mathbb{E}\left\|X^{0,x}_{km+l}-X^{0,y}_{km+l}\right\|^2
+(\lambda_2+\lambda_3)\delta\mathbb{E}\left\|X^{0,x}_{km}-X^{0,y}_{km}\right\|^2
\\
&-(2\lambda_1-1)\delta\mathbb{E}\left\|X^{0,x}_{km+l+1}-X^{0,y}_{km+l+1}\right\|^2.
\end{align*}
Therefore
\begin{equation}\label{Dependence on  initial value-BE1}
\begin{split}
\mathbb{E}\left\|X^{0,x}_{km+l+1}-X^{0,y}_{km+l+1}\right\|^2
\leq& (1-\alpha_2\delta)\mathbb{E}\left\|X^{0,x}_{km+l}-X^{0,y}_{km+l}\right\|^2
+\beta_2\delta\mathbb{E}\left\|X^{0,x}_{km}-X^{0,y}_{km}\right\|^2
\\
\leq&\left(\frac{\beta_2}{\alpha_2}+\left(1-\frac{\beta_2}{\alpha_2}\right)(1-\alpha_2\delta)^{l+1}\right)\mathbb{E}\left\|X^{0,x}_{km}-X^{0,y}_{km}\right\|^2
\\
\leq&\bar{r}_1(l)\mathbb{E}\left\|X^{0,x}_{km}-X^{0,y}_{km}\right\|^2.
\end{split}
\end{equation}
If $l=m-1$, then
$$\mathbb{E}\left\|X^{0,x}_{(k+1)m}-X^{0,y}_{(k+1)m}\right\|^2
\leq \bar{r}_1(m-1)\mathbb{E}\left\|X^{0,x}_{km}-X^{0,y}_{km}\right\|^2.$$
Therefore, \eqref{Dependence on  initial value-BE1} yields
\begin{equation}\label{Dependence on  initial value-BE2}
\begin{split}
\mathbb{E}\left\|X^{0,x}_{km+l+1}-X^{0,y}_{km+l+1}\right\|^2
\leq& \bar{r}_1(l)\mathbb{E}\left\|X^{0,x}_{km}-X^{0,y}_{km}\right\|^2
\\
\leq&\frac{1}{\bar{r}_1(m-1)}e^{(km+l+1)\delta\log\bar{r}_1(m-1)}\|x-y\|^2
\end{split}
\end{equation}
and
\begin{equation}\label{Dependence on  initial value-BE3}
\begin{split}
\mathbb{E}\left\|Y^{0,x}_{k+1}-Y^{0,y}_{k+1}\right\|^2
=\mathbb{E}\left\|X^{0,x}_{(k+1)m}-X^{0,y}_{(k+1)m}\right\|^2
\leq\frac{1}{\bar{r}_1(m-1)}e^{(k+1)\log\bar{r}_1(m-1)}\|x-y\|^2.
\end{split}
\end{equation}
The proof is completed.
\end{proof}

Now, we are in a position to show the existence and uniqueness of the BE method's invariant measure.
\begin{theorem}\label{existence of BE}
  Under Assumptions \ref{assumption1} and \ref{assumption2}, if $\lambda_1-\lambda_2-2\lambda_3-1>0$, then, for any $\delta>0$, the Markov chain $\{Y^{0,x}_k\}_{k\in\mathbb{N}}$ admits a unique invariant measure $\pi^\delta$ and there exist $C_3,\bar{\nu}>0$ independent of $x$, $\delta$ and $k$ such that
  \begin{equation}\label{ exp-con-BE}
\left|\mathbb{E}\varphi(Y^{0,x}_k)-\int_{\mathbb{R}^d}\varphi(x)\pi^\delta(dx)\right|\leq C_3(1+\|x\|)e^{-\bar{\nu} k},\quad \forall ~\varphi \in C^1_b(\mathbb{R}^d).
\end{equation}
   \end{theorem}
   \begin{proof}
According to the Chebyshev inequality and \eqref{boundedness-BE}, it follows
    that the transition probability measure $\{P^\delta_k(x,\cdot)\}_{k\in\mathbb{N}}$ is tight. Then by the Prokhorov theorem \cite{Prato}, $\{P^\delta_k(x,\cdot)\}_{k\in\mathbb{N}}$  is weakly relatively compact, i.e., there exists $\pi^\delta\in\mathscr{P}(\mathbb{R}^d)$ such that the subsequence $\{P^\delta_{k_i}(x,\cdot)\}_{k_i\in\mathbb{N}}$ is weakly convergent to $\pi^\delta$ as $k_i\rightarrow \infty$. Moreover, for any $k\in\mathbb{N}$ and $\varphi\in B_b(\mathbb{R}^d)$, the Chapman-Kolmogorov equation leads to
     \begin{align*}
\int_{\mathbb{R}^d}\varphi(z)\pi^\delta(dz)
=&\lim_{k_i\rightarrow \infty}\int_{\mathbb{R}^d}\varphi(z)P^\delta_{k+k_i}(x,dz)=\lim_{k_i\rightarrow \infty}\int_{\mathbb{R}^d}\int_{\mathbb{R}^d}\varphi(z)P^\delta_{k}(y,dz)P^{\delta}_{k_i}(x,dy)
\\
=&\lim_{k_i\rightarrow \infty}\int_{\mathbb{R}^d}P^{\delta}_k\varphi(y)P^\delta_{k_i}(x,dy)
=\int_{\mathbb{R}^d}P^{\delta}_k\varphi(y)\pi^\delta(dy),
\end{align*}
where $P^{\delta}_k\varphi(y)=\mathbb{E}\varphi(Y_k^{0,y})$. This means that the Markov chain $\{Y^{0,x}_k\}_{k\in\mathbb{N}}$ admits an invariant measure $\pi^\delta\in\mathscr{P}(\mathbb{R}^d)$. Moreover, for any $x\in\mathbb{R}^d$ and $\varphi\in B_b(\mathbb{R}^d)$, we have
$$\lim_{k_i\rightarrow\infty}P^\delta_{k_i}\varphi(x)=\lim_{k_i\rightarrow\infty}\int_{\mathbb{R}^d}\varphi(z)P^\delta_{k_i}(x,dz)=\int_{\mathbb{R}^d}\varphi(z)\pi^\delta(dz).$$

In addition, if $\tilde{\pi}^\delta\in\mathscr{P}(\mathbb{R}^d)$ is another invariant measure of $\{Y^{0,x}_k\}_{k\in\mathbb{N}}$, then
$$\int_{\mathbb{R}^d}\varphi(z)\tilde{\pi}^\delta(dz)=\int_{\mathbb{R}^d}P^\delta_{k_i}\varphi(z)\tilde{\pi}^\delta(dz),\qquad k_i\in\mathbb{N},~\varphi\in B_b(\mathbb{R}^d).$$
Taking $k_i\rightarrow \infty$, we get
$$\int_{\mathbb{R}^d}\varphi(z)\tilde{\pi}^\delta(dz)=\lim_{k_i\rightarrow \infty}\int_{\mathbb{R}^d}P^\delta_{k_i}\varphi(z)\tilde{\pi}^\delta(dz)=\int_{\mathbb{R}^d}\varphi(z)\pi^\delta(dz),\qquad \varphi\in B_b(\mathbb{R}^d).$$
This means that $\pi^{\delta}$ is the unique invariant measure for $\{Y^{0,x}_k\}_{k\in\mathbb{N}}$.

Again using the Chapman-Kolmogorov equation, for any $\varphi\in C_b^1(\mathbb{R}^d)$, we have
 \begin{align}\label{ exp-con-BE1}
\left|\mathbb{E}\varphi(Y^{0,x}_k)-\int_{\mathbb{R}^d}\varphi(x)\pi^\delta(dx)\right|
=&\left|\lim_{k_i\rightarrow \infty}P^\delta_k\varphi(x)-\int_{\mathbb{R}^d}\varphi(z)P_{k+k_{i}}^\delta(x,dz)\right|\nonumber
\\
=&\left|\lim_{k_i\rightarrow \infty}\int_{\mathbb{R}^d}P^\delta_k\varphi(x)P_{k_{i}}^\delta(x,dy)-\int_{\mathbb{R}^d}P^\delta_k\varphi(y)P_{k_{i}}^\delta(x,dy)\right|
\\
    \leq&\lim_{k_i\rightarrow \infty}\int_{\mathbb{R}^d}\left|P^\delta_k\varphi(x)-P^\delta_k\varphi(y)\right|P_{k_{i}}^\delta(x,dy)\nonumber
    \\
    \leq&\|\varphi\|_1\cdot\lim_{k_i\rightarrow \infty}\int_{\mathbb{R}^d}\mathbb{E}\left\|Y^{0,x}_{k}-Y^{0,y}_{k}\right\|P_{k_{i}}^\delta(x,dy).\nonumber
\end{align}
From Lemma \ref{Dependence on  initial value of BE method}, it follows
\begin{equation}\label{ exp-con-BE2}
 \begin{split}
\left|\mathbb{E}\varphi(Y^{0,x}_k)-\int_{\mathbb{R}^d}\varphi(x)\pi^\delta(dx)\right|
 \leq&\|\varphi\|_1\cdot\lim_{k_i\rightarrow \infty}\int_{\mathbb{R}^d}\left(\mathbb{E}\left\|Y^{0,x}_{k}-Y^{0,y}_{k}\right\|^2\right)^\frac{1}{2}P_{k_{i}}^\delta(x,dy)
    \\
     \leq&\frac{\|\varphi\|_1}{\sqrt{\bar{r}_1(m-1)}}e^{k/2\log\bar{r}_1(m-1)}\int_{\mathbb{R}^d}\left\|x-y\right\|P_{k_{i}}^\delta(x,dy)
     \\
    =&\frac{\|\varphi\|_1}{\sqrt{\bar{r}_1(m-1)}}e^{k/2\log\bar{r}_1(m-1)}\mathbb{E}\|x-Y_{k_i}^{0,x}\|.
    \end{split}
\end{equation}
By Lemma \ref{Boundedness of BE method}, \eqref{ exp-con-BE2} yields
$$\left|\mathbb{E}\varphi(Y^{0,x}_k)-\int_{\mathbb{R}^d}\varphi(x)\pi^\delta(dx)\right|
\\
    \leq C_3(1+\|x\|)e^{-\bar{\nu} k},$$
 where $C_3=\frac{2\sqrt{C_2}\|\varphi\|_1}{\sqrt{\bar{r}_1(m-1)}}$ and $\bar{\nu}=-\frac{1}{2}\log\bar{r}_1(m-1)$. We complete the proof.
   \end{proof}

  \section{Approximation of the Invariant Measures}
In this section, we aim to estimate the error between invariant measures $\pi$ and $\pi^{\delta}$, i.e.
\begin{equation*}
\begin{split}
\left|\int_{\mathbb{R}^d}\phi(x)\pi(dx)-\int_{\mathbb{R}^d}\phi(x)\pi^{\delta}(dx)\right|.
\end{split}
\end{equation*}
Unless otherwise specified, we assume that $B(t)$ is a 1-dimensional Brownian motion throughout this section.
According to the uniqueness of the invariant measures $\pi$ and $\pi^\delta$, we know that both the Markov chains $\{X(k)\}_{k\in\mathbb{N}}$ and $\{Y_k\}_{k\in\mathbb{N}}$ are ergodic, that is
\begin{equation*}
\begin{split}
\lim_{K\rightarrow \infty}\frac{1}{K}\sum_{k=0}^{K-1}\mathbb{E}\phi(X(k))=\int_{\mathbb{R}^d}\phi(x)\pi(dx)
\end{split}
\end{equation*}
and
\begin{equation*}
\begin{split}
\lim_{K\rightarrow \infty}\frac{1}{K}\sum_{k=0}^{K-1}\mathbb{E}\phi(Y_k)=\int_{\mathbb{R}^d}\phi(x)\pi^{\delta}(dx).
\end{split}
\end{equation*}
Therefore,
\begin{equation*}
\begin{split}
\left|\int_{\mathbb{R}^d}\phi(x)\pi(dx)-\int_{\mathbb{R}^d}\phi(x)\pi^{\delta}(dx)\right|
=&\left|\lim_{K\rightarrow \infty}\frac{1}{K}\sum_{k=0}^{K-1}\mathbb{E}\phi(X(k))-\sum_{k=0}^{K-1}\mathbb{E}\phi(Y_k)\right|\\
\leq&\lim_{K\rightarrow \infty}\frac{1}{K}\sum_{k=0}^{K-1}\left|\mathbb{E}\phi(X(k))-\mathbb{E}\phi(Y_k)\right|.
\end{split}
\end{equation*}
From the inequality above, it is observed that the error between $\pi$ and $\pi^\delta$ can be estimated by the weak error of numerical methods. So, we contribute on the time-independent weak convergence analysis of the BE method and then give the approximation between $\pi$ and $\pi^\delta$.

 \subsection{A Priori Estimates}

Suppose that the Fr\'echet  partial derivatives of $f$ and $g$ exist, for any $\xi\in\mathbb{R}^d$, then the definition of Fr\'echet derivatives and Assumption \ref{assumption2} lead to
\begin{equation}\label{derivative-f-x1}
\xi^T\frac{\partial f}{\partial x}(x,y)\xi\leq -\lambda_1\|\xi\|^2,\qquad
\left\|\frac{\partial f}{\partial y}(x,y)\xi\right\|^2\leq \lambda_2\|\xi\|^2
\end{equation}
and
\begin{equation}\label{derivative-g-x1}
\left\|\frac{\partial g}{\partial x}(x,y)\xi\right\|^2
\leq \lambda_3\|\xi\|^2,\qquad
\left\|\frac{\partial g}{\partial y}(x,y)\xi\right\|^2\leq \lambda_3\|\xi\|^2.
\end{equation}

In the following, we will use $\mathcal{D}$ and $D$ to denote the Malliavin differentiation operator and the Fr\'echet differentiation operator, respectively. Under the estimates for the partial derivatives of $f$ and $g$, we derive the uniform estimation of the Fr\'echet derivative of $X^{i,\eta}(t)$ as follows.
\begin{lemma}\label{DX-uniform bounded}
Assume that the Fr\'echet derivatives of $f$ and $g$ exist and the conditions in Lemma \ref{uniform-bounded-2p} are satisfied. Then there exist two positive constants $C$ and $\nu_1$ independent of $t$ such that
\begin{equation}
\mathbb{E}\left\|DX^{i,\eta}(t)\xi\right\|^{2p}\leq Ce^{-\nu_1(t-i)}\|\xi\|^{2p},\quad t\geq i
\end{equation}
for  $\xi\in\mathbb{R}^d$ and $\eta\in L^{2p}(\Omega,\mathbb{R}^d;\mathcal{F}_{i})$, $i\in\mathbb{N}$.
 \end{lemma}
 \begin{proof}
For any $\xi\in\mathbb{R}^d$, we have
 \begin{equation*}
\begin{split}
DX^{i,\eta}(t)\xi=&DX^{i,\eta}([t])\xi+\int_{[t]}^t\frac{\partial f}{\partial x}(X^{i,\eta}(s),X^{i,\eta}([t]))DX^{i,\eta}(s)\xi ds\\
&+\int_{[t]}^t\frac{\partial f}{\partial y}(X^{i,\eta}(s),X^{i,\eta}([t]))DX^{i,\eta}([t])\xi ds\\
&+\int_{[t]}^t\frac{\partial g}{\partial x}(X^{i,\eta}(s),X^{i,\eta}([t]))DX^{i,\eta}(s)\xi dB(s)\\
&+\int_{[t]}^t\frac{\partial g}{\partial y}(X^{i,\eta}(s),X^{i,\eta}([t]))DX^{i,\eta}([t])\xi dB(s).
\end{split}
\end{equation*}
Denote $H(t):=DX^{i,\eta}(t)\xi$. For any $\alpha>0$, applying It\^o's formula  to $e^{2\alpha pt}\left\|H(t)\right\|^{2p}$, we obtain \begin{align*}
e^{2\alpha pt}\mathbb{E}\left\|H(t)\right\|^{2p}
\leq&e^{2\alpha p[t]}\mathbb{E}\left\|H([t])\right\|^{2p}
+2\alpha p\mathbb{E} \int_{[t]}^te^{2\alpha ps}\left\|H(s)\right\|^{2p} ds\\
&+2p\mathbb{E} \int_{[t]}^te^{2\alpha ps}\left\|H(s)\right\|^{2(p-1)}\left\langle H(s), \frac{\partial f}{\partial x}(X^{i,\eta}(s),X^{i,\eta}([t]))H(s)\right\rangle ds\\
&+2p\mathbb{E} \int_{[t]}^te^{2\alpha ps}\left\|H(s)\right\|^{2(p-1)}\left\langle H(s), \frac{\partial f}{\partial y}(X^{i,\eta}(s),X^{i,\eta}([t]))H([t])\right\rangle ds\\
&+2p(2p-1)\mathbb{E} \int_{[t]}^te^{2\alpha ps}\left\|H(s)\right\|^{2(p-1)}\left\|\frac{\partial g}{\partial x}(X^{i,\eta}(s),X^{i,\eta}([t]))H(s)\right\|^2 ds\\
&+2p(2p-1)\mathbb{E} \int_{[t]}^te^{2\alpha ps}\left\|H(s)\right\|^{2(p-1)}\left\|\frac{\partial g}{\partial y}(X^{i,\eta}(s),X^{i,\eta}([t]))H([t])\right\|^2 ds.
\end{align*}
Taking $2\alpha_7 p =2\lambda_1 p -2\lambda_3(2p-1)^2-p-\lambda_2(p-1)$ and $\beta_7=\lambda_2+2\lambda_3(2p-1)$, then $2\alpha_7 p>\beta_7$ since $\lambda_1-\lambda_2-2\lambda_3-1>4\lambda_3(p-1)$. From \eqref{derivative-f-x1}, \eqref{derivative-g-x1} and Young's inequality, it follows that
 \begin{equation*}
\begin{split}
e^{2\alpha_7 pt}\mathbb{E}\left\|H(t)\right\|^{2p}\leq&e^{2\alpha_7 p[t]}\mathbb{E}\left\|H([t])\right\|^{2p}
+\left(2\alpha_7 -2\lambda_1+2\lambda_3(2p-1)+1\right)p\mathbb{E} \int_{[t]}^te^{2\alpha_7 ps}\left\|H(s)\right\|^{2p} ds\\
&+\left(\lambda_2p+2\lambda_3p(2p-1)\right)\mathbb{E} \int_{[t]}^te^{2\alpha_7 ps}\left\|H(s)\right\|^{2(p-1)}\left\|H([t])\right\|^2ds\\
\leq&e^{2\alpha_7 p[t]}\mathbb{E}\left\|H([t])\right\|^{2p}
+\beta_7\mathbb{E} \int_{[t]}^te^{2\alpha_7 ps}\left\|H([t])\right\|^{2p}ds\\
=&\left(e^{2\alpha_7 p[t]}+\frac{\beta_7}{2\alpha_7 p}\left(e^{2\alpha_7 pt}-e^{2\alpha_7 p[t]}\right)\right)\mathbb{E}\left\|H([t])\right\|^{2p}.
\end{split}
\end{equation*}
Hence
 \begin{equation*}
\begin{split}
\mathbb{E}\left\|DX^{i,\eta}(t)\xi\right\|^{2p}\leq&r_2(\{t\})\mathbb{E}\left\|DX^{i,\eta}([t])\xi\right\|^{2p},
\end{split}
\end{equation*}
where $r_2(\{t\})=\frac{\beta_7}{2\alpha_7 p}+\left(1-\frac{\beta_7}{2\alpha_7 p}\right)e^{-2\alpha_7 p\{t\}}$ and $0<r_2(\{t\})<1$. If $t=k$, then
 \begin{equation*}
\begin{split}
\mathbb{E}\left\|DX^{i,\eta}(k)\xi\right\|^{2p}=&\lim_{t\rightarrow k^-}\mathbb{E}\left\|DX^{i,\eta}(t)\xi\right\|^{2p}\leq r_2(1)\mathbb{E}\left\|DX^{i,\eta}(k-1)\xi\right\|^{2p}.
\end{split}
\end{equation*}
Therefore
\begin{equation*}
\begin{split}
\mathbb{E}\left\|DX^{i,\eta}(t)\xi\right\|^{2p}\leq&r_2(\{t\})r_2(1)^{[t]-i}\mathbb{E}\left\|\xi\right\|^{2p}\leq Ce^{-\nu_1(t-i)}\mathbb{E}\left\|\xi\right\|^{2p},
\end{split}
\end{equation*}
where $C=\frac{1}{r_2(1)}$ and $\nu_1=-\log r_2(1)$. The proof is completed.
 \end{proof}

 Next, we show the uniform estimate of the Malliavin derivative of $X^{i,\eta}(t)$.
\begin{lemma}\label{BE-Malliavin-uniform-bounded-2p}
 Let the conditions in Lemma \ref{uniform-bounded-2p} hold. Then there exists $C>0$ independent of $\delta$, $k$ and $l$ such that
  \begin{equation}\label{BE-uniformly bounded-1-0}
 \mathbb{E}\|\mathcal{D}_uX_{km+l+1}\|^{2p}\leq C
 \end{equation}
for all  $k\in\mathbb{N}$, $l=0,1,\cdots,m-1$ and $\delta\in(0,\tilde{\delta}_0)$ with  $\tilde{\delta}_0$ being sufficiently small.
  \end{lemma}
\begin{proof}
From \eqref{Xn}, the Malliavin derivative of $X_{km+l+1}$ is
\begin{equation}\label{MC}
\begin{split}
\mathcal{D}_uX_{km+l+1}=&\mathcal{D}_uX_{km+l}\mathbf{1}_{\{u<t_{km+l}\}}+g(X_{km+l},X_{km})\mathbf{1}_{\{t_{km+l}\leq u<t_{km+l+1}\}}\\
&+\delta\frac{\partial f}{\partial x}(X_{km+l+1},X_{km})\mathcal{D}_uX_{km+l+1}+\delta\frac{\partial f}{\partial y}(X_{km+l+1},X_{km})\mathcal{D}_uX_{km}\\
&+\frac{\partial g}{\partial x}(X_{km+l},X_{km})\mathcal{D}_uX_{km+l}\Delta B_{km+1}+\frac{\partial g}{\partial y}(X_{km+l},X_{km})\mathcal{D}_uX_{km}\Delta B_{km+1}\\
\end{split}
\end{equation}
for $0\leq u<t_{km+l+1}$. If $u\geq t_{km+l+1}$, then  $\mathcal{D}_uX_{km+l+1}=0$ and then we only consider the case of $0\leq u<t_{km+l+1}$ in the following.

\textbf{Case 1.} If $t_{km+l}\leq u<t_{km+l+1}$, then \eqref{MC} becomes
\begin{equation}\label{MC-1}
\begin{split}
\mathcal{D}_uX_{km+l+1}=&g(X_{km+l},X_{km})+\delta\frac{\partial f}{\partial x}(X_{km+l+1},X_{km})\mathcal{D}_uX_{km+l+1}.
\end{split}
\end{equation}
Multiplying \eqref{MC-1} by $\mathcal{D}_uX_{km+l+1}$, \eqref{derivative-f-x1} leads to
\begin{equation}\label{MC-2}
\begin{split}
\left\|\mathcal{D}_uX_{km+l+1}\right\|^2=&\left\langle \mathcal{D}_uX_{km+l+1},g(X_{km+l},X_{km})\right\rangle\\
&+\delta\left\langle \mathcal{D}_uX_{km+l+1},\frac{\partial f}{\partial x}(X_{km+l+1},X_{km})\mathcal{D}_uX_{km+l+1}\right\rangle\\
\leq&\left(\frac{1}{2}-\lambda_1\delta\right)\left\|\mathcal{D}_uX_{km+l+1}\right\|^2+\frac{1}{2}\mathbb{E}\left\|g(X_{km+l},X_{km})\right\|^2.
\end{split}
\end{equation}
Then multiplying \eqref{MC-2} by $\left\|\mathcal{D}_uX_{km+l+1}\right\|^{2(p-1)}$ and taking expectation, Young's inequality leads to
\begin{equation*}
\begin{split}
\mathbb{E}\left\|\mathcal{D}_uX_{km+l+1}\right\|^{2p}
\leq&\left(\frac{1}{2}-\lambda_1\delta\right)\mathbb{E}\left\|\mathcal{D}_uX_{km+l+1}\right\|^{2p}+\frac{1}{2}\mathbb{E}\left\|\mathcal{D}_uX_{km+l+1}\right\|^{2(p-1)}\left\|g(X_{km+l},X_{km})\right\|^2\\
\leq&\left(\frac{3}{4}-\lambda_1\delta\right)\mathbb{E}\left\|\mathcal{D}_uX_{km+l+1}\right\|^{2p}+\frac{1}{2p}\left(\frac{2(p-1)}{p}\right)^{p-1}\mathbb{E}\left\|g(X_{km+l},X_{km})\right\|^{2p},
\end{split}
\end{equation*}
which implies
\begin{equation*}
\begin{split}
\mathbb{E}\left\|\mathcal{D}_uX_{km+l+1}\right\|^{2p}\leq\frac{2}{\left(1+4\lambda_1\delta\right)p}\left(\frac{2(p-1)}{p}\right)^{p-1}\mathbb{E}\left\|g(X_{km+l},X_{km})\right\|^{2p}.
\end{split}
\end{equation*}
Since $\mathbb{E}\|X_{km+l}\|^{2p}\leq C$ and $g$ satisfies the global Lipschitz condition, we prove \eqref{BE-uniformly bounded-1-0} when $u\in[t_{km+l},t_{km+l+1})$.

\textbf{Case 2.} If $t_{km}\leq u<t_{km+l}$, then \eqref{MC} becomes
\begin{equation}\label{MC-4}
\begin{split}
\mathcal{D}_uX_{km+l+1}-\mathcal{D}_uX_{km+l}
=&\delta\frac{\partial f}{\partial x}(X_{km+l+1},X_{km})\mathcal{D}_uX_{km+l+1}+\frac{\partial g}{\partial x}(X_{km+l},X_{km})\mathcal{D}_uX_{km+l}\Delta B_{km+1}.
\end{split}
\end{equation}
We prove the assertion \eqref{BE-uniformly bounded-1-0} by induction. Let us first show that $\mathbb{E}\left\|\mathcal{D}_uX_{km+l+1}\right\|^{2p'}\leq C$ with $p'=1$.
Multiplying \eqref{MC-4} by $\mathcal{D}_uX_{km+l+1}$ leads to
\begin{equation}\label{MC-5}
\begin{split}
&\frac{1}{2}\left(\mathbb{E}\left\|\mathcal{D}_uX_{km+l+1}\right\|^2-\mathbb{E}\left\|\mathcal{D}_uX_{km+l}\right\|^2+\mathbb{E}\left\|\mathcal{D}_uX_{km+l+1}-\mathcal{D}_uX_{km+l}\right\|^2\right)\\
=&\delta\mathbb{E}\left\langle\mathcal{D}_uX_{km+l+1},\frac{\partial f}{\partial x}(X_{km+l+1},X_{km})\mathcal{D}_uX_{km+l+1}\right\rangle\\
&+\mathbb{E}\left\langle\mathcal{D}_uX_{km+l+1}-\mathcal{D}_uX_{km+l},\frac{\partial g}{\partial x}(X_{km+l},X_{km})\mathcal{D}_uX_{km+l}\Delta B_{km+1}\right\rangle\\
\leq&-\lambda_1\delta\mathbb{E}\left\|\mathcal{D}_uX_{km+l+1}\right\|^2+\frac{1}{2}\mathbb{E}\left\|\mathcal{D}_uX_{km+l+1}-\mathcal{D}_uX_{km+l}\right\|^2+\frac{1}{2}\lambda_3\delta\mathbb{E}\left\|\mathcal{D}_uX_{km+l}\right\|^2,
\end{split}
\end{equation}
which implies
\begin{equation*}
\begin{split}
(1+2\lambda_1\delta)\mathbb{E}\left\|\mathcal{D}_uX_{km+l+1}\right\|^2
\leq(1+\lambda_3\delta)\mathbb{E}\left\|\mathcal{D}_uX_{km+l}\right\|^2.
\end{split}
\end{equation*}
Since $\lambda_1-1-\lambda_2-2\lambda_3>0$, and for any $u$, there exist $n\in\mathbb{N}$ and $w\in\{0,1,\cdots,m-1\}$ such that $u\in[t_{nm+w},t_{nm+w+1})$, combining Case 1, we obtain
\begin{equation*}
\begin{split}
\mathbb{E}\left\|\mathcal{D}_uX_{km+l+1}\right\|^2\leq\mathbb{E}\left\|\mathcal{D}_uX_{km+l}\right\|^2\leq\cdots\leq \mathbb{E}\left\|\mathcal{D}_uX_{nm+w+1}\right\|^2\leq C.
\end{split}
\end{equation*}
If $p'=2$, without taking expectation in \eqref{MC-5}, multiplying \eqref{MC-5} by $\left\|\mathcal{D}_uX_{km+l+1}\right\|^2$ and then
taking expectation, Young's inequality leads to
\begin{align*}
&\frac{1}{2}\left(\mathbb{E}\left\|\mathcal{D}_uX_{km+l+1}\right\|^4-\mathbb{E}\left\|\mathcal{D}_uX_{km+l}\right\|^4+\mathbb{E}\left(\left\|\mathcal{D}_uX_{km+l+1}\right\|^2-\left\|\mathcal{D}_uX_{km+l}\right\|^2\right)\right)\\
&+\mathbb{E}\left\|\mathcal{D}_uX_{km+l+1}\right\|^2\left\|\mathcal{D}_uX_{km+l+1}-\mathcal{D}_uX_{km+l}\right\|^2\\
=&2\delta\mathbb{E}\left\|\mathcal{D}_uX_{km+l+1}\right\|^2\left\langle\mathcal{D}_uX_{km+l+1},\frac{\partial f}{\partial x}(X_{km+l+1},X_{km})\mathcal{D}_uX_{km+l+1}\right\rangle\\
&+2\mathbb{E}\left\|\mathcal{D}_uX_{km+l+1}\right\|^2\left\langle\mathcal{D}_uX_{km+l+1}-\mathcal{D}_uX_{km+l},\frac{\partial g}{\partial x}(X_{km+l},X_{km})\mathcal{D}_uX_{km+l}\Delta B_{km+1}\right\rangle\\
&+2\mathbb{E}\left(\left\|\mathcal{D}_uX_{km+l+1}\right\|^2-\left\|\mathcal{D}_uX_{km+l}\right\|^2\right)\left\langle\mathcal{D}_uX_{km+l},\frac{\partial g}{\partial x}(X_{km+l},X_{km})\mathcal{D}_uX_{km+l}\Delta B_{km+1}\right\rangle\\
=:&\tilde{R}_1+\tilde{R}_2+\tilde{R}_3.
\end{align*}
From \eqref{derivative-f-x1}, $\tilde{R}_1$ satisfies
\begin{align*}
\tilde{R}_1\leq-2\lambda_1\delta\mathbb{E}\left\|\mathcal{D}_uX_{km+l+1}\right\|^4+\tilde{R}_1+\tilde{R}_2.
\end{align*}
For $\tilde{R}_2$ and $\tilde{R}_3$, Young's inequality and \eqref{derivative-g-x1} lead to
\begin{align*}
\tilde{R}_2\leq&\mathbb{E}\left\|\mathcal{D}_uX_{km+l+1}\right\|^2\left\|\mathcal{D}_uX_{km+l+1}-\mathcal{D}_uX_{km+l}\right\|^2\\
&+\mathbb{E}\left\|\mathcal{D}_uX_{km+l+1}\right\|^2\left\|\frac{\partial g}{\partial x}(X_{km+l},X_{km})\mathcal{D}_uX_{km+l}\Delta B_{km+1}\right\|^2\\
\leq&\mathbb{E}\left\|\mathcal{D}_uX_{km+l+1}\right\|^2\left\|\mathcal{D}_uX_{km+l+1}-\mathcal{D}_uX_{km+l}\right\|^2+\lambda_3\delta\mathbb{E}\left\|\mathcal{D}_uX_{km+l}\right\|^4\\
&+\frac{1}{2}\epsilon_1\mathbb{E}\left(\left\|\mathcal{D}_uX_{km+l+1}\right\|^2-\left\|\mathcal{D}_uX_{km+l}\right\|^2\right)^2
+\frac{1}{2\epsilon_1}\mathbb{E}\left\|\frac{\partial g}{\partial x}(X_{km+l},X_{km})\mathcal{D}_uX_{km+l}\Delta B_{km+1}\right\|^4\\
\leq&\mathbb{E}\left\|\mathcal{D}_uX_{km+l+1}\right\|^2\left\|\mathcal{D}_uX_{km+l+1}-\mathcal{D}_uX_{km+l}\right\|^2+\left(\lambda_3\delta+\frac{3\lambda_3^3\delta^2}{2\epsilon_1}\right)\mathbb{E}\left\|\mathcal{D}_uX_{km+l}\right\|^4\\
&+\frac{1}{2}\epsilon_1\mathbb{E}\left(\left\|\mathcal{D}_uX_{km+l+1}\right\|^2-\left\|\mathcal{D}_uX_{km+l}\right\|^2\right)^2
\end{align*}
and
\begin{equation*}
\begin{split}
\tilde{R}_3\leq&\epsilon_2\mathbb{E}\left(\left\|\mathcal{D}_uX_{km+l+1}\right\|^2-\left\|\mathcal{D}_uX_{km+l}\right\|^2\right)^2\\
&+\frac{1}{\epsilon_2}\mathbb{E}\left\langle\mathcal{D}_uX_{km+l},\frac{\partial g}{\partial x}(X_{km+l},X_{km})\mathcal{D}_uX_{km+l}\Delta B_{km+1}\right\rangle^2\\
\leq&\epsilon_2\mathbb{E}\left(\left\|\mathcal{D}_uX_{km+l+1}\right\|^2-\left\|\mathcal{D}_uX_{km+l}\right\|^2\right)^2+\frac{\lambda_3\delta}{\epsilon_2}\mathbb{E}\left\|\mathcal{D}_uX_{km+l}\right\|^4.
\end{split}
\end{equation*}
Let $\epsilon_2+\frac{1}{2}\epsilon_1=\frac{1}{2}$, then
\begin{equation*}
\begin{split}
&\left(1+4\lambda_1\delta\right)\mathbb{E}\left\|\mathcal{D}_uX_{km+l+1}\right\|^4
\leq
\left(1+2\left(1+\frac{1}{\epsilon_2}\right)\lambda_3\delta+\frac{3\lambda_3^3\delta^2}{\epsilon_1}\right)\mathbb{E}\left\|\mathcal{D}_uX_{km+l}\right\|^4.
\end{split}
\end{equation*}
Since $\lambda_1-\lambda_2-1-2\lambda_3>0$, there exists $\delta'>0$ such that
$\mathbb{E}\left\|\mathcal{D}_uX_{km+l+1}\right\|^4
\leq\mathbb{E}\left\|\mathcal{D}_uX_{km+l}\right\|^4$ for all $\delta\in(0,\delta')$.
Following the same procedure as the case $p'=1$, there exists $C>0$ independent of $\delta$, $k$ and $l$ such that
\begin{equation*}
\begin{split}
\mathbb{E}\left\|\mathcal{D}_uX_{km+l+1}\right\|^4\leq C.
\end{split}
\end{equation*}
The inequality above shows $\mathbb{E}\|\mathcal{D}_uX_{km+l+1}\|^{2p'}\leq C$ with $p'=2$.
Repeating the procedure as above, we prove \eqref{BE-uniformly bounded-1-0} when $u\in[t_{km},t_{km+l})$.

\textbf{Case 3.} If $u<t_{km}$, then
\begin{equation}\label{MC-7}
\begin{split}
&\mathcal{D}_uX_{km+l+1}-\mathcal{D}_uX_{km+l}\\
=&\delta\frac{\partial f}{\partial x}(X_{km+l+1},X_{km})\mathcal{D}_uX_{km+l+1}+\delta\frac{\partial f}{\partial y}(X_{km+l+1},X_{km})\mathcal{D}_uX_{km}\\
&+\frac{\partial g}{\partial x}(X_{km+l},X_{km})\mathcal{D}_uX_{km+l}\Delta B_{km+1}+\frac{\partial g}{\partial y}(X_{km+l},X_{km})\mathcal{D}_uX_{km}\Delta B_{km+1}.
\end{split}
\end{equation}
Multiplying \eqref{MC-7} by $\mathcal{D}_uX_{km+l+1}$ and taking expectation, we have $LHS=RHS$, where
\begin{equation}\label{MC-8}
\begin{split}
LHS
=&\frac{1}{2}\left(\mathbb{E}\|\mathcal{D}_uX_{km+l+1}\|^2-\mathbb{E}\left\|\mathcal{D}_uX_{km+l}\right\|^2+\mathbb{E}\|\mathcal{D}_uX_{km+l+1}-\mathcal{D}_uX_{km+l}\|^2\right)
\end{split}
\end{equation}
and
\begin{equation}\label{MC-8-1}
\begin{split}
RHS=&\delta\mathbb{E}\left\langle \mathcal{D}_uX_{km+l+1},\frac{\partial f}{\partial x}(X_{km+l+1},X_{km})\mathcal{D}_uX_{km+l+1}\right\rangle\\
&+\delta\mathbb{E}\left\langle \mathcal{D}_uX_{km+l+1},\frac{\partial f}{\partial y}(X_{km+l+1},X_{km})\mathcal{D}_uX_{km}\right\rangle\\
&+\mathbb{E}\left\langle \mathcal{D}_uX_{km+l+1},\frac{\partial g}{\partial x}(X_{km+l},X_{km})\mathcal{D}_uX_{km+l}\Delta B_{km+1}\right\rangle\\
&+\mathbb{E}\left\langle \mathcal{D}_uX_{km+l+1},\frac{\partial g}{\partial y}(X_{km+l},X_{km})\mathcal{D}_uX_{km}\Delta B_{km+1}\right\rangle.
\end{split}
\end{equation}
By \eqref{derivative-f-x1} and Young's inequality, $RHS$ yields
\begin{align*}
RHS
\leq&\left(-\lambda_1+\frac{1}{2}\right)\delta\mathbb{E}\left\| \mathcal{D}_uX_{km+l+1}\right\|^2
+\frac{1}{2}\delta\mathbb{E}\left\|\frac{\partial f}{\partial y}(X_{km+l+1},X_{km})\mathcal{D}_uX_{km}\right\|^2\\
&+\mathbb{E}\left\langle \mathcal{D}_uX_{km+l+1}-\mathcal{D}_uX_{km+l},\frac{\partial g}{\partial x}(X_{km+l},X_{km})\mathcal{D}_uX_{km+l}\Delta B_{km+1}\right\rangle\\
&+\mathbb{E}\left\langle \mathcal{D}_uX_{km+l+1}-\mathcal{D}_uX_{km+l},\frac{\partial g}{\partial y}(X_{km+l},X_{km})\mathcal{D}_uX_{km}\Delta B_{km+1}\right\rangle\\
\leq&\left(-\lambda_1+\frac{1}{2}\right)\delta\mathbb{E}\left\| \mathcal{D}_uX_{km+l+1}\right\|^2
+\frac{1}{2}\mathbb{E}\left\| \mathcal{D}_uX_{km+l+1}-\mathcal{D}_uX_{km+l}\right\|^2\\
&+\left(\frac{1}{2}\lambda_2+\lambda_3\right)\delta\mathbb{E}\left\|\mathcal{D}_uX_{km}\right\|^2+\lambda_3\delta\mathbb{E}\left\|\mathcal{D}_uX_{km+l}\right\|^2,
\end{align*}
which implies
\begin{equation*}
\mathbb{E}\|\mathcal{D}_uX_{km+l+1}\|^2
\leq\left(1-\alpha_5\delta\right)\mathbb{E}\left\|\mathcal{D}_uX_{km+l}\right\|^2+\beta_5\delta\mathbb{E}\left\|\mathcal{D}_uX_{km}\right\|^2,
\end{equation*}
where $\alpha_5=\frac{2\lambda_1-1-2\lambda_3}{1+2\lambda_1\delta-\delta}$ and $\beta_5=\frac{\lambda_2+2\lambda_3}{1+2\lambda_1\delta-\delta}$. Since $\alpha_5>\beta_5>0$ and $0<\frac{\beta_5}{\alpha_5}+\left(1-\frac{\beta_5}{\alpha_5}\right)e^{-\alpha_5(l+1)\delta}<1$,
$l=0,1,\cdots,m-1$, under the condition $\lambda_1-1-\lambda_2-2\lambda_3>0$, we get
\begin{equation*}
\mathbb{E}\|\mathcal{D}_uX_{km+l+1}\|^2
\leq\left(\frac{\beta_5}{\alpha_5}+\left(1-\frac{\beta_5}{\alpha_5}\right)e^{-\alpha_5(l+1)\delta}\right)\mathbb{E}\left\|\mathcal{D}_uX_{km}\right\|^2.
\end{equation*}
If $l=m-1$, then $\mathbb{E}\|\mathcal{D}_uX_{(k+1)m}\|^2\leq\left(\frac{\beta_5}{\alpha_5}+\left(1-\frac{\beta_5}{\alpha_5}\right)e^{-\alpha_5}\right)\mathbb{E}\left\|\mathcal{D}_uX_{km}\right\|^2$. Since for any $u\geq 0$, there exists $n\in\mathbb{N}$ such that $u\in[n-1,n)$. Combining the result in Case 2, we obtain
\begin{equation*}
\begin{split}
\mathbb{E}\|\mathcal{D}_uX_{km+l+1}\|^2
\leq&\left(\frac{\beta_5}{\alpha_5}+\left(1-\frac{\beta_5}{\alpha_5}\right)e^{-\alpha_5(l+1)\delta}\right)\left(\frac{\beta_5}{\alpha_5}+\left(1-\frac{\beta_5}{\alpha_5}\right)e^{-\alpha}\right)^{k-n}\mathbb{E}\left\|\mathcal{D}_uX_{nm}\right\|^2\\
\leq&\mathbb{E}\left\|\mathcal{D}_uX_{nm}\right\|^2\leq C.
\end{split}
\end{equation*}
The inequality above shows $\mathbb{E}\|\mathcal{D}_uX_{km+l+1}\|^{2p'}\leq C$ with $p'=1$.
Following the same procedure as in Case 2, there exists $C>0$ independent of $\delta$, $k$ and $l$ such that $\mathbb{E}\|\mathcal{D}_uX_{km+l+1}\|^{2p'}\leq C$ for all $\delta\in(0,\delta_0'')$ with $\delta''_0>0$ sufficiently small and $p'=2,3,\cdots,p$.

Choosing $\tilde{\delta_0}=\min\{\delta_0',\delta_0''\}$, then \eqref{BE-uniformly bounded-1-0} holds for all $\delta\in(0,\tilde{\delta_0})$. The proof is completed.
\end{proof}
Except for the estimations of the partial derivatives of $f$ and $g$ with order 1, we also require that the coefficients $f$ and $g$ have partial derivatives up to  order 3.  For a general function $h: \mathbb{R}^d\times\mathbb{R}^d\rightarrow \mathbb{R}^d$, $(x,y)\mapsto h(x,y)$, we use $D_1^{(n)}$ and $D_2^{(n)}$ to denote the partial differentiation operators with order $n$ of $h$ with respect to the vectors $x$ and $y$, respectively. Then we also need the assumptions as follows.
\begin{assumption}\label{f-polynomial}
For any $x$, $x'$, $y$, $y'\in\mathbb{R}^d$, there exist two positive constants $K$ and $q$ such that
\begin{equation*}
\begin{split}
\left\|f(x,y)-f(x',y)\right\|^2\leq K\left(1+\|x\|^q+\|x'\|^q\right)\|x-x'\|^2.
\end{split}
\end{equation*}
\end{assumption}

\begin{assumption}\label{higher derivatives}
Assume that $f$ and $g$ have all continuous partial derivatives up to  order 2. For any $x$, $x'$, $y$, $y'$, $\xi$ and $\eta\in\mathbb{R}^d$, there exist two positive constants $K$ and $q$ such that
\begin{align*}
\left\|D_i^{(1)}f_I(x,y)\xi-D_i^{(1)}f_I(x',y)\xi\right\|^2\leq K\left(1+\|x\|^q+\|x'\|^q\right)\|x-x'\|^2\|\xi\|^2,\\
\left\|D_i^{(1)}f_I(x,y)\xi-D_i^{(1)}f_I(x,y')\xi\right\|^2\leq K\|y-y'\|^2\|\xi\|^2,\\
\left\|D_i^{(2)}f_I(x,y)(\xi,\eta)-D_i^{(2)}f_I(x',y)(\xi,\eta)\right\|^2\leq K\left(1+\|x\|^q+\|x'\|^q\right)\|x-x'\|^2\|\xi\|^2\|\eta\|^2,\\
\left\|D_i^{(2)}f_I(x,y)(\xi,\eta)-D_i^{(2)}f_I(x,y')(\xi,\eta)\right\|^2\leq K\|y-y'\|^2\|\xi\|^2\|\eta\|^2,\\
\left\|D_i^{(1)}g_I(x,y)\xi-D_i^{(1)}g_I(x',y')\xi\right\|^2\leq K\left(\|x-x'\|^2+\|y-y'\|^2\right)\|\xi\|^2,\\
\left\|D_i^{(2)}g_I(x,y)(\xi,\eta)-D_i^{(2)}g_I(x',y')(\xi,\eta)\right\|^2\leq K\left(\|x-x'\|^2+\|y-y'\|^2\right)\|\xi\|^2\|\eta\|^2
\end{align*}
for $i=1,2$ and $I=1,2,\cdots,d$, where $f_I$ and $g_I$ are the $I$th components of $f$ and $g$, respectively.
\end{assumption}

Form the assumption above, it follows that
\begin{equation*}
\left\|\frac{\partial^2 f_I}{\partial x^2}(x,y)(\xi,\eta)\right\|^2
\leq K\left(1+2\|x\|^q\right)\|\xi\|^2\|\eta\|^2,
\qquad
\left\|\frac{\partial^2 f_I}{\partial x\partial y}(x,y)(\xi,\eta)\right\|^2\leq K\|\xi\|^2\|\eta\|^2,
\end{equation*}
\begin{equation*}
\left\|\frac{\partial^2 f_I}{\partial y^2}(x,y)(\xi,\eta)\right\|^2
\leq K\|\xi\|^2\|\eta\|^2,
\qquad
\left\|\frac{\partial^2 f_I}{\partial y\partial x}(x,y)(\xi,\eta)\right\|^2
\leq K\left(1+2\|x\|^q\right)\|\xi\|^2\|\eta\|^2,
\end{equation*}
\begin{equation*}
\left\|\frac{\partial^2 g_I}{\partial x^2}(x,y)(\xi,\eta)\right\|^2
\leq K\|\xi\|^2\|\eta\|^2,
\qquad
\left\|\frac{\partial^2 g_I}{\partial x\partial y}(x,y)(\xi,\eta)\right\|^2
\leq K\|\xi\|^2\|\eta\|^2
\end{equation*}
and
\begin{equation*}
\left\|\frac{\partial^2 g_I}{\partial y^2}(x,y)(\xi,\eta)\right\|^2
\leq K\|\xi\|^2\|\eta\|^2,
\qquad
\left\|\frac{\partial^2 g_I}{\partial y\partial x}(x,y)(\xi,\eta)\right\|^2
\leq K\|\xi\|^2\|\eta\|^2.
\end{equation*}
Similarly, if $f$ and $g$ have all continuous partial derivatives up to  order 3, then
\begin{equation*}
\left\|\frac{\partial^3 f_I}{\partial x^3}(x,y)(\xi,\eta,\gamma)\right\|^2
\leq K\left(1+2\|x\|^q\right)\|\xi\|^2\|\eta\|^2\|\gamma\|^2
\end{equation*}
and
\begin{equation*}
\left\|\frac{\partial^3 f_I}{\partial x^2\partial y}(x,y)(\xi,\eta,\gamma)\right\|^2
\leq K\|\xi\|^2\|\eta\|^2\|\gamma\|^2,
\qquad
\left\|\frac{\partial^3 f_I}{\partial x\partial y^2}(x,y)(\xi,\eta,\gamma)\right\|^2
\leq K\|\xi\|^2\|\eta\|^2\|\gamma\|^2.
\end{equation*}

 \begin{lemma}\label{DuX-uniform bounded}
Let the conditions in Lemma \ref{uniform-bounded-2p} hold.  Assume that the Fr\'echet derivatives of $f$ and $g$ exist and $\mathbb{E}\|\mathcal{D}_u\eta\|^{2p}<\infty$, then
there exist two positive constants $C$ and $\nu_2$ independent of $t$ such that
\begin{equation}
\mathbb{E}\left\|\mathcal{D}_uX^{i,\eta}(t)\right\|^{2p}\leq Ce^{-\nu_2(t-u\vee i)}\left(1+\mathbb{E}\left\|\mathcal{D}_u\eta\right\|^{2p}\right),\quad t\geq i.
\end{equation}
 \end{lemma}
 \begin{proof}
Since
\begin{equation*}
\begin{split}
X^{i,\eta}(t)=\eta+\int_{i}^tf(X^{i,\eta}(s),X^{i,\eta}([s]))ds+\int_{i}^tg(X^{i,\eta}(s),X^{i,\eta}([s]))dB(s),
\end{split}
\end{equation*}
we have $\mathcal{D}_uX^{i,\eta}(t)=0$ for $u\geq t$, and for $u<t$,
\begin{equation}\label{DuX}
\begin{split}
\mathcal{D}_uX^{i,\eta}(t)=&\mathcal{D}_u\eta\mathbf{1}_{\{u<i\}}+g(X^{i,\eta}(u),X^{i,\eta}([u]))\mathbf{1}_{\{i\leq u<t\}}\\
&+\int_{u}^t\frac{\partial f}{\partial x}(X^{i,\eta}(s),X^{i,\eta}([s]))\mathcal{D}_uX^{i,\eta}(s)\mathbf{1}_{[i,t]}(s)ds\\
&+\int_{u}^t\frac{\partial f}{\partial y}(X^{i,\eta}(s),X^{i,\eta}([s]))\mathcal{D}_uX^{i,\eta}([s])\mathbf{1}_{[i,t]}(s)ds\\
&+\int_{u}^t\frac{\partial g}{\partial x}(X^{i,\eta}(s),X^{i,\eta}([s]))\mathcal{D}_uX^{i,\eta}(s)\mathbf{1}_{[i,t]}(s)dB(s)\\
&+\int_{u}^t\frac{\partial g}{\partial y}(X^{i,\eta}(s),X^{i,\eta}([s]))\mathcal{D}_uX^{i,\eta}([s])\mathbf{1}_{[i,t]}(s)dB(s).
\end{split}
\end{equation}

\textbf{Case 1.} If $i \leq u<t$, by denoting $I(t):=\mathcal{D}_uX^{i,\eta}(t)$, then
\begin{equation*}
\begin{split}
I(t)=&I(u)+\int_{u}^t\frac{\partial f}{\partial x}(X^{i,\eta}(s),X^{i,\eta}([s]))I(s)ds+\int_{u}^t\frac{\partial f}{\partial y}(X^{i,\eta}(s),X^{i,\eta}([s]))I([s])ds\\
&+\int_{u}^t\frac{\partial g}{\partial x}(X^{i,\eta}(s),X^{i,\eta}([s]))I(s)dB(s)+\int_{u}^t\frac{\partial g}{\partial y}(X^{i,\eta}(s),X^{i,\eta}([s]))I([s])dB(s)
\end{split}
\end{equation*}
where $I(u)=g(X^{i,\eta}(u),X^{i,\eta}([u]))$. If $[t]\leq u< t$, then $\mathcal{D}_uX^{i,\eta}([t])=0$. And  It\^o's formula leads to
\begin{equation*}
\begin{split}
\mathbb{E}\left\|I(t)\right\|^{2p}\leq&\mathbb{E}\left\|I(u)\right\|^{2p}
+2p\mathbb{E}\int_{u}^t\left\|I(s)\right\|^{2(p-1)}\left\langle I(s),\frac{\partial f}{\partial x}(X^{i,\eta}(s),X^{i,\eta}([s]))I(s)\right\rangle ds\\
&+p(2p-1)\mathbb{E}\int_{u}^t\left\|I(s)\right\|^{2(p-1)}\left\|\frac{\partial g}{\partial x}(X^{i,\eta}(s),X^{i,\eta}([s]))I(s)\right\|^2ds\\
\leq&\mathbb{E}\left\|I(u)\right\|^{2p}
-\left(2\lambda_1-\lambda_3(2p-1)\right)p\int_{u}^t\mathbb{E}\left\|I(s)\right\|^{2p} ds.
\end{split}
\end{equation*}
From \cite[Lemma 8.2]{Ito}, it follows
\begin{equation*}
\begin{split}
\mathbb{E}\left\|I(t)\right\|^{2p}\leq e^{-(2\lambda_1-\lambda_3(2p-1))p(t-u)}\mathbb{E}\left\|I(u)\right\|^{2p}.
\end{split}
\end{equation*}
And if $i\leq u<[t]$, for any $\alpha>0$, applying It\^o's formula to $e^{2\alpha pt}\left\|I(t)\right\|^{2p}$, we obtain
\begin{align*}
e^{2\alpha pt}\mathbb{E}\left\|I(t)\right\|^{2p}\leq&e^{2\alpha p[t]}\mathbb{E}\left\|I([t])\right\|^{2p}+2\alpha p\mathbb{E}\int_{[t]}^t e^{2\alpha ps}\mathbb{E}\left\|I(s)\right\|^{2p}ds\\
&+2p\mathbb{E}\int_{[t]}^te^{2\alpha ps}\left\|I(s)\right\|^{2(p-1)}\left\langle I(s),\frac{\partial f}{\partial x}(X^{i,\eta}(s),X^{i,\eta}([s]))I(s)\right\rangle ds\\
&+2p\mathbb{E}\int_{[t]}^te^{2\alpha ps}\left\|I(s)\right\|^{2(p-1)}\left\langle I(s),\frac{\partial f}{\partial y}(X^{i,\eta}(s),X^{i,\eta}([s]))I([s])\right\rangle ds\\
&+2p(2p-1)\mathbb{E}\int_{[t]}^te^{2\alpha ps}\left\|I(s)\right\|^{2(p-1)}\left\|\frac{\partial g}{\partial x}(X^{i,\eta}(s),X^{i,\eta}([s]))I(s)\right\|^2ds\\
&+2p(2p-1)\mathbb{E}\int_{[t]}^te^{2\alpha ps}\left\|I(s)\right\|^{2(p-1)}\left\|\frac{\partial g}{\partial y}(X^{i,\eta}(s),X^{i,\eta}([s]))I([s])\right\|^2ds.
\end{align*}
Following the same procedure as in Lemma \ref{DX-uniform bounded}, we get
\begin{equation*}
\begin{split}
\mathbb{E}\left\|I(t)\right\|^{2p}\leq&Ce^{-\nu_1(t-[u]-1)}\mathbb{E}\left\|I([u]+1)\right\|^{2p}\\
\leq&Ce^{-\nu_1(t-[u]-1)}\cdot e^{-(2\lambda_1-\lambda_3(2p-1))p([u]+1-u)}\mathbb{E}\left\|I(u)\right\|^{2p}\\
=&Ce^{-\nu_2(t-u)}\mathbb{E}\left\|I(u)\right\|^{2p},
\end{split}
\end{equation*}
where $\nu_2=\min\{\nu_1,(2\lambda_1-\lambda_3(2p-1))p\}$.

\textbf{Case 2.} If $u<i$, then $I(u)=\mathcal{D}_u\eta$ and
$$\mathbb{E}\left\|I(t)\right\|^{2p}\leq Ce^{-\nu_2(t-i)}\mathbb{E}\left\|I(u)\right\|^{2p}.$$
Since $\mathbb{E}\|X(t)\|^{2p}\leq C$ for all $t>0$ and $C$ is independent of $t$, we obtain
\begin{equation*}
\begin{split}
\mathbb{E}\left\|I(u)\right\|^{2p}
\leq&\mathbb{E}\left\|\mathcal{D}_u\eta\right\|^{2p}+\mathbb{E}\left\|g(X^{i,\eta}(u),X^{i,\eta}([u]))\right\|^{2p}
\leq C\left(1+\mathbb{E}\left\|\mathcal{D}_u\eta\right\|^{2p}\right).
\end{split}
\end{equation*}
Hence
\begin{equation*}
\begin{split}
\mathbb{E}\left\|\mathcal{D}_uX^{i,\eta}(t)\right\|^{2p}\leq&Ce^{-\nu_2(t-u\vee i)}\left(1+\mathbb{E}\left\|\mathcal{D}_u\eta\right\|^{2p}\right),
\end{split}
\end{equation*}
where $C$ is independent of $t$. We complete the proof
 \end{proof}
 \begin{lemma}\label{DuDX-uniform bounded}
Let the conditions in Lemma \ref{uniform-bounded-2p} hold and $p\geq4$. Assume also that Assumptions \ref{f-polynomial} and \ref{higher derivatives} hold, then
there exist two positive constants $C$ and $\nu_3$ independent of $t$ such that
\begin{equation}\label{DuDX-2p bounded}
\mathbb{E}\left\|\mathcal{D}_uDX^{i,\eta}(t)\xi\right\|^{2p'}\leq Ce^{-\nu_3(t-u\vee i)}\|\xi\|^{2p'}+Ce^{-\nu_1([t]-i-1)p'/p}\left(1+\mathbb{E}\left\|\mathcal{D}_u\eta\right\|^{2p'}\right)\|\xi\|^{2p'},\quad t\geq i
\end{equation}
for any $1\leq p'\leq \min\{\frac{p}{4},\frac{p}{q}\}$, $\xi\in\mathbb{R}^d$, $\eta\in \mathbb{D}^{1,2p}$ and $i\in\mathbb{N}$.
 \end{lemma}
 \begin{proof}
 For any $\xi\in\mathbb{R}^d$, denote $J(t):=\mathcal{D}_uDX^{i,\eta}(t)\xi$, then
\begin{align*}
J(t)
=&\frac{\partial g}{\partial x}(X^{i,\eta}(u),X^{i,\eta}([u]))H(u)\mathbf{1}_{\{i\leq u<t\}}
+\frac{\partial g}{\partial y}(X^{i,\eta}(u),X^{i,\eta}([u]))H([u])\mathbf{1}_{\{i\leq u<t\}}\\
&+\!\int_{u}^t\left(\frac{\partial^2 f}{\partial x^2}(X^{i,\eta}(s),X^{i,\eta}([s]))(H(s),I(s))+\frac{\partial^2 f}{\partial x\partial y}(X^{i,\eta}(s),X^{i,\eta}([s]))(H(s),I([s]))\right)\mathbf{1}_{[i,t]}(s)ds\\
&+\!\int_{u}^t\left(\frac{\partial f}{\partial x}(X^{i,\eta}(s),X^{i,\eta}([s]))J(s)+\frac{\partial f}{\partial y}(X^{i,\eta}(s),X^{i,\eta}([s]))J([s])\right) \mathbf{1}_{[i,t]}(s)ds\\
&+\!\int_{u}^t\left(\frac{\partial^2 f}{\partial y\partial x}(X^{i,\eta}(s),X^{i,\eta}([s]))(H([s]),I(s))+\frac{\partial^2 f}{\partial y^2}(X^{i,\eta}(s),X^{i,\eta}([s]))(H([s]),I([s]))\right)\mathbf{1}_{[i,t]}(s)ds\\
&+\!\int_{ u}^t\left(\frac{\partial^2 g}{\partial x^2}(X^{i,\eta}(s),X^{i,\eta}([s]))(H(s),I(s))+\frac{\partial^2 g}{\partial x\partial y}(X^{i,\eta}(s),X^{i,\eta}([s]))(H(s),I([s]))\right)\mathbf{1}_{[i,t]}(s)dB(s)\\
&+\int_{u}^t\left(\frac{\partial g}{\partial x}(X^{i,\eta}(s),X^{i,\eta}([s]))J(s) +\frac{\partial g}{\partial y}(X^{i,\eta}(s),X^{i,\eta}([s]))J([s]) \right)\mathbf{1}_{[i,t]}(s)dB(s)\\
&+\!\int_{ u}^t\left(\frac{\partial^2 g}{\partial y\partial x}(X^{i,\eta}(s),X^{i,\eta}([s]))(H([s]),I(s))+\frac{\partial^2 g}{\partial y^2}(X^{i,\eta}(s),X^{i,\eta}([s]))(H([s]),I([s]))\right)\mathbf{1}_{[i,t]}(s)dB(s)\\
\end{align*}
for $u<t$. Since $J(t)=0$ for $u\geq t$, we only consider the case $u<t$.

\textbf{Case 1.} If $u\geq [t]$, then $[s]=[t]$, $I([t])=0$, $J([t])=0$ and $J(u)=\frac{\partial g}{\partial x_1}(X^{i,\eta}(u),X^{i,\eta}([u]))H(u)
+\frac{\partial g}{\partial x_2}(X^{i,\eta}(u),X^{i,\eta}([u]))H([u])$. From It\^o's formula, it follows
\begin{align*}
\mathbb{E}\left\|J(t)\right\|^{2p'}
\leq&\mathbb{E}\left\|J(u)\right\|^{2p'}
+2p'\mathbb{E}\int_{ u}^t\left\|J(s)\right\|^{2(p'-1)}\left\langle J(s), \frac{\partial^2 f}{\partial x^2}(X^{i,\eta}(s),X^{i,\eta}([t]))(H(s),I(s))\right\rangle ds\\
&+2p'\mathbb{E}\int_{ u}^t\left\|J(s)\right\|^{2(p'-1)}\left\langle J(s), \frac{\partial f}{\partial x}(X^{i,\eta}(s),X^{i,\eta}([t]))J(s)\right\rangle  ds\\
&+2p'\mathbb{E}\int_{ u}^t\left\|J(s)\right\|^{2(p'-1)}\left\langle J(s),\frac{\partial^2 f}{\partial y\partial x}(X^{i,\eta}(s),X^{i,\eta}([t]))(H([t]),I(s))\right\rangle ds\\
&+2p'(2p'-1)\mathbb{E}\int_{ u}^t\left\|J(s)\right\|^{2(p'-1)}\left\|\frac{\partial g}{\partial x}(X^{i,\eta}(s),X^{i,\eta}([t]))J(s) \right\|^2ds
\\
&+4p'(2p'-1)\mathbb{E}\int_{u}^t\left\|J(s)\right\|^{2(p'-1)}\left\|\frac{\partial^2 g}{\partial x^2}(X^{i,\eta}(s),X^{i,\eta}([t]))(H(s),I(s))\right\|^2ds\\
&+4p'(2p'-1)\mathbb{E}\int_{ u}^t\left\|J(s)\right\|^{2(p'-1)}\left\|\frac{\partial^2 g}{\partial y\partial x}(X^{i,\eta}(s),X^{i,\eta}([t]))(H([t]),I(s))\right\|^2ds.
\end{align*}
And Young's inequality yields
\begin{align*}
\mathbb{E}\left\|J(t)\right\|^{2p'}
\leq&\mathbb{E}\left\|J(u)\right\|^{2p'}-\left(2\lambda_1-2\lambda_3(2p'-1)-4\epsilon_1-8(2p'-1)\epsilon_2\right)p'
\mathbb{E}\int_{ u}^t\left\|J(s)\right\|^{2p'} ds\\
&+\left(\frac{2p'-1}{2p'\epsilon_1}\right)^{2p'-1}\mathbb{E}\int_{ u}^t\left\| \frac{\partial^2 f}{\partial x^2}(X^{i,\eta}(s),X^{i,\eta}([t]))(H(s),I(s))\right\|^{2p'}ds\\
&+\left(\frac{2p'-1}{2p'\epsilon_1}\right)^{2p'-1}\mathbb{E}\int_{ u}^t\left\|\frac{\partial^2 f}{\partial y\partial x}(X^{i,\eta}(s),X^{i,\eta}([t]))(H([t]),I(s))\right\|^{2p'} ds\\
&+2(2p'-1)\left(\frac{p'-1}{p'\epsilon_2}\right)^{p'-1}\mathbb{E}\int_{u}^t\left\|\frac{\partial^2 g}{\partial x^2}(X^{i,\eta}(s),X^{i,\eta}([t]))(H(s),I(s))\right\|^{2p'}ds\\
&+2(2p'-1)\left(\frac{p'-1}{p'\epsilon_2}\right)^{p'-1}\mathbb{E}\int_{ u}^t\left\|\frac{\partial^2 g}{\partial y\partial x}(X^{i,\eta}(s),X^{i,\eta}([t]))(H([t]),I(s))\right\|^{2p'}ds.
\end{align*}
Taking $\epsilon_1=\frac{1}{2}\lambda_3$ and $\epsilon_2=\frac{2p}{4(2p-1)}\lambda_3$
and using the estimates of the partial derivatives of $f$ and $g$ with order 1 and 2, we obtain
\begin{align*}
\mathbb{E}\left\|J(t)\right\|^{2p'}
\leq&\mathbb{E}\left\|J(u)\right\|^{2p'}-2\left(\lambda_1-2\lambda_3-4\lambda_3(p'-1)\right)p'
\mathbb{E}\int_{ u}^t\left\|J(s)\right\|^{2p'} ds\\
&+\left(\frac{2p'-1}{\lambda_3p'}\right)^{2p'-1}\mathbb{E}\int_{ u}^t\left|\sum_{I=1}^d\left|\frac{\partial^2 f_I}{\partial x^2}(X^{i,\eta}(s),X^{i,\eta}([t]))(H(s),I(s))\right|^2\right|^{p'}ds\\
&+\left(\frac{2p'-1}{\lambda_3p'}\right)^{2p'-1}\mathbb{E}\int_{ u}^t\left|\sum_{I=1}^d\left|\frac{\partial^2 f}{\partial y\partial x}(X^{i,\eta}(s),X^{i,\eta}([t]))(H([t]),I(s))\right|^2\right|^{p'} ds\\
&+2(2p'-1)\left(\frac{2(p'-1)(2p-1)}{p'^2\lambda_3}\right)^{p'-1}\mathbb{E}\int_{u}^t\left|\sum_{I=1}^d\left|\frac{\partial^2 g}{\partial x^2}(X^{i,\eta}(s),X^{i,\eta}([t]))(H(s),I(s))\right|^2\right|^{p'}ds\\
&+2(2p'-1)\left(\frac{2(p'-1)(2p'-1)}{p'^2\lambda_3}\right)^{p'-1}\mathbb{E}\int_{ u}^t\left|\sum_{I=1}^d\left|\frac{\partial^2 g}{\partial y\partial x}(X^{i,\eta}(s),X^{i,\eta}([t]))(H([t]),I(s))\right|^2\right|^{p'}ds\\
\leq&\mathbb{E}\left\|J(u)\right\|^{2p'}-2\left(\lambda_1-2\lambda_3-4\lambda_3(p'-1)\right)p'
\mathbb{E}\int_{ u}^t\left\|J(s)\right\|^{2p'} ds\\
&+\left(\frac{2p'-1}{\lambda_3p'}\right)^{2p'-1}(Kd)^{p'}\int_{ u}^t\mathbb{E}\left(\left(1+2\left\|X^{i,\eta}(s)\right\|^q\right)^{p'}\|H(s)\|^{2p'}\|I(s)\|^{2p'}\right)ds\\
&+\left(\frac{2p'-1}{\lambda_3p'}\right)^{2p'-1}(Kd)^{p'}\int_{ u}^t\mathbb{E}\left(\left(1+2\left\|X^{i,\eta}(s)\right\|^q\right)^{p'}\|H([t])\|^{2p'}\|I(s)\|^{2p'}\right)ds
\\
&+2(2p'-1)\left(\frac{2(p'-1)(2p-1)}{p'^2\lambda_3}\right)^{p'-1}(Kd)^{p'}\int_{u}^t\mathbb{E}\left(\left\|H(s)\right\|^{2p'}\left\|I(s))\right\|^{2p'}\right)ds\\
&+2(2p'-1)\left(\frac{2(p'-1)(2p-1)}{p'^2\lambda_3}\right)^{p'-1}(Kd)^{p'}\int_{ u}^t\mathbb{E}\left(\left\|H([t])\right\|^{2p'}\left\|I(s))\right\|^{2p'}\right)ds.
\end{align*}
Using H\"{o}lder inequality, $qp'\leq p$ and $4p'\leq p$, Lemmas \ref{DX-uniform bounded} and \ref{DuX-uniform bounded} lead to
 \begin{equation*}
\begin{split}
&\mathbb{E}\left(\left(1+2\left\|X^{i,\eta}(s)\right\|^q\right)^{p'}\|H(s)\|^{2p'}\|I(s)\|^{2p'}\right)\\
\leq&\left(\mathbb{E}\left(1+2\left\|X^{i,\eta}(s)\right\|^q\right)^{2p'}\right)^{\frac{1}{2}}\left(\mathbb{E}\|H(s)\|^{8p'}\right)^{\frac{1}{4}}\left(\mathbb{E}\|I(s)\|^{8p'}\right)^{\frac{1}{4}}\\
\leq&\left(\mathbb{E}\left(1+2\left\|X^{i,\eta}(s)\right\|^q\right)^{2p'}\right)^{\frac{1}{2}}\left(\mathbb{E}\|H(s)\|^{2p}\right)^{\frac{p'}{p}}\left(\mathbb{E}\|I(s)\|^{2p}\right)^{\frac{p'}{p}}\\
\leq&C\left(1+\mathbb{E}\left\|\mathcal{D}_u\eta\right\|^{2p'}\right)\|\xi\|^{2p'}e^{-\nu_1(s-i)p'/p}e^{-\nu_2(s-u)p'/p}\\
\leq&C\left(1+\mathbb{E}\left\|\mathcal{D}_u\eta\right\|^{2p'}\right)\|\xi\|^{2p'}e^{-\nu_1([t]-i)p'/p}.
\end{split}
\end{equation*}
Similarly,
\begin{equation*}
\begin{split}
&\mathbb{E}\left(\|H(s)\|^{2p'}\|I(s)\|^{2p'}\right)\leq Ce^{-\nu_1([t]-i)p'/p}\left(1+\mathbb{E}\left\|\mathcal{D}_u\eta\right\|^{2p'}\right)\|\xi\|^{2p'},
\end{split}
\end{equation*}
 \begin{equation*}
\begin{split}
&\mathbb{E}\left(\left(1+2\left\|X^{i,\eta}(s)\right\|^q\right)^{p'}\|H([t])\|^{2p'}\|I(s)\|^{2p'}\right)
\leq Ce^{-\nu_1([t]-i)p'/p}\left(1+\mathbb{E}\left\|\mathcal{D}_u\eta\right\|^{2p'}\right)\|\xi\|^{2p'}
\end{split}
\end{equation*}
and
 \begin{equation*}
\begin{split}
\mathbb{E}\left(\|H([t])\|^{2p'}\|I(s)\|^{2p'}\right)
\leq Ce^{-\nu_1([t]-i)p'/p}\left(1+\mathbb{E}\left\|\mathcal{D}_u\eta\right\|^{2p'}\right)\|\xi\|^{2p'}.
\end{split}
\end{equation*}
Let $\bar{\alpha}=\lambda_1-2\lambda_3-4\lambda_3(p'-1)$, then $2\bar{\alpha}p'>0$ by the condition $\lambda_1-\lambda_2-1-2\lambda_3\geq 4\lambda_3(p'-1)$. Therefore, \cite[Lemma 8.2]{Ito} leads to
\begin{equation*}
\begin{split}
\mathbb{E}\left\|J(t)\right\|^{2p'}
\leq&e^{-2\bar{\alpha} p'(t-u)}\mathbb{E}\left\|J(u)\right\|^{2p'}+Ce^{-\nu_1([t]-i)p'/p}\left(1+\mathbb{E}\left\|\mathcal{D}_u\eta\right\|^{2p'}\right)\|\xi\|^{2p'}.
\end{split}
\end{equation*}

\textbf{Case 2. } If $i\leq u<[t]$, applying It\^o's formula to $e^{2\tilde{\alpha}p'}\left\|J(t)\right\|^{2p'}$, $\tilde{\alpha}>0$, we have
\begin{align*}
e^{2\tilde{\alpha}p't} \mathbb{E}\left\|J(t)\right\|^{2p'}\leq&e^{2\tilde{\alpha}p'[t]} \mathbb{E}\left\|J([t])\right\|^{2p'}+2\tilde{\alpha}p'\mathbb{E}\int_{[t]}^te^{2\tilde{\alpha}ps} \left\|J(s)\right\|^{2p'}ds\\
&+2p'\mathbb{E}\int_{[t]}^te^{2\tilde{\alpha}ps} \left\|J(s)\right\|^{2(p'-1)}\left\langle J(s),\frac{\partial f}{\partial x}(X^{i,\eta}(s),X^{i,\eta}([t]))J(s) \right\rangle ds\\
&+2p'\mathbb{E}\int_{[t]}^te^{2\tilde{\alpha}ps} \left\|J(s)\right\|^{2(p'-1)}\left\langle J(s),\frac{\partial f}{\partial y}(X^{i,\eta}(s),X^{i,\eta}([t]))J([t]) \right\rangle ds\\
&+2p'\mathbb{E}\int_{ [t]}^te^{2\tilde{\alpha}ps} \left\|J(s)\right\|^{2(p'-1)}\left\langle J(s),\mathcal{A}(s)\right\rangle ds\\
&+p'(2p'-1)\mathbb{E}\int_{ [t]}^te^{2\tilde{\alpha}ps} \left\|J(s)\right\|^{2(p'-1)}\bigg\|\frac{\partial g}{\partial x}(X^{i,\eta}(s),X^{i,\eta}([t]))J(s)
\\
&\quad\quad\quad\quad\quad\quad\quad\quad\quad\quad\quad+\frac{\partial g}{\partial y}(X^{i,\eta}(s),X^{i,\eta}([t]))J([t]) +
\mathcal{B}(s)\bigg\|^2ds,
\end{align*}
where
\begin{align*}
\mathcal{A}(s)=&\frac{\partial^2 f}{\partial x^2}(X^{i,\eta}(s),X^{i,\eta}([t]))(H(s),I(s))
+\frac{\partial^2 f}{\partial x\partial y}(X^{i,\eta}(s),X^{i,\eta}([t]))(H(s),I([t]))\\
&+\frac{\partial^2 f}{\partial y\partial x}(X^{i,\eta}(s),X^{i,\eta}([t]))(H([t]),I(s))
+\frac{\partial^2 f}{\partial y^2}(X^{i,\eta}(s),X^{i,\eta}([t]))(H([t]),I([t]))
\end{align*}
and
\begin{align*}
\mathcal{B}(s)=&\frac{\partial^2 g}{\partial x^2}(X^{i,\eta}(s),X^{i,\eta}([t]))(H(s),I(s))
+\frac{\partial^2 g}{\partial x\partial y}(X^{i,\eta}(s),X^{i,\eta}([t]))(H(s),I([t]))\\
&+\frac{\partial^2 g}{\partial y\partial x}(X^{i,\eta}(s),X^{i,\eta}([t]))(H([t]),I(s))
+\frac{\partial^2 g}{\partial y^2}(X^{i,\eta}(s),X^{i,\eta}([t]))(H([t]),I([t])).
\end{align*}
By the estimates of all the partial derivatives of $f$ and $g$ up to order 2 and Young's inequality, we get
\begin{align*}
&e^{2\tilde{\alpha}p't} \mathbb{E}\left\|J(t)\right\|^{2p'}\\
\leq&e^{2\tilde{\alpha}p'[t]} \mathbb{E}\left\|J([t])\right\|^{2p'}+\left(2\tilde{\alpha}-2\lambda_1+1+2\epsilon_1+4\lambda_3(2p'-1)\right)p'\mathbb{E}\int_{[t]}^te^{2\tilde{\alpha}p's} \left\|J(s)\right\|^{2p'}ds\\
&+p'\mathbb{E}\int_{[t]}^te^{2\tilde{\alpha}p's} \left\|J(s)\right\|^{2(p'-1)}\left\|\frac{\partial f}{\partial y}(X^{i,\eta}(s),X^{i,\eta}([t]))J([t]) \right\|^2 ds\\
&+\left(\frac{2p'-1}{2p'\epsilon_1}\right)^{2p'-1}\mathbb{E}\int_{ [t]}^te^{2\tilde{\alpha}p's} \left\|\mathcal{A}(s)\right\|^{2p'} ds
+4p'(2p'-1)\mathbb{E}\int_{ [t]}^te^{2\tilde{\alpha}p's} \left\|J(s)\right\|^{2(p'-1)}\left\|\mathcal{B}(s)\right\|^2ds\\
&+4p'(2p'-1)\mathbb{E}\int_{ [t]}^te^{2\tilde{\alpha}p's} \left\|J(s)\right\|^{2(p'-1)}\left\|\frac{\partial g}{\partial y}(X^{i,\eta}(s),X^{i,\eta}([t]))J([t]) \right\|^2ds\\
    \leq&e^{2\tilde{\alpha}p'[t]} \mathbb{E}\left\|J([t])\right\|^{2p'}+\left(\lambda_2+4\lambda_3(2p'-1)\right)\mathbb{E}\int_{[t]}^te^{2\tilde{\alpha}p's} \left\|J([t]) \right\|^{2p'} ds\\
&+\left(\left(2\tilde{\alpha}-2\lambda_1+1+2\epsilon_1+6\lambda_3(2p'-1)+4(2p'-1)\epsilon_2+\lambda_2\right)p'+\lambda_2+4\lambda_3(2p'-1)\right)\mathbb{E}\int_u^te^{2\tilde{\alpha}p's} \left\|J(s)\right\|^{2p'}ds\\
&+\left(\frac{2p'-1}{2p'\epsilon_1}\right)^{2p'-1}\mathbb{E}\int_{ [t]}^te^{2\tilde{\alpha}p's} \left\|\mathcal{A}(s)\right\|^{2p'} ds
+2(2p'-1)\left(\frac{p'-1}{p'\epsilon_2}\right)^{p'-1}\mathbb{E}\int_{ [t]}^te^{2\tilde{\alpha}p's} \left\|\mathcal{B}(s)\right\|^{2p'}ds.
\end{align*}
Taking  $2\tilde{\alpha}p'=\left(2\lambda_1-1-2\epsilon_1-6\lambda_3(2p'-1)-4(2p'-1)\epsilon_2-\lambda_2\right)p'+\lambda_2+4\lambda_3(2p'-1)$, $\tilde{\beta}=\lambda_2+4\lambda_3(2p'-1)$, $\epsilon_1=\frac{1}{4}$ and $\epsilon_2=\frac{1}{8(2p'-1)}$, then $2\tilde{\alpha}p'=2\lambda_1p'-2p'-\lambda_2p'-6\lambda_3(2p'-1)+\lambda_2+4\lambda_3(2p'-1)$, and
\begin{align*}
e^{2\tilde{\alpha}p't} \mathbb{E}\left\|J(t)\right\|^{2p'}
    \leq&e^{2\tilde{\alpha}p'[t]} \mathbb{E}\left\|J([t])\right\|^{2p'}+\tilde{\beta}\mathbb{E}\int_{[t]}^te^{2\tilde{\alpha}p's} \left\|J([t]) \right\|^{2p'} ds\\
&+C\int_{ [t]}^te^{2\tilde{\alpha}p's}\mathbb{E}\left((1+2\|X^{i,\eta}(s)\|^q)^{p'}\|H(s)\|^{2p'}\|I(s)\|^{2p'}\right) ds\\
&+C\int_{ [t]}^te^{2\tilde{\alpha}p's}\mathbb{E}\left(\|H(s)\|^{2p'}\|I([t])\|^{2p'} \right)ds\\
&+C\int_{ [t]}^te^{2\tilde{\alpha}p's}\mathbb{E}\left((1+2\|X^{i,\eta}(s)\|^q)^{p'}\|H([t])\|^{2p'}\|I(s)\|^{2p'}\right) ds\\
&+C\int_{ [t]}^te^{2\tilde{\alpha}p's}\mathbb{E}\left(\|H([t])\|^{2p'}\|I([t])\|^{2p'}\right) ds,
\end{align*}
where $C$ is independent of $t$. Similarly to the case $u\geq [t]$, we get
\begin{align*}
e^{2\tilde{\alpha}p't} \mathbb{E}\left\|J(t)\right\|^{2p'}
    \leq&e^{2\tilde{\alpha}p'[t]} \mathbb{E}\left\|J([t])\right\|^{2p'}+\tilde{\beta}\mathbb{E}\int_{[t]}^te^{2\tilde{\alpha}p's} \left\|J([t]) \right\|^{2p'} ds\\
&+Ce^{-\nu_1([t]-i)p'/p}\left(1+\mathbb{E}\left\|\mathcal{D}_u\eta\right\|^{2p'}\right)\|\xi\|^{2p'}\int_{ [t]}^te^{2\tilde{\alpha}p's}ds\\
\leq&\left(\frac{\tilde\beta}{2\tilde{\alpha}p'}+\left(1-\frac{\tilde\beta}{2\tilde{\alpha}p'}\right)e^{-2\tilde{\alpha}p'\{t\}}\right)e^{2\tilde{\alpha}p't} \mathbb{E}\left\|J([t])\right\|^{2p'}\\
&+Ce^{-\nu_1([t]-i)p'/p}\left(1+\mathbb{E}\left\|\mathcal{D}_u\eta\right\|^{2p'}\right)\|\xi\|^{2p'}\left(e^{2\tilde{\alpha}p't}-e^{2\tilde{\alpha}p'[t]}\right),
\end{align*}
which implies
\begin{equation*}
\begin{split}
\mathbb{E}\left\|J(t)\right\|^{'}
\leq&\tilde{r}(\{t\}) \mathbb{E}\left\|J([t])\right\|^{2p'}
+Ce^{-\nu_1([t]-i)p'/p}\left(1+\mathbb{E}\left\|\mathcal{D}_u\eta\right\|^{2p'}\right)\|\xi\|^{2p'},
\end{split}
\end{equation*}
where $\tilde{r}(\{t\})=\frac{\tilde\beta}{2\tilde{\alpha}p'}+\left(1-\frac{\tilde\beta}{2\tilde{\alpha}p'}\right)e^{-2\tilde{\alpha}p'\{t\}}$. Since $\lambda_1-\lambda_2-1-2\lambda_3>4\lambda_3(p-1)$, we have $2\tilde{\alpha}p'>\tilde{\beta}$ and $0<\tilde{r}(\{t\})<1$. Similarly to the proof of  Lemma \ref{DX-uniform bounded}, we obtain
\begin{align*}
\mathbb{E}\left\|J(t)\right\|^{2p'}
\leq&\tilde{r}(\{t\}) \mathbb{E}\left\|J([t])\right\|^{2p'}
+Ce^{-\nu_1([t]-i)p'/p}\left(1+\mathbb{E}\left\|\mathcal{D}_u\eta\right\|^{2p'}\right)\|\xi\|^{2p'}\\
\leq&\tilde{r}(\{t\})\tilde{r}(1)^{[t]-[u]-1} \mathbb{E}\left\|J([u]+1)\right\|^{2p'}\\
&+C\left(1+\mathbb{E}\left\|\mathcal{D}_u\eta\right\|^{2p'}\right)\|\xi\|^{2p'}e^{-\nu_1([t]-i)p'/p}\cdot\left(\sum_{j=1}^{[t]-[u]-1}\tilde{r}(1)^{j-1}e^{\nu_1jp'/p}\right)\\
&+C\left(1+\mathbb{E}\left\|\mathcal{D}_u\eta\right\|^{2p'}\right)\|\xi\|^{2p'}e^{-\nu_1([t]-i)p'/p}.
\end{align*}
Without loss of generality, we assume $\tilde{r}(1)e^{\nu_1p'/p}<1$, then
\begin{align*}
\mathbb{E}\left\|J(t)\right\|^{2p'}
\leq&\tilde{r}(\{t\})\tilde{r}(1)^{[t]-[u]-1} \mathbb{E}\left\|J([u]+1)\right\|^{2p'}\\
&+C\left(1+\mathbb{E}\left\|\mathcal{D}_u\eta\right\|^{2p'}\right)\|\xi\|^{2p'}e^{-\nu_1([t]-i-1)p'/p}\\
\leq&\frac{1}{\tilde{r}(1)}e^{(t-[u]-1)\log \tilde{r}(1)}
e^{-2\bar{\alpha}p'([u]+1-u)}\mathbb{E}\left\|J(u)\right\|^{2p'}\\
&+Ce^{-\nu_1([t]-i-1)p'/p}\left(1+\mathbb{E}\left\|\mathcal{D}_u\eta\right\|^{2p'}\right)\|\xi\|^{2p'}\\
\leq&Ce^{-\nu_3(t-u)}\mathbb{E}\left\|J(u)\right\|^{2p'}+Ce^{-\nu_1([t]-i)p'/p}\left(1+\mathbb{E}\left\|\mathcal{D}_u\eta\right\|^{2p'}\right)\|\xi\|^{2p'},
\end{align*}
where $\nu_3=\min\{-\log \tilde {r}(1),2\tilde{\alpha}p'\}$. By the estimates and the uniform boundedness of $X(s)$, we have
$$\mathbb{E}\|J(u)\|^{2p'}\leq C\|\xi\|^{2p'}.$$
Therefore
\begin{equation*}
\begin{split}
\mathbb{E}\left\|J(t)\right\|^{2p'}
\leq&Ce^{-\nu_3(t-u)}\|\xi\|^{2p'}+Ce^{-\nu_1([t]-i)p'/p}\left(1+\mathbb{E}\left\|\mathcal{D}_u\eta\right\|^{2p'}\right)\|\xi\|^{2p'}.
\end{split}
\end{equation*}

\textbf{Case 3.} If $u<i$, then similarly to Case 2, we have
\begin{equation*}
\begin{split}
\mathbb{E}\left\|J(t)\right\|^{2p'}
\leq&Ce^{-\nu_3(t-i)}\|\xi\|^{2p'}+Ce^{-\nu_1([t]-i)p'/p}\left(1+\mathbb{E}\left\|\mathcal{D}_u\eta\right\|^{2p'}\right)\|\xi\|^{2p'}.
\end{split}
\end{equation*}
The proof is completed.
 \end{proof}

 If $f$ and $g$ are continuously differentiable with order 3, then the following lemma holds.
 \begin{lemma}\label{DuDsDX-uniform bounded}
Let $f$ and $g$ have continuous partial derivatives up to order 3 and conditions in Lemma \ref{uniform-bounded-2p} hold with $p\geq4$ . Suppose also that Assumptions \ref{f-polynomial}-\ref{higher derivatives} hold and $\mathbb{E}\left\|\mathcal{D}_u\eta\right\|^{2p}<\infty$, then
\begin{equation}\label{DuDsX-2p bounded}
\begin{split}
\mathbb{E}\left\|\mathcal{D}_w\mathcal{D}_uX^{i,\eta}(t)\xi\right\|^{2p'}\leq &Ce^{-\nu_3(t-w\vee u\vee i)}\mathbb{E}\|\mathcal{D}_w\mathcal{D}_u\eta\|^{2p'}\\
&+Ce^{-\nu_2([t]-w\vee u\vee i)p'/p}\left(1+\mathbb{E}\left\|\mathcal{D}_u\eta\right\|^{2p'}+\mathbb{E}\left\|\mathcal{D}_w\eta\right\|^{2p'}\right)
\end{split}
\end{equation}
and
\begin{equation}\label{DuDsDX-2p bounded}
\begin{split}
\mathbb{E}\left\|\mathcal{D}_w\mathcal{D}_uDX^{i,\eta}(t)\xi\right\|^{2p'} \leq& Ce^{-\nu_3(t-w\vee u\vee i)}\mathbb{E}\|\mathcal{D}_w\mathcal{D}_u\eta\|^{2p'}\\
&+Ce^{-\nu_2([t]-w\vee u\vee i)p'/p}\left(1+\mathbb{E}\left\|\mathcal{D}_u\eta\right\|^{2p'}+\mathbb{E}\left\|\mathcal{D}_w\eta\right\|^{2p'}\right)
\end{split}
\end{equation}
for any $1\leq p'\leq \min\{\frac{p}{4},\frac{p}{q}\}$, $\xi\in\mathbb{R}^d$.
 \end{lemma}

 \subsection{Errors of $\pi$ and $\pi^\delta$}
Let us first derive the weak error of $X(k)$ and $Y_k$.
 \begin{theorem}\label{weak error}
 Let  conditions in Lemma \ref{uniform-bounded-2p} with $p\geq4$ hold. Suppose also that $\phi\in C_b^3$ and Assumptions \ref{f-polynomial}-\ref{higher derivatives} hold, then  there exist $C>0$  independent of $\delta$ and $\delta_1>0$ such that
 \begin{equation*}
 \left|\mathbb{E}\phi(X^{0,x}(k))-\mathbb{E}\phi(Y^{0,x}_k)\right|\leq C\delta
 \end{equation*}
 for any $\delta\in(0,\delta_1)$.
 \end{theorem}
 \begin{proof}
 For any $\phi\in C_b^3$,
 \begin{equation*}
\begin{split}
&\left|\mathbb{E}\phi(X^{0,x}(k))-\mathbb{E}\phi(Y^{0,x}_k)\right|\\
=&\left|\sum_{i=0}^{k-1}\left(\mathbb{E}\phi(X^{i+1,X_{(i+1)m}^{0,x}}(k))-\mathbb{E}
\phi(X^{i,X_{im}^{0,x}}(k))\right)\right|\\
=&\left|\sum_{i=0}^{k-1}\left(\mathbb{E}\phi(X^{i+1,X_{(i+1)m}^{0,x}}(k))-\mathbb{E}
\phi(X^{i+1,X^{i,X_{im}^{0,x}}(i+1)}(k))\right)\right|\\
\leq&\sum_{i=0}^{k-1}\left|\mathbb{E}\phi(X^{i+1,X_{(i+1)m}^{0,x}}(k))-\mathbb{E}
\phi(X^{i+1,X^{i,X_{im}^{0,x}}(i+1)}(k))\right|\\
=&\sum_{i=0}^{k-1}\left|\mathbb{E}\!\!\int_0^1\!D(\phi\circ X)(k; i+1,\theta X_{(i+1)m}^{0,x}+(1-\theta)X^{i,X_{im}^{0,x}}(i+1))\left(X_{(i+1)m}^{0,x}-X^{i,X_{im}^{0,x}}(i+1)\right)d\theta\right|.
\end{split}
\end{equation*}
Denote $D(\phi\circ X)_{i+1}(k,\theta)=D(\phi\circ X)(k; i+1,\theta X_{(i+1)m}^{0,x}+(1-\theta)X^{i,X_{im}^{0,x}}(i+1))$. From
 \begin{align*}
&X_{(i+1)m}^{0,x}-X^{i,X_{im}^{0,x}}(i+1)
=X_{(i+1)m}^{i,X_{im}^{0,x}}-X^{i,X_{im}^{0,x}}(i+1)\\
=&\delta\sum_{l=0}^{m-1}f(X_{im+l+1},X_{im}^{0,x})+\sum_{l=0}^{m-1}g(X_{im+l},X_{im}^{0,x})\Delta_{im+l}\\
&-\int_i^{i+1}f(X^{i,X_{im}^{0,x}}(s),X_{im}^{0,x})ds-\int_i^{i+1}g(X^{i,X_{im}^{0,x}}(s),X_{im}^{0,x})dB(s)\\
=&\sum_{l=0}^{m-1}\int_{t_{im+l}}^{t_{im+l+1}}\left(f(X_{im+l+1},X_{im}^{0,x})-f(X^{i,X_{im}^{0,x}}(s),X_{im}^{0,x})\right)ds\\
&+\sum_{l=0}^{m-1}\int_{t_{im+l}}^{t_{im+l+1}}\left(g(X_{im+l},X_{im}^{0,x})-g(X^{i,X_{im}^{0,x}}(s),X_{im}^{0,x})\right)dB(s)\\
=&\sum_{l=0}^{m-1}\int_{t_{im+l}}^{t_{im+l+1}}\left(f(X_{im+l+1},X_{im}^{0,x})-f(X_{im+l},X_{im}^{0,x})\right)ds\\
&+\sum_{l=0}^{m-1}\int_{t_{im+l}}^{t_{im+l+1}}\left(f(X_{im+l},X_{im}^{0,x})-f(X^{i,X_{im}^{0,x}}(s),X_{im}^{0,x})\right)ds\\
&+\sum_{l=0}^{m-1}\int_{t_{im+l}}^{t_{im+l+1}}\left(g(X_{im+l},X_{im}^{0,x})-g(X^{i,X_{im}^{0,x}}(s),X_{im}^{0,x})\right)dB(s),
\end{align*}
it follows that
\begin{align*}
&\left|\mathbb{E}\phi(X^{0,x}(k))-\mathbb{E}\phi(Y^{0,x}_k)\right|\\
\leq&\sum_{i=0}^{k-1}\sum_{l=0}^{m-1}\left|\mathbb{E}\!\!\int_0^1\!D(\phi\circ X)_{i+1}(k,\theta)\cdot\int_{t_{im+l}}^{t_{im+l+1}}\left(f(X_{im+l+1},X_{im}^{0,x})-f(X_{im+l},X_{im}^{0,x})\right)dsd\theta\right|\\
&+\sum_{i=0}^{k-1}\sum_{l=0}^{m-1}\left|\mathbb{E}\!\!\int_0^1\!D(\phi\circ X)_{i+1}(k,\theta)\cdot\int_{t_{im+l}}^{t_{im+l+1}}\left(f(X^{i,X_{im}^{0,x}}(s),X_{im}^{0,x})-f(X_{im+l},X_{im}^{0,x})\right)dsd\theta\right|\\
&+\sum_{i=0}^{k-1}\sum_{l=0}^{m-1}\left|\mathbb{E}\!\!\int_0^1\!D(\phi\circ X)_{i+1}(k,\theta)\cdot\int_{t_{im+l}}^{t_{im+l+1}}\left(g(X^{i,X_{im}^{0,x}}(s),X_{im}^{0,x})-g(X_{im+l},X_{im}^{0,x})\right)dB(s)d\theta\right|\\
=&:\sum_{i=0}^{k-1}\sum_{l=0}^{m-1}\left(I_1+I_2+I_3\right).
\end{align*}
Denote $X_{im+l+\tau}:=\tau X_{im+l+1}+(1-\tau)X_{im+l}$. Then the estimate of  $I_1$ is
\begin{equation*}
\begin{split}
I_1=&\delta\left|\mathbb{E}\!\!\int_0^1\!\!\int_0^1D(\phi\circ X)_{i+1}(k,\theta)\cdot\frac{\partial f}{\partial x}(X_{im+l+\tau},X_{im}^{0,x})\left(X_{im+l+1}-X_{im+l}\right)d\tau d\theta\right|\\
\leq&\delta^2\!\!\int_0^1\!\!\int_0^1\bigg|\mathbb{E}\left(D(\phi\circ X)_{i+1}(k,\theta)\cdot\frac{\partial f}{\partial x}(X_{im+l+\tau},X_{im}^{0,x}) f(X_{im+l+1},X^{0,x}_{im})\right)\bigg|d\tau d\theta\\
&+\delta\int_0^1\!\!\int_0^1\bigg|\mathbb{E}\left(D(\phi\circ X)_{i+1}(k,\theta)\cdot\frac{\partial f}{\partial x}(X_{im+l+\tau},X_{im}^{0,x}) g(X_{im+l},X^{0,x}_{im})\Delta B_{im+l}\right)\bigg|d\tau d\theta\\
=&:I_{11}+I_{12}.
\end{split}
\end{equation*}
The chain rule of the Fr\'echet derivative leads to
$$D(\phi\circ X)_{i+1}(k,\theta)=D\phi(X^{i+1,\theta X^{0,x}_{(i+1)m}+(1-\theta) X^{i,X^{0,x}_{im}}(i+1)}(k))\cdot DX^{i+1,\theta X^{0,x}_{(i+1)m}+(1-\theta) X^{i,X^{0,x}_{im}}(i+1)}(k).$$
From $\phi\in C_b^1$, H\"{o}lder inequality and Lemma \ref{DX-uniform bounded}, it follows
\begin{align*}
&I_{11}\leq C\delta^2\!\!\int_0^1\!\!\int_0^1\bigg\|\mathbb{E}\left(\!DX^{i+1,\theta X^{0,x}_{(i+1)m}+(1-\theta) X^{i,X^{0,x}_{im}}(i+1)}(k)\cdot\frac{\partial f}{\partial x}(X_{im+l+\tau},X_{im}^{0,x}) f(X_{im+l+1},X^{0,x}_{im})\!\right)\bigg\|d\tau d\theta\\
\leq&C\delta^2\!\!\int_0^1\!\!\int_0^1\left(\mathbb{E}\left\|\!DX^{i+1,\theta X^{0,x}_{(i+1)m}+(1-\theta) X^{i,X^{0,x}_{im}}(i+1)}(k)\cdot\frac{\partial f}{\partial x}(X_{im+l+\tau},X_{im}^{0,x}) f(X_{im+l+1},X^{0,x}_{im})\!\right\|^2\right)^{\frac{1}{2}}d\tau d\theta\\
\leq&Ce^{-\frac{1}{2}\nu_1(k-i-1)}\delta^2\int_0^1\left(\mathbb{E}\left\|\frac{\partial f}{\partial x}(X_{im+l+\tau},X_{im}^{0,x}) f(X_{im+l+1},X^{0,x}_{im})\!\right\|^2\right)^{\frac{1}{2}}d\tau \\
\leq&Ce^{-\frac{1}{2}\nu_1(k-i-1)}\delta^2\int_0^1\bigg(\mathbb{E}\bigg(2Kd(1+\|X_{im+l+\tau}\|^q)\|X_{im+l+\tau}\|^2\left\|f(X_{im+l+1},X^{0,x}_{im})\right\|^2\\
&\qquad\qquad+4Kd\|X_{im}^{0,x}\|^2\left\| f(X_{im+l+1},X^{0,x}_{im})\right\|^2+4\left\|\frac{\partial f}{\partial x}(0,0)f(X_{im+l+1},X^{0,x}_{im})\right\|^2\bigg)\bigg)^{\frac{1}{2}}d\tau.
\end{align*}
By the $L^{2p}$ ($p\geq 4$) uniform boundedness of the numerical solutions and Assumptions \ref{f-polynomial} and \ref{higher derivatives},
\begin{equation*}
\begin{split}
I_{11}
\leq Ce^{-\frac{1}{2}\nu_1(k-i-1)}\delta^2.
\end{split}
\end{equation*}
For $I_{12}$, the duality formula of Malliavin derivative \cite[P. 43]{Nualart2006} leads to
\begin{equation*}
\begin{split}
I_{12}=&\delta\int_0^1\!\!\int_0^1\bigg\|\mathbb{E}\int_{t_{im+l}}^{t_{im+l+1}}\mathcal{D}_uD(\phi\circ X)_{i+1}(k,\theta)\cdot\frac{\partial f}{\partial x}(X_{im+l+\tau},X_{im}^{0,x}) g(X_{im+l},X^{0,x}_{im})du\bigg\|d\tau d\theta\\
\leq&\delta\int_0^1\!\!\int_0^1\!\!\int_{t_{im+l}}^{t_{im+l+1}}\bigg\|\mathbb{E}\left(\mathcal{D}_uD(\phi\circ X)_{i+1}(k,\theta)\cdot\frac{\partial f}{\partial x}(X_{im+l+\tau},X_{im}^{0,x}) g(X_{im+l},X^{0,x}_{im})\right)\bigg\|dud\tau d\theta
\end{split}
\end{equation*}
 Then by the chain rule for Fr\'echet derivatives and the chain rule as well as the product rule for  Malliavin derivatives \cite[P. 37]{Nualart2006}, we obtain
\begin{equation*}
\begin{split}
&\mathcal{D}_uD(\phi\circ X)_{i+1}(k,\theta)\xi\\
=&\mathcal{D}_u\left(D\phi( X^{i+1,\theta X_{(i+1)m}^{0,x}+(1-\theta)X^{i,X^{0,x}_{im}}(i+1)}(k))\cdot DX^{i+1,\theta X_{(i+1)m}^{0,x}+(1-\theta)X^{i,X^{0,x}_{im}}(i+1)}(k)\xi\right)\\
=&\left(\mathcal{D}_uX^{i+1,\theta X_{(i+1)m}^{0,x}+(1-\theta)X^{i,X^{0,x}_{im}}(i+1)}(k)\right)^T\cdot D^2\phi( X^{i+1,\theta X_{(i+1)m}^{0,x}+(1-\theta)X^{i,X^{0,x}_{im}}(i+1)}(k))\\
&\qquad\qquad\qquad\qquad\qquad\qquad\qquad\qquad\qquad\cdot DX^{i+1,\theta X_{(i+1)m}^{0,x}+(1-\theta)X^{i,X^{0,x}_{im}}(i+1)}(k)\xi\\
&+ D\phi( X^{i+1,\theta X_{(i+1)m}^{0,x}+(1-\theta)X^{i,X^{0,x}_{im}}(i+1)}(k))\cdot\mathcal{D}_uDX^{i+1,\theta X_{(i+1)m}^{0,x}+(1-\theta)X^{i,X^{0,x}_{im}}(i+1)}(k)\xi.
\end{split}
\end{equation*}
Thus,  Lemmas  \ref{DX-uniform bounded}-\ref{DuDX-uniform bounded}, Assumptions \ref{assumption2}, \ref{f-polynomial} and $\phi\in C_b^2$ lead to
\begin{align*}
I_{12}
\leq&\delta\!\!\!\int_0^1\!\!\!\int_0^1\!\!\!\int_{t_{im+l}}^{t_{im+l+1}}\!\bigg\|\mathbb{E}\bigg(\!\!\left(\mathcal{D}_uX^{i+1,\theta X_{(i+1)m}^{0,x}+(1-\theta)X^{i,X^{0,x}_{im}}(i+1)}(k)\!\!\right)^T\!\!\cdot \!\!D^2\phi( X^{i+1,\theta X_{(i+1)m}^{0,x}+(1-\theta)X^{i,X^{0,x}_{im}}(i+1)}(k))\\
&\quad\qquad\qquad\cdot DX^{i+1,\theta X_{(i+1)m}^{0,x}+(1-\theta)X^{i,X^{0,x}_{im}}(i+1)}(k)
\cdot\frac{\partial f}{\partial x}(X_{im+l+\tau},X_{im}^{0,x}) g(X_{im+l},X^{0,x}_{im})\bigg)\bigg\|dud\tau d\theta\\
&+\delta\!\!\!\int_0^1\!\!\!\int_0^1\!\!\!\int_{t_{im+l}}^{t_{im+l+1}}\!\bigg\|\mathbb{E}\bigg(\!\!D\phi( X^{i+1,\theta X_{(i+1)m}^{0,x}+(1-\theta)X^{i,X^{0,x}_{im}}(i+1)}(k))\cdot\mathcal{D}_uDX^{i+1,\theta X_{(i+1)m}^{0,x}+(1-\theta)X^{i,X^{0,x}_{im}}(i+1)}(k)\\
&\quad\qquad\qquad\qquad\qquad\qquad\qquad\qquad\qquad\qquad\qquad\cdot\frac{\partial f}{\partial x}(X_{im+l+\tau},X_{im}^{0,x}) g(X_{im+l},X^{0,x}_{im})\bigg)\bigg\|dud\tau d\theta\\
\leq&C\delta\!\!\!\int_0^1\!\!\!\int_0^1\!\!\!\int_{t_{im+l}}^{t_{im+l+1}}\!\mathbb{E}\bigg\|\!\!\left(\mathcal{D}_uX^{i+1,\theta X_{(i+1)m}^{0,x}+(1-\theta)X^{i,X^{0,x}_{im}}(i+1)}(k)\!\!\right)^T\!\!\cdot DX^{i+1,\theta X_{(i+1)m}^{0,x}+(1-\theta)X^{i,X^{0,x}_{im}}(i+1)}(k)\\
&\qquad\qquad\qquad\qquad\qquad\qquad\qquad\qquad\qquad\qquad\cdot\frac{\partial f}{\partial x}(X_{im+l+\tau},X_{im}^{0,x}) g(X_{im+l},X^{0,x}_{im})\bigg\|dud\tau d\theta\\
&+C\delta\!\!\!\int_0^1\!\!\!\int_0^1\!\!\!\int_{t_{im+l}}^{t_{im+l+1}}\!\bigg\|\mathbb{E}\bigg(\mathcal{D}_uDX^{i+1,\theta X_{(i+1)m}^{0,x}+(1-\theta)X^{i,X^{0,x}_{im}}(i+1)}(k)\\
&\quad\qquad\qquad\qquad\qquad\qquad\qquad\qquad\qquad\qquad\qquad\cdot\frac{\partial f}{\partial x}(X_{im+l+\tau},X_{im}^{0,x}) g(X_{im+l},X^{0,x}_{im})\bigg)\bigg\|dud\tau d\theta\\
\leq&C\delta\!\!\!\int_0^1\!\!\!\int_0^1\!\!\!\int_{t_{im+l}}^{t_{im+l+1}}\!\left(\mathbb{E}\left\|\mathcal{D}_uX^{i+1,\theta X_{(i+1)m}^{0,x}+(1-\theta)X^{i,X^{0,x}_{im}}(i+1)}(k)\right\|^2\right)^{\frac{1}{2}}\\
&\cdot \left(\mathbb{E}\left\|DX^{i+1,\theta X_{(i+1)m}^{0,x}+(1-\theta)X^{i,X^{0,x}_{im}}(i+1)}(k)
\cdot\frac{\partial f}{\partial x}(X_{im+l+\tau},X_{im}^{0,x}) g(X_{im+l},X^{0,x}_{im})\right\|^2\right)^{\frac{1}{2}}dud\tau d\theta\\
&+C\delta\!\!\!\int_0^1\!\!\!\int_0^1\!\!\!\int_{t_{im+l}}^{t_{im+l+1}}\!\bigg(\mathbb{E}\bigg\|\mathcal{D}_uDX^{i+1,\theta X_{(i+1)m}^{0,x}+(1-\theta)X^{i,X^{0,x}_{im}}(i+1)}(k)\\
&\quad\qquad\qquad\qquad\qquad\qquad\qquad\qquad\qquad\qquad\cdot\frac{\partial f}{\partial x}(X_{im+l+\tau},X_{im}^{0,x}) g(X_{im+l},X^{0,x}_{im})\bigg\|^2\bigg)^{\frac{1}{2}}dud\tau d\theta\\
\leq&Ce^{-\frac{1}{2}(\nu_1+\nu_2)(k-i-1)}\delta^2+\left(Ce^{-\frac{1}{2}\nu_3(k-i-1)}+Ce^{-\frac{1}{2p}\nu_1(k-i-1)}\right)\delta^2\\
\leq&Ce^{-\frac{1}{2}\nu_4(k-i-1)}\delta^2,
\end{align*}
where $\nu_4=\min\{\nu_3,\frac{1}{p}\nu_1\}$ and $u\leq i+1$ are used.

Next, we estimate $I_2$.  It\^o's formula implies
\begin{align*}
I_2=&\left|\mathbb{E}\!\!\int_0^1\int_{t_{im+l}}^{t_{im+l+1}}\!D(\phi\circ X)_{i+1}(k,\theta)\cdot\left(f(X^{i,X_{im}^{0,x}}(s),X_{im}^{0,x})-f(X_{im+l},X_{im}^{0,x})\right)dsd\theta\right|\\
\leq&\left|\mathbb{E}\!\!\int_0^1\int_{t_{im+l}}^{t_{im+l+1}}\!D(\phi\circ X)_{i+1}(k,\theta)\cdot\left(f(X^{i,X_{im}^{0,x}}(s),X_{im}^{0,x})-f(X^{i,X_{im}^{0,x}}(t_{im+l}),X_{im}^{0,x})
\right)dsd\theta\right|\\
&+\delta\left|\mathbb{E}\!\!\int_0^1\!D(\phi\circ X)_{i+1}(k,\theta)\cdot\left(f(X^{i,X_{im}^{0,x}}(t_{im+l}),X_{im}^{0,x})
-f(X_{im+l},X_{im}^{0,x})\right)d\theta\right|\\
\leq&\left|\mathbb{E}\!\!\int_0^1\int_{t_{im+l}}^{t_{im+l+1}}\!D(\phi\circ X)_{i+1}(k,\theta)\cdot\int_{t_{im+l}}^s\frac{\partial f}{\partial x}(X^{i,X_{im}^{0,x}}(u),X_{im}^{0,x})f(X^{i,X_{im}^{0,x}}(u),X_{im}^{0,x})dudsd\theta\right|\\
&+\frac{1}{2}\bigg|\mathbb{E}\!\!\int_0^1\int_{t_{im+l}}^{t_{im+l+1}}\!D(\phi\circ X)_{i+1}(k,\theta)\cdot\int_{t_{im+l}}^s\frac{\partial^2 f}{\partial x^2}(X^{i,X_{im}^{0,x}}(u),X_{im}^{0,x})\\
&\qquad\qquad\qquad\qquad\qquad\qquad\qquad\qquad\qquad\qquad(g(X^{i,X_{im}^{0,x}}(u),X_{im}^{0,x}),g(X^{i,X_{im}^{0,x}}(u),X_{im}^{0,x}))dudsd\theta\bigg|\\
&+\left|\mathbb{E}\!\!\int_0^1\int_{t_{im+l}}^{t_{im+l+1}}\!D(\phi\circ X)_{i+1}(k,\theta)\cdot\int_{im+l}^s\frac{\partial f}{\partial x}(X^{i,X_{im}^{0,x}}(u),X_{im}^{0,x})g(X^{i,X_{im}^{0,x}}(u),X_{im}^{0,x})dB(u)dsd\theta\right|\\
&+\delta\left|\mathbb{E}\!\!\int_0^1\!D(\phi\circ X)_{i+1}(k,\theta)\cdot\left(f(X^{i,X_{im}^{0,x}}(t_{im+l}),X_{im}^{0,x})
-f(X_{im+l},X_{im}^{0,x})\right)d\theta\right|\\
=&:I_{21}+I_{22}+I_{23}+I_{24}.
\end{align*}
By  Assumptions \ref{assumption2} and \ref{f-polynomial}, the $L_{2p}$ $(p\geq4)$ uniform boundedness of $X^{i,\eta}(t)$ and the chain rule for Fr\'echet derivatives, and applying H\"{o}lder's inequality and Lemma \ref{DX-uniform bounded}, we have
\begin{equation*}
\begin{split}
I_{21}
\leq&\!\!\int_0^1\!\!\int_{t_{im+l}}^{t_{im+l+1}}\!\!\!\int_{t_{im+l}}^s\!\left(\mathbb{E}\!\left|D(\phi\circ X)_{i+1}(k,\theta)\cdot\frac{\partial f}{\partial x}(X^{i,X_{im}^{0,x}}(u),X_{im}^{0,x})f(X^{i,X_{im}^{0,x}}(u),X_{im}^{0,x})\right|^2\right)^{\frac{1}{2}}dudsd\theta\\
\leq&Ce^{-\frac{1}{2}\nu_1(k-i-1)}\delta^2,
\end{split}
\end{equation*}
and
\begin{equation*}
\begin{split}
I_{22}\leq&\frac{1}{2}\!\int_0^1\!\!\int_{t_{im+l}}^{t_{im+l+1}}\!\!\int_{t_{im+l}}^s\bigg(\mathbb{E}\bigg\|D(\phi\circ X)_{i+1}(k,\theta)\cdot\frac{\partial^2 f}{\partial x^2}(X^{i,X_{im}^{0,x}}(u),X_{im}^{0,x})\\
&\quad\quad\qquad\qquad\qquad\qquad\qquad\qquad\qquad\qquad(g(X^{i,X_{im}^{0,x}}(u),X_{im}^{0,x}),g(X^{i,X_{im}^{0,x}}(u),X_{im}^{0,x}))\bigg\|^2\bigg)^{\frac{1}{2}}dudsd\theta\\
\leq&Ce^{-\frac{1}{2}\nu_1(k-i-1)}\delta^2.
\end{split}
\end{equation*}
For $I_{23}$, the similar estimate of $I_{12}$ implies
\begin{equation*}
\begin{split}
I_{23}\leq Ce^{-\frac{1}{2}\nu_4(k-i-1)}\delta^2.
\end{split}
\end{equation*}
For $I_{24}$, similar as above, there exists $\nu_5>0$ such that
\begin{equation*}
\begin{split}
I_{24}\leq Ce^{-\nu_5(k-i-1)}\delta^2.
\end{split}
\end{equation*}

The estimate of $I_3$ is as follows.  It\^o's formula leads to
\begin{align*}
I_3=&\left|\mathbb{E}\!\!\int_0^1\!D(\phi\circ X)_{i+1}(k,\theta)\cdot\int_{t_{im+l}}^{t_{im+l+1}}\left(g(X^{i,X_{im}^{0,x}}(s),X_{im}^{0,x})-g(X_{im+l},X_{im}^{0,x})\right)dB(s)d\theta\right|\\
=&\left|\mathbb{E}\!\!\int_0^1\!D(\phi\circ X)_{i+1}(k,\theta)\cdot\int_{t_{im+l}}^{t_{im+l+1}}\!\!\int_{t_{im+l}}^s\frac{\partial g}{\partial x}(X^{i,X_{im}^{0,x}}(u),X_{im}^{0,x})f(X^{i,X_{im}^{0,x}}(u),X_{im}^{0,x})dudB(s)d\theta\right|\\
&+\frac{1}{2}\bigg|\mathbb{E}\!\!\int_0^1\!D(\phi\circ X)_{i+1}(k,\theta)\cdot\int_{t_{im+l}}^{t_{im+l+1}}\!\!\int_{t_{im+l}}^s\frac{\partial^2 g}{\partial x^2}(X^{i,X_{im}^{0,x}}(u),X_{im}^{0,x})\\
&\quad\quad\qquad\qquad\qquad\qquad\qquad\qquad\qquad\qquad(g(X^{i,X_{im}^{0,x}}(u),X_{im}^{0,x}),g(X^{i,X_{im}^{0,x}}(u),X_{im}^{0,x}))dudB(s)d\theta\bigg|\\
&+\left|\mathbb{E}\!\!\int_0^1\!\!D(\phi\circ X)_{i+1}(k,\theta)\!\cdot\!\!\int_{t_{im+l}}^{t_{im+l+1}}\!\!\!\int_{t_{im+l}}^s\frac{\partial g}{\partial x}(X^{i,X_{im}^{0,x}}(u),X_{im}^{0,x})g(X^{i,X_{im}^{0,x}}(u),X_{im}^{0,x})dB(u)dB(s)d\theta\right|\\
=&:I_{31}+I_{32}+I_{33}.
\end{align*}
The estimates of $I_{31}$ and $I_{32}$ are similar to that  of $I_{12}$,
\begin{align*}
I_{31}=&\left|\int_0^1\!\!\int_{t_{im+l}}^{t_{im+l+1}}\!\!\mathbb{E}\left(\mathcal{D}_sD(\phi\circ X)_{i+1}(k,\theta)\!\cdot\!\!\int_{t_{im+l}}^s\frac{\partial g}{\partial x}(X^{i,X_{im}^{0,x}}(u),X_{im}^{0,x})f(X^{i,X_{im}^{0,x}}(v),X_{im}^{0,x})du\right)dsd\theta\right|\\
\leq&\int_0^1\!\!\int_{t_{im+l}}^{t_{im+l+1}}\!\!\int_{t_{im+l}}^s\!\!\mathbb{E}\left|\mathcal{D}_sD(\phi\circ X)_{i+1}(k,\theta)\!\cdot\frac{\partial g}{\partial x}(X^{i,X_{im}^{0,x}}(u),X_{im}^{0,x})f(X^{i,X_{im}^{0,x}}(u),X_{im}^{0,x})\right|dudsd\theta\\
\leq&Ce^{-\frac{1}{2}\nu_4(k-i-1)}\delta^2
\end{align*}
and
\begin{align*}
I_{32}=&
\frac{1}{2}\bigg|\!\int_0^1\!\!\int_{t_{im+l}}^{t_{im+l+1}}\!\mathbb{E}\bigg(\mathcal{D}_sD(\phi\circ X)_{i+1}(k,\theta)\cdot\!\!\int_{t_{im+l}}^s\frac{\partial^2 g}{\partial x^2}(X^{i,X_{im}^{0,x}}(u),X_{im}^{0,x})\\
&\quad\quad\qquad\qquad\qquad\qquad\qquad\qquad\qquad\qquad(g(X^{i,X_{im}^{0,x}}(u),X_{im}^{0,x}),g(X^{i,X_{im}^{0,x}}(u),X_{im}^{0,x}))\bigg)dudsd\theta\bigg|\\
\leq&\frac{1}{2}\!\int_0^1\!\!\int_{t_{im+l}}^{t_{im+l+1}}\!\!\int_{t_{im+l}}^s\mathbb{E}\bigg|\mathcal{D}_sD(\phi\circ X)_{i+1}(k,\theta)\cdot\frac{\partial^2 g}{\partial x^2}(X^{i,X_{im}^{0,x}}(u),X_{im}^{0,x})\\
&\quad\quad\qquad\qquad\qquad\qquad\qquad\qquad\qquad\qquad(g(X^{i,X_{im}^{0,x}}(u),X_{im}^{0,x}),g(X^{i,X_{im}^{0,x}}(u),X_{im}^{0,x}))\bigg|dudsd\theta\\
\leq&Ce^{-\frac{1}{2}\nu_4(k-i-1)}\delta^2.
\end{align*}
For $I_{33}$, we obtain
\begin{align*}
I_{33}=
&\left|\mathbb{E}\!\!\int_0^1\!\!D(\phi\circ X)_{i+1}(k,\theta)\!\cdot\!\!\int_{t_{im+l}}^{t_{im+l+1}}\!\!\!\int_{t_{im+l}}^s\frac{\partial g}{\partial x}(X^{i,X_{im}^{0,x}}(u),X_{im}^{0,x})g(X^{i,X_{im}^{0,x}}(u),X_{im}^{0,x})dB(u)dB(s)d\theta\right|\\
=&\left|\int_0^1\!\!\int_{t_{im+l}}^{t_{im+l+1}}\!\!\mathbb{E}\!\left(\mathcal{D}_sD(\phi\circ X)_{i+1}(k,\theta)\!\cdot\!\!\!\int_{t_{im+l}}^s\frac{\partial g}{\partial x}(X^{i,X_{im}^{0,x}}(u),X_{im}^{0,x})g(X^{i,X_{im}^{0,x}}(u),X_{im}^{0,x})dB(u)\!\right)dsd\theta\right|\\
=&\left|\int_0^1\!\!\int_{t_{im+l}}^{t_{im+l+1}}\!\!\!\int_{t_{im+l}}^s\mathbb{E}\!\left(\mathcal{D}_u\mathcal{D}_sD(\phi\circ X)_{i+1}(k,\theta)\!\cdot\frac{\partial g}{\partial x}(X^{i,X_{im}^{0,x}}(u),X_{im}^{0,x})g(X^{i,X_{im}^{0,x}}(u),X_{im}^{0,x})\!\right)dudsd\theta\right|\\
\leq&\int_0^1\!\!\int_{t_{im+l}}^{t_{im+l+1}}\!\!\!\int_{t_{im+l}}^s\mathbb{E}\left|\mathcal{D}_u\mathcal{D}_sD(\phi\circ X)_{i+1}(k,\theta)\!\cdot\frac{\partial g}{\partial x}(X^{i,X_{im}^{0,x}}(u),X_{im}^{0,x})g(X^{i,X_{im}^{0,x}}(u),X_{im}^{0,x})\right|dudsd\theta.
\end{align*}
Taking Malliavin derivative $\mathcal{D}_u$ on $\mathcal{D}_sD(\phi\circ X)_{i+1}(k,\theta)\xi$ yields
\begin{align*}
&\mathcal{D}_u\mathcal{D}_sD(\phi\circ X)_{i+1}(k,\theta)\xi\\
=&\mathcal{D}_u\bigg(\left(\mathcal{D}_sX^{i+1,\theta X_{(i+1)m}^{0,x}+(1-\theta)X^{i,X^{0,x}_{im}}(i+1)}(k)\right)^T\cdot D^2\phi( X^{i+1,\theta X_{(i+1)m}^{0,x}+(1-\theta)X^{i,X^{0,x}_{im}}(i+1)}(k))\\
&\qquad\qquad\qquad\qquad\qquad\qquad\qquad\qquad\qquad\cdot DX^{i+1,\theta X_{(i+1)m}^{0,x}+(1-\theta)X^{i,X^{0,x}_{im}}(i+1)}(k)\xi\\
&+ D\phi( X^{i+1,\theta X_{(i+1)m}^{0,x}+(1-\theta)X^{i,X^{0,x}_{im}}(i+1)}(k))\cdot\mathcal{D}_sDX^{i+1,\theta X_{(i+1)m}^{0,x}+(1-\theta)X^{i,X^{0,x}_{im}}(i+1)}(k)\xi\bigg)\\
=&\left(\mathcal{D}_u\mathcal{D}_sX^{i+1,\theta X_{(i+1)m}^{0,x}+(1-\theta)X^{i,X^{0,x}_{im}}(i+1)}(k)\right)^T\cdot D^2\phi( X^{i+1,\theta X_{(i+1)m}^{0,x}+(1-\theta)X^{i,X^{0,x}_{im}}(i+1)}(k))\\
&\qquad\qquad\qquad\qquad\qquad\qquad\qquad\qquad\qquad\cdot DX^{i+1,\theta X_{(i+1)m}^{0,x}+(1-\theta)X^{i,X^{0,x}_{im}}(i+1)}(k)\xi\\
&+D^3\phi( X^{i+1,\theta X_{(i+1)m}^{0,x}+(1-\theta)X^{i,X^{0,x}_{im}}(i+1)}(k))(DX^{i+1,\theta X_{(i+1)m}^{0,x}+(1-\theta)X^{i,X^{0,x}_{im}}(i+1)}(k)\xi,\\
&\qquad\mathcal{D}_sX^{i+1,\theta X_{(i+1)m}^{0,x}+(1-\theta)X^{i,X^{0,x}_{im}}(i+1)}(k),\mathcal{D}_uX^{i+1,\theta X_{(i+1)m}^{0,x}+(1-\theta)X^{i,X^{0,x}_{im}}(i+1)}(k) )\\
&+\left(\mathcal{D}_sX^{i+1,\theta X_{(i+1)m}^{0,x}+(1-\theta)X^{i,X^{0,x}_{im}}(i+1)}(k)\right)^T\cdot D^2\phi( X^{i+1,\theta X_{(i+1)m}^{0,x}+(1-\theta)X^{i,X^{0,x}_{im}}(i+1)}(k)\\
&\qquad\qquad\qquad\qquad\qquad\qquad\qquad\qquad\qquad\cdot \mathcal{D}_uDX^{i+1,\theta X_{(i+1)m}^{0,x}+(1-\theta)X^{i,X^{0,x}_{im}}(i+1)}(k)\xi\\
&+\left( \mathcal{D}_uX^{i+1,\theta X_{(i+1)m}^{0,x}+(1-\theta)X^{i,X^{0,x}_{im}}(i+1)}(k)\right)^T\cdot D^2\phi( X^{i+1,\theta X_{(i+1)m}^{0,x}+(1-\theta)X^{i,X^{0,x}_{im}}(i+1)}(k))\\
&\qquad\qquad\qquad\qquad\qquad\qquad\qquad\qquad\qquad\cdot\mathcal{D}_sDX^{i+1,\theta X_{(i+1)m}^{0,x}+(1-\theta)X^{i,X^{0,x}_{im}}(i+1)}(k)\xi\\
&+ D\phi( X^{i+1,\theta X_{(i+1)m}^{0,x}+(1-\theta)X^{i,X^{0,x}_{im}}(i+1)}(k))\cdot\mathcal{D}_u\mathcal{D}_sDX^{i+1,\theta X_{(i+1)m}^{0,x}+(1-\theta)X^{i,X^{0,x}_{im}}(i+1)}(k)\xi.
\end{align*}
By the estimates of Lemmas \ref{DX-uniform bounded}-\ref{DuDsDX-uniform bounded}, there exists $\nu_6>0$ such that
$$I_{33}\leq Ce^{-\nu_6(k-i-1)}\delta^2.$$
Combining the estimates of $I_1$, $I_2$ and $I_3$, we conclude that there exists $\nu>0$ such that
$$I_1+I_2+I_3\leq Ce^{-\nu(k-i-1)}\delta^2,$$
which implies that
\begin{equation*}
\begin{split}
\sum_{i=0}^{k-1}\sum_{l=0}^{m-1}\left(I_1+I_2+I_3\right)\leq C\delta\sum_{i=0}^{k-1}e^{-\nu(k-i-1)}=C\delta\frac{e^\nu-e^{-\nu(k-1)}}{e^\nu-1}\leq C\delta.
\end{split}
\end{equation*}
Here $C$ is independent of $k$ and $m\delta=1$ is used.
The proof is completed.
\end{proof}

Up to now, we have proven that the uniform weak convergence order of the BE method is 1. As a consequence, the convergence order between the invariant measures $\pi$ and $\pi^\delta$ is obtained based on the uniqueness of invariant measure, that is
\begin{align*}
&\left|\int_{\mathbb{R}^d}\phi(x)\pi(dx)-\int_{\mathbb{R}^d}\phi(x)\pi^{\delta}(dx)\right|\\
=&\left|\lim_{K\rightarrow \infty}\frac{1}{K}\left(\sum_{k=0}^{K-1}\mathbb{E}\phi(X(k))-\sum_{k=0}^{K-1}\mathbb{E}\phi(Y_k)\right)\right|\\
\leq&\lim_{K\rightarrow \infty}\frac{1}{K}\sum_{k=0}^{K-1}\left|\mathbb{E}\phi(X(k))-\mathbb{E}\phi(Y_k)\right|\\
\leq&C\delta.
\end{align*}
 \section{Numerical Simulations}

In this section, we present two examples to verify the theoretical results.

\textbf{Example 1.} Consider the following $1$-dimensional equation with additive noise
\begin{equation}\label{SDEPCA-linear}
\begin{cases}
dX(t)=(-\theta_1X(t)+\theta_2X([t]))dt+dB(t)\\
X(0)=x,
\end{cases}
\end{equation}
where $\theta_1>0$ and $\theta_2\in\mathbb{R}$. If $t\in[k,k+1)$, $k\in\mathbb{N}$, then the solution of \eqref{SDEPCA-linear} is
\begin{equation}\label{solution}
\begin{split}
X(t)&=X(k)\left(e^{-\theta_1(t-k)}+\frac{\theta_2}{\theta_1}\left(1-e^{-\theta_1(t-k)}\right)\right)+\int_{k}^te^{-\theta_1(t-s)}dB(s).
\end{split}
\end{equation}
It can be seen that the solution obeys Gaussian distribution. And the expectation of the solution is
\begin{equation}\label{expectation-X(t)}
\mathbb{E}X(t)=x\left(\frac{\theta_2}{\theta_1}+\left(1-\frac{\theta_2}{\theta_1}\right) e^{-\theta_1}\right)^k\left(\frac{\theta_2}{\theta_1}+\left(1-\frac{\theta_2}{\theta_1}\right) e^{-\theta_1(t-k)}\right).
\end{equation}
Let $\mu(1)=\frac{\theta_2}{\theta_1}+\left(1-\frac{\theta_2}{\theta_1}\right)e^{-\theta_1}$, $\mu(\{t\})=\frac{\theta_2}{\theta_1}+\left(1-\frac{\theta_2}{\theta_1}\right) e^{-\theta_1(t-[t])}$, $\sigma(1)=\frac{1}{2\theta_1}\left(1-e^{-2\theta_1}\right)$ and $\sigma(\{t\})=\frac{1}{2\theta_1}\left(1-e^{-2\theta_1(t-[t])}\right)$, where $\{t\}$ denotes the fractional part of $t$. Then the variance of the solution is
\begin{equation}\label{variance-X(t)}
Var(X(t))=\left(\frac{1-\mu(1)^{2k}}{1-\mu(1)^2}\sigma(1)\right)\mu(\{t\})^2+\sigma(\{t\}).
\end{equation}
Especially, the expectation and variance of $X(t)$ at the integral time $t=k$ are, respectively,
\begin{equation}\label{expectation-X(k)}
\mathbb{E}X(k)=x\mu(1)^k~\textrm{and}~ Var(X(k))=\frac{1-\mu(1)^{2k}}{1-\mu(1)^2}\sigma(1).
\end{equation}
The sufficient and necessary condition under which $X(k)$ may admit a stationary distribution is
$$ |\mu(1)|<1\Leftrightarrow-\frac{1+e^{-\theta_1}}{1-e^{-\theta_1}}\theta_1<\theta_2<\theta_1.$$

Firstly, we verify that the solution $\{X(t)\}_{t\geq0}$ does not admit a stationary distribution while the chain $\{X(k)\}_{k\in\mathbb{N}}$ does. Let the initial value $x=1$. Fig. \ref{Gassian} shows the expectations and variances of both $X(t)$ and $X(k)$ with three different parameters which satisfy $|\mu(1)|<1$. It can be seen that the variances of the solution $\{X(t)\}_{t\geq0}$ are not convergent as $t$ tends to infinity, though the expectations of $\{X(t)\}_{t\geq0}$ converge to zero. However, both the expectations and variances of the chain $\{X(k)\}_{k\in\mathbb{N}}$ (i.e. the solution of \eqref{SDEPCA-linear} at integer time $t=k$) converge as $k$ goes to infinity. This means that the chain $\{X(k)\}_{k\in\mathbb{N}}$ admits a stationary Gaussian distribution. Comparing Fig. \ref{Gassian} (a) with (b), we observe that the distribution of $\{X(k)\}_{k\in\mathbb{N}}$ converges more rapidly for larger $\theta_1$, which implies that  the convergence rate increases with the increases of dissipativity.

\begin{figure}[htbp]
	\centering
	\subfigure[$\theta_1=3$, $\theta_2=1$]{
		\begin{minipage}[t]{0.31\linewidth}
			\centering
			\includegraphics[scale=0.35]{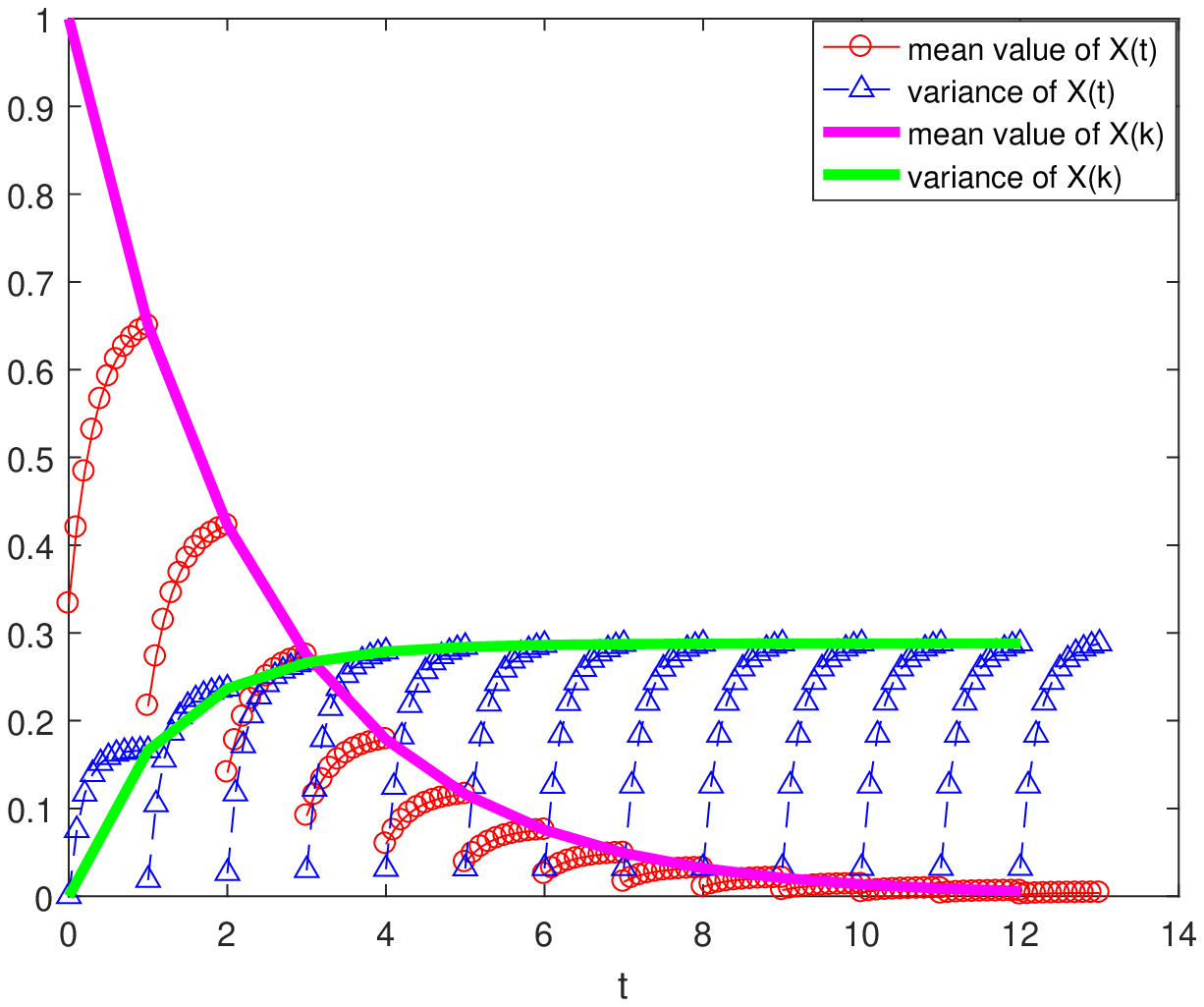}
		\end{minipage}%
	}%
	\subfigure[$\theta_1=2.5$, $\theta_2=1$]{
		\begin{minipage}[t]{0.31\linewidth}
			\centering
			\includegraphics[scale=0.35]{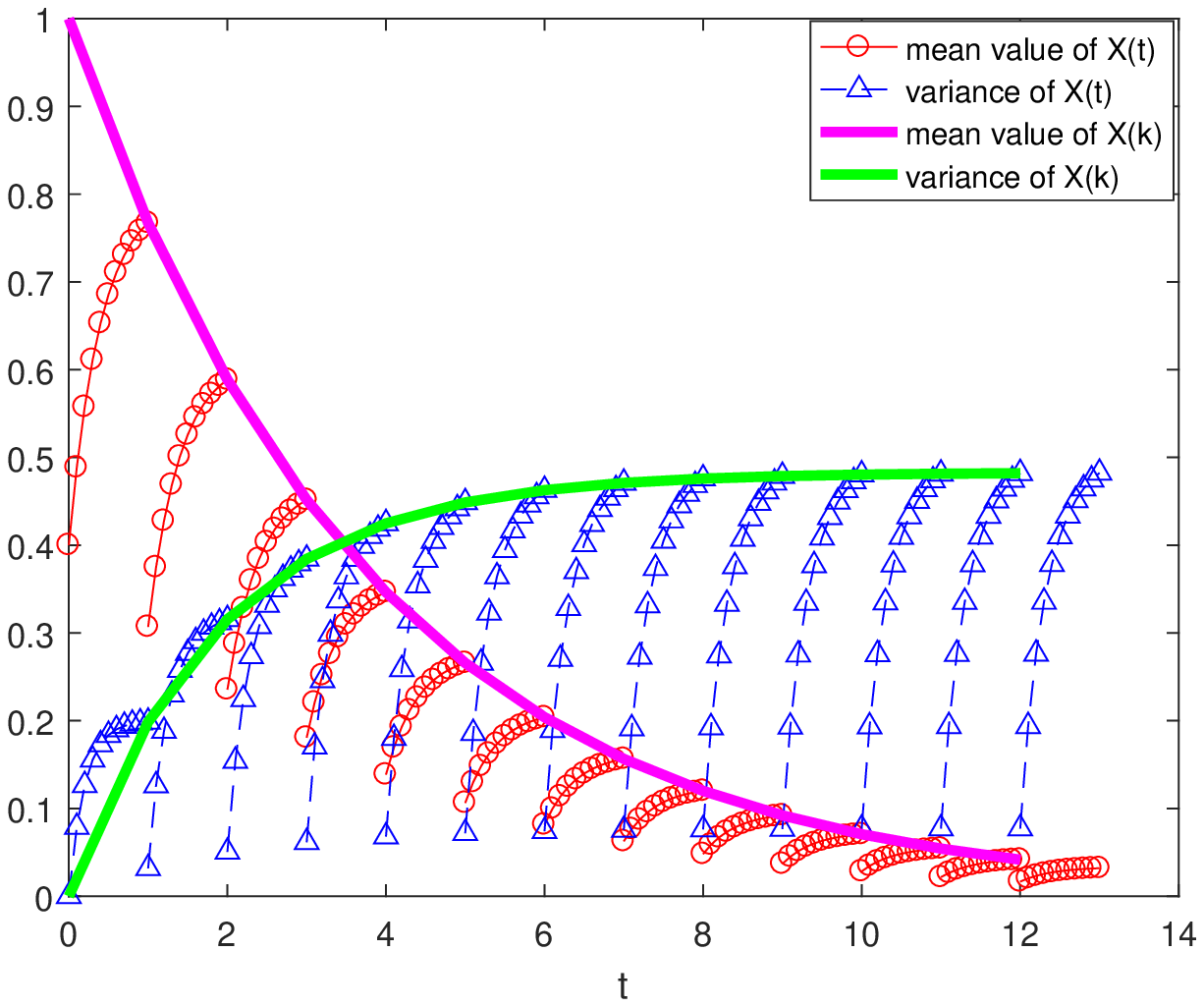}
		\end{minipage}%
	}%
	\subfigure[$\theta_1=2.5$, $\theta_2=-1$]{
		\begin{minipage}[t]{0.31\linewidth}
			\centering
			\includegraphics[scale=0.35]{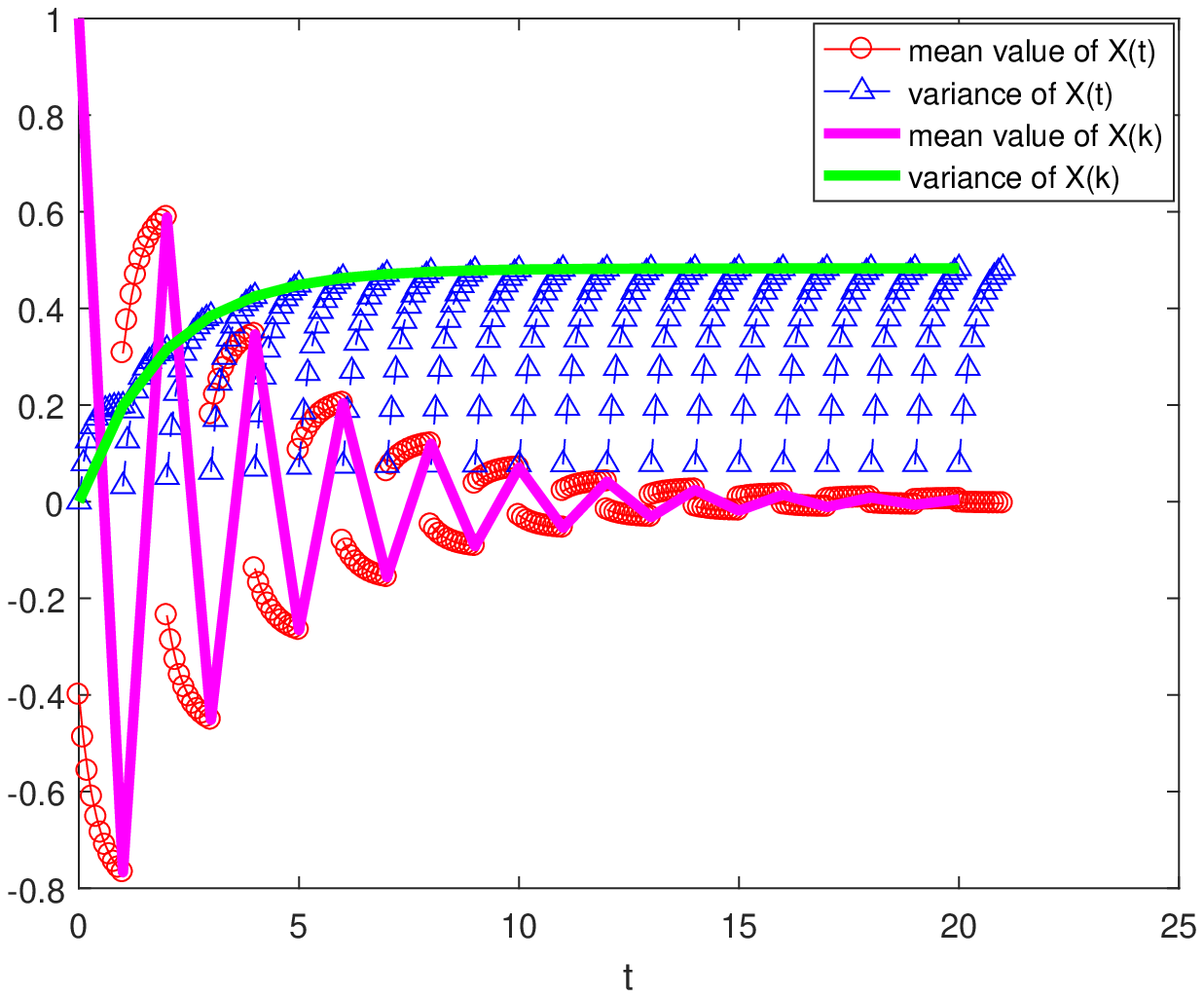}
		\end{minipage}
	}%
	
	\centering
	\caption{The expectations and variances of $X(t)$ and $X(k)$}\label{Gassian}
\end{figure}

\begin{figure}[htbp]
	\centering
	\subfigure[$\theta_1=3$, $\theta_2=1$]{
		\begin{minipage}[t]{0.31\linewidth}
			\centering
			\includegraphics[scale=0.35]{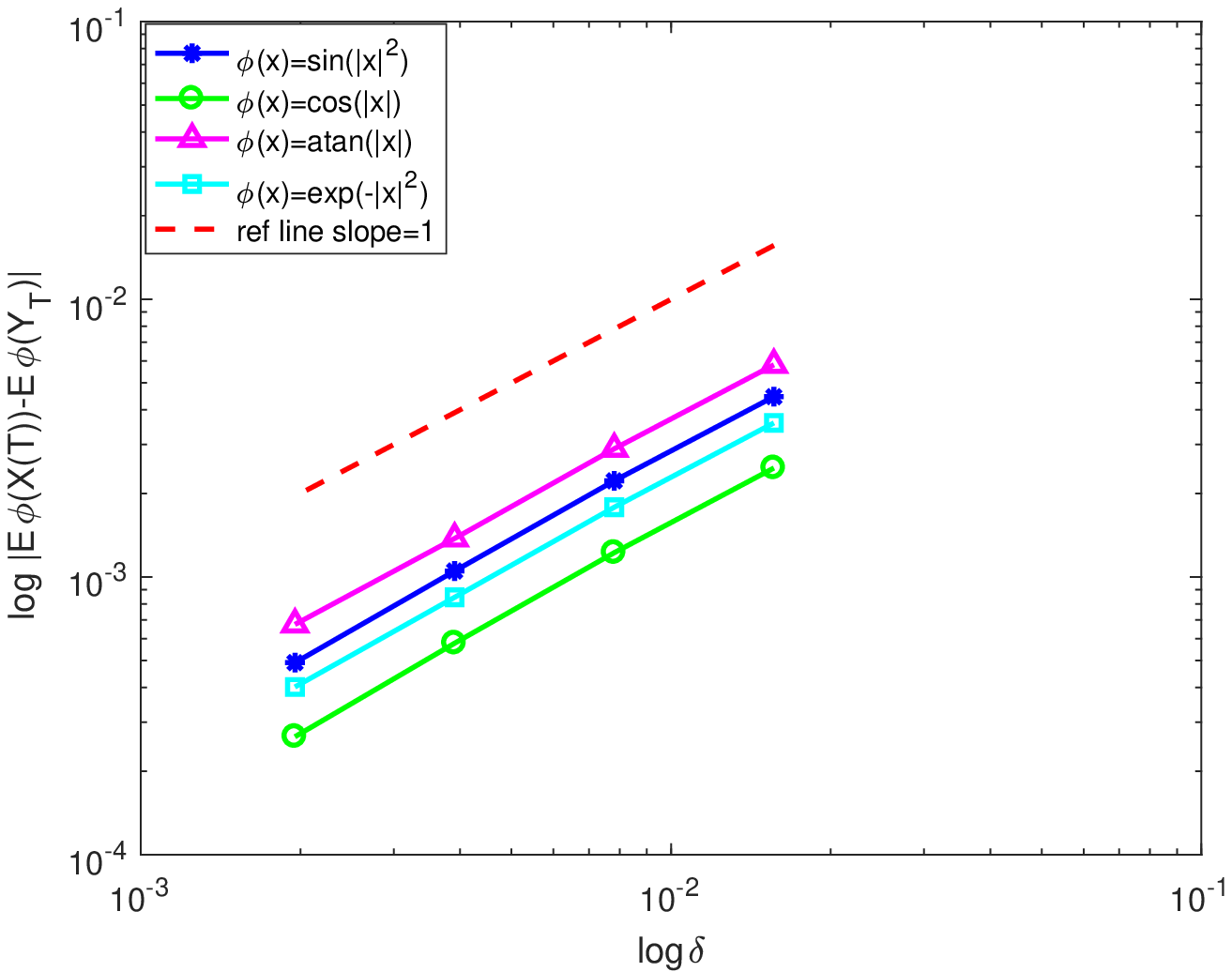}
		\end{minipage}%
	}%
	\subfigure[$\theta_1=2.5$, $\theta_2=1$]{
		\begin{minipage}[t]{0.31\linewidth}
			\centering
			\includegraphics[scale=0.35]{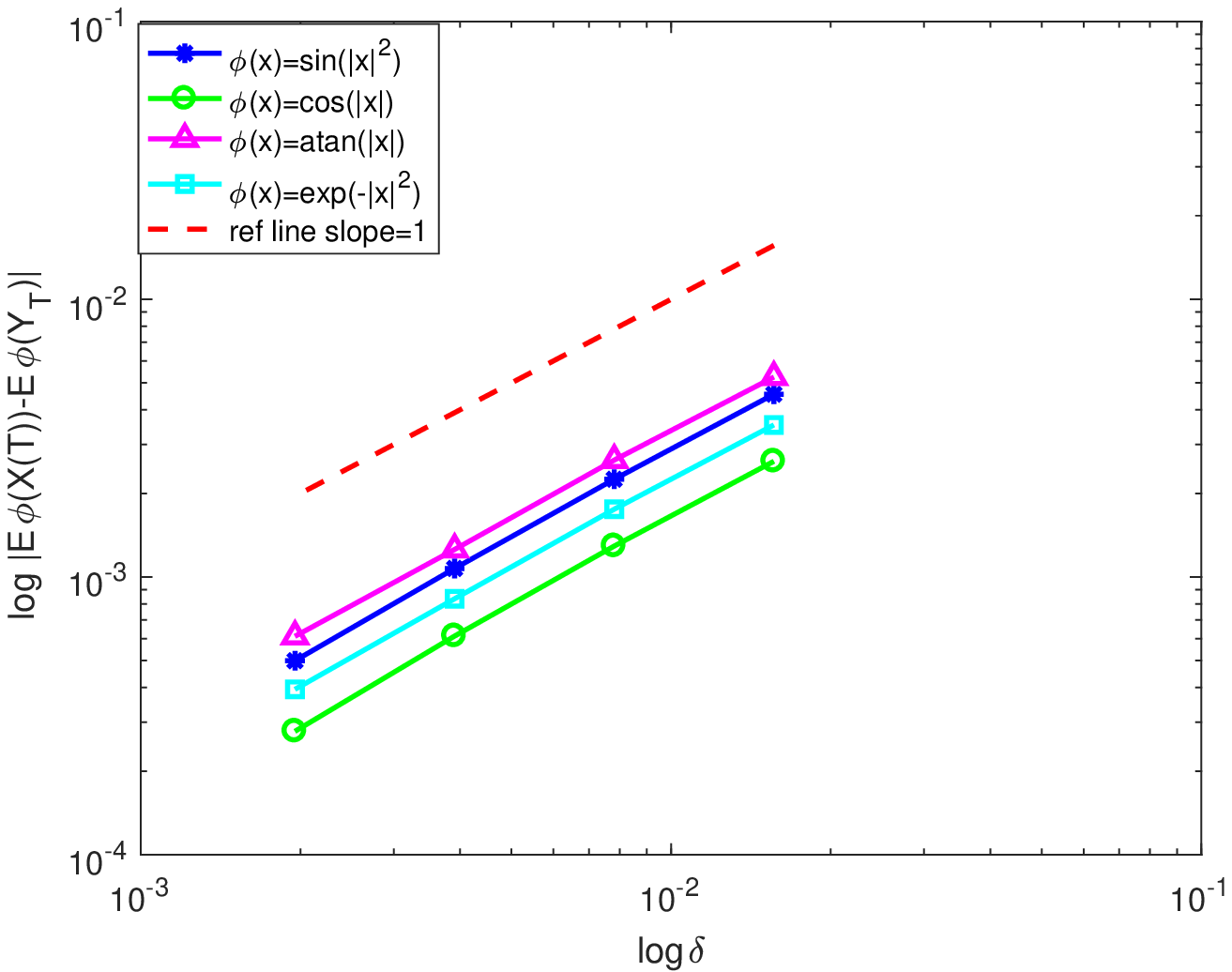}
		\end{minipage}%
	}%
	\subfigure[$\theta_1=2.5$, $\theta_2=-1$]{
		\begin{minipage}[t]{0.31\linewidth}
			\centering
			\includegraphics[scale=0.35]{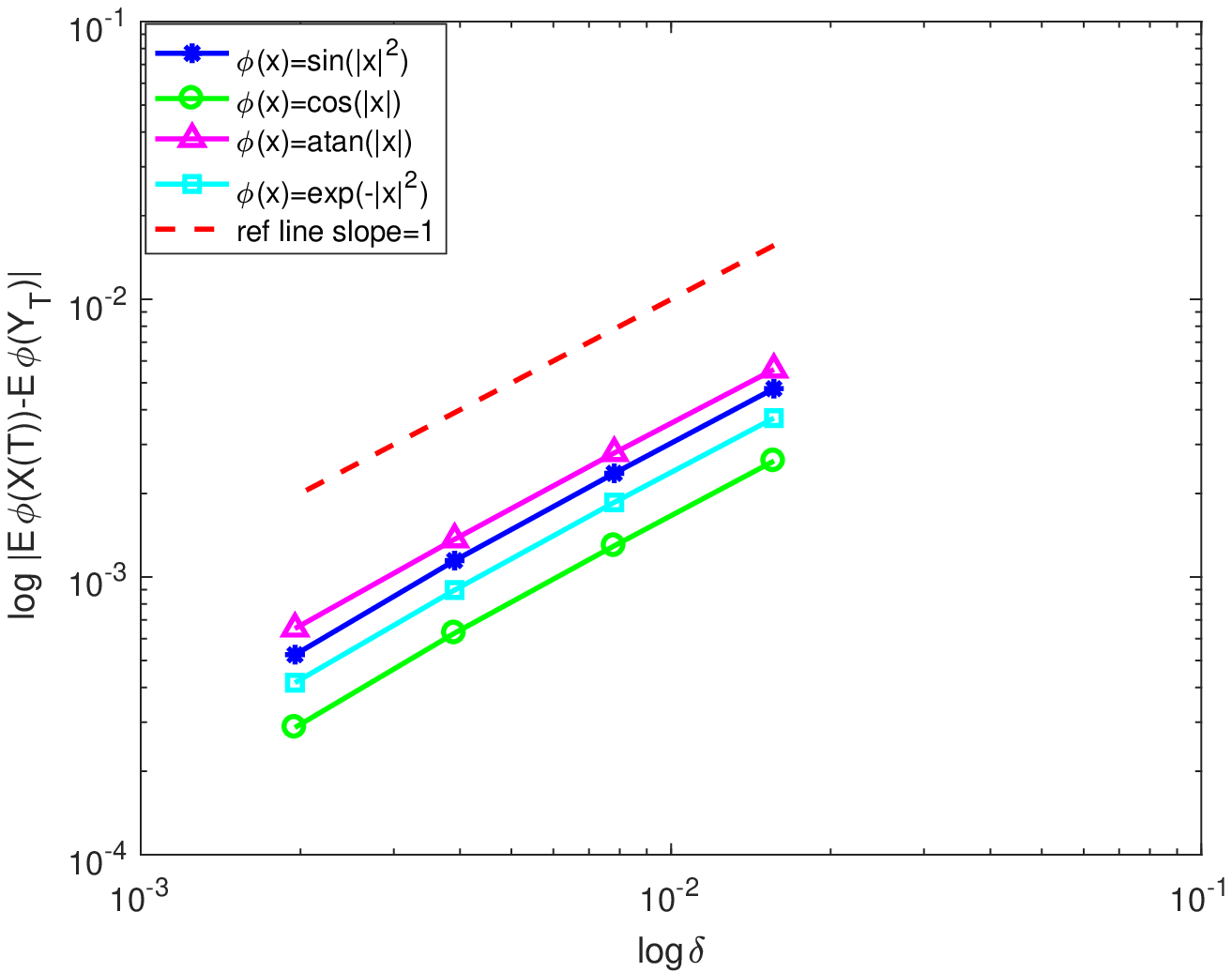}
		\end{minipage}
	}%
	
	\centering
	\caption{Order of weak convergence of BE method }\label{order of weak convergence}
\end{figure}

Next the weak convergence order of the BE method is tested. In fact, the solution of \eqref{SDEPCA-linear} can be expressed as
$$X(t)=x\mu(1)^k\mu(\{t\})+\sum_{i=1}^{k}\mu(\{t\})\mu(1)^i
 \int_{i-1}^ie^{-\theta_1(i-s)}dB(s)+\int_k^te^{-\theta_1(t-s)}dB(s).$$
Let $T=5$. We create 1000 discretized Brownian paths over $[0,T]$ with a small step-size $\bar{\delta}= 2^{-11}$ and
 approximate the stochastic integral in the exact solution above using the Euler method with this small step-size.
 We also compute the numerical solutions of the BE method using 4 different step-sizes $\delta = 2^{-6}, 2^{-7}, 2^{-8}, 2^{-9}$ on the same Brownian path at $T=5$. Moreover, we choose 4 different test functions $\phi(x)=\sin(|x|^2)$, $\phi(x)=\cos(|x|)$, $\phi(x)=\arctan(|x|)$ and $\phi(x)=e^{-|x|^2}$ as the test functions for weak convergence. Fig. \ref{order of weak convergence} plots
 the weak errors $\mathbb{E}|\phi(X(T))-\phi(Y_T)|$ against $\delta$ on a log-log scale, where $X(T)$ and $Y_T$ denote the
 exact and numerical solutions at the endpoint $T$, respectively. The red dashed line represents a reference line with slope 1. From Fig. \ref{order of weak convergence}, it is observed that the BE method is convergent in the weak sense with order 1.

\begin{figure}[htbp]
	\centering
	\subfigure[$\phi=\arctan(|x|)$]{
		\begin{minipage}[t]{0.31\linewidth}
			\centering
			\includegraphics[scale=0.35]{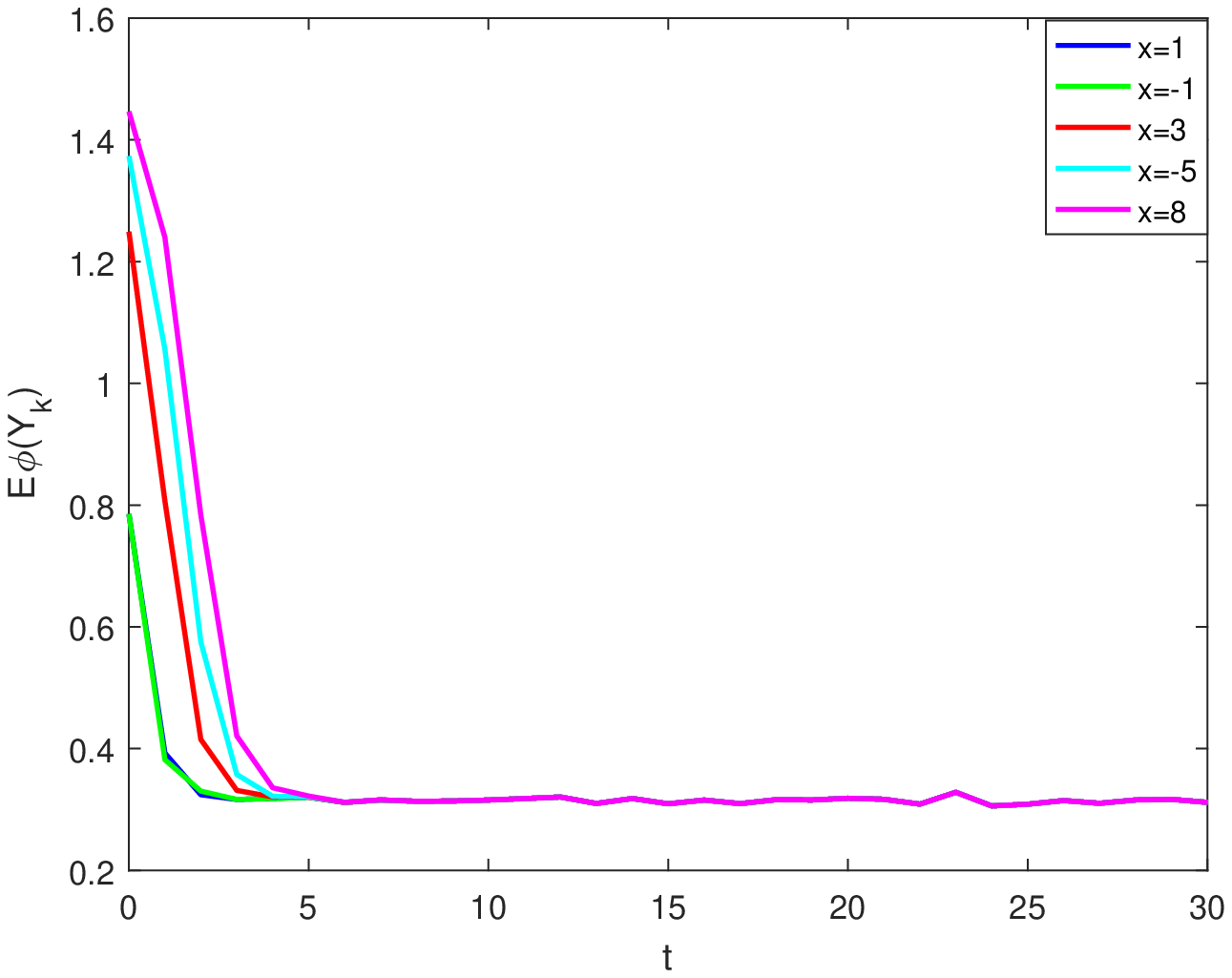}
		\end{minipage}%
	}%
	\subfigure[$\phi=\cos(|x|)$]{
		\begin{minipage}[t]{0.31\linewidth}
			\centering
			\includegraphics[scale=0.35]{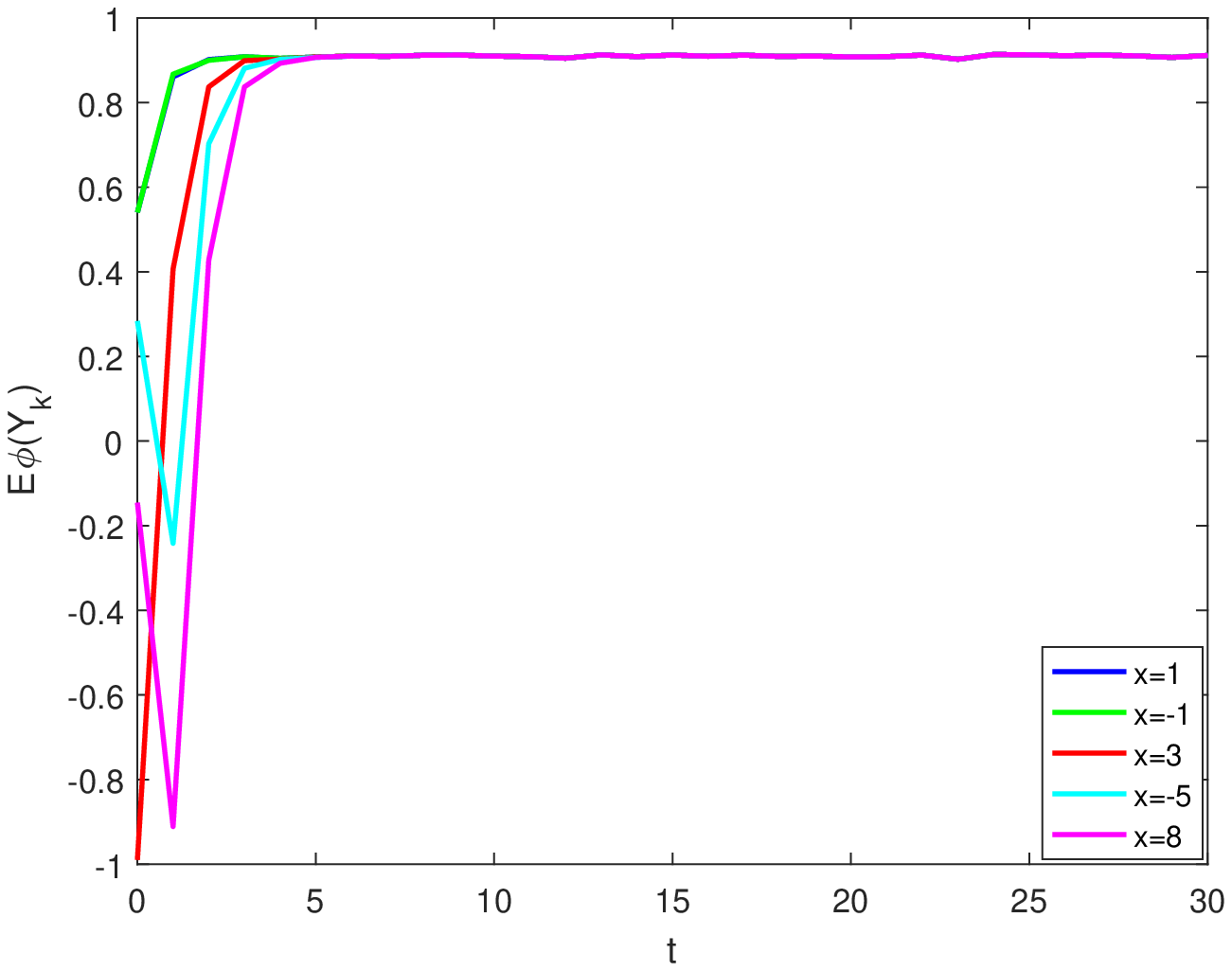}
		\end{minipage}%
	}%
	\subfigure[$\phi=\sin(|x|^2)$]{
		\begin{minipage}[t]{0.31\linewidth}
			\centering
			\includegraphics[scale=0.35]{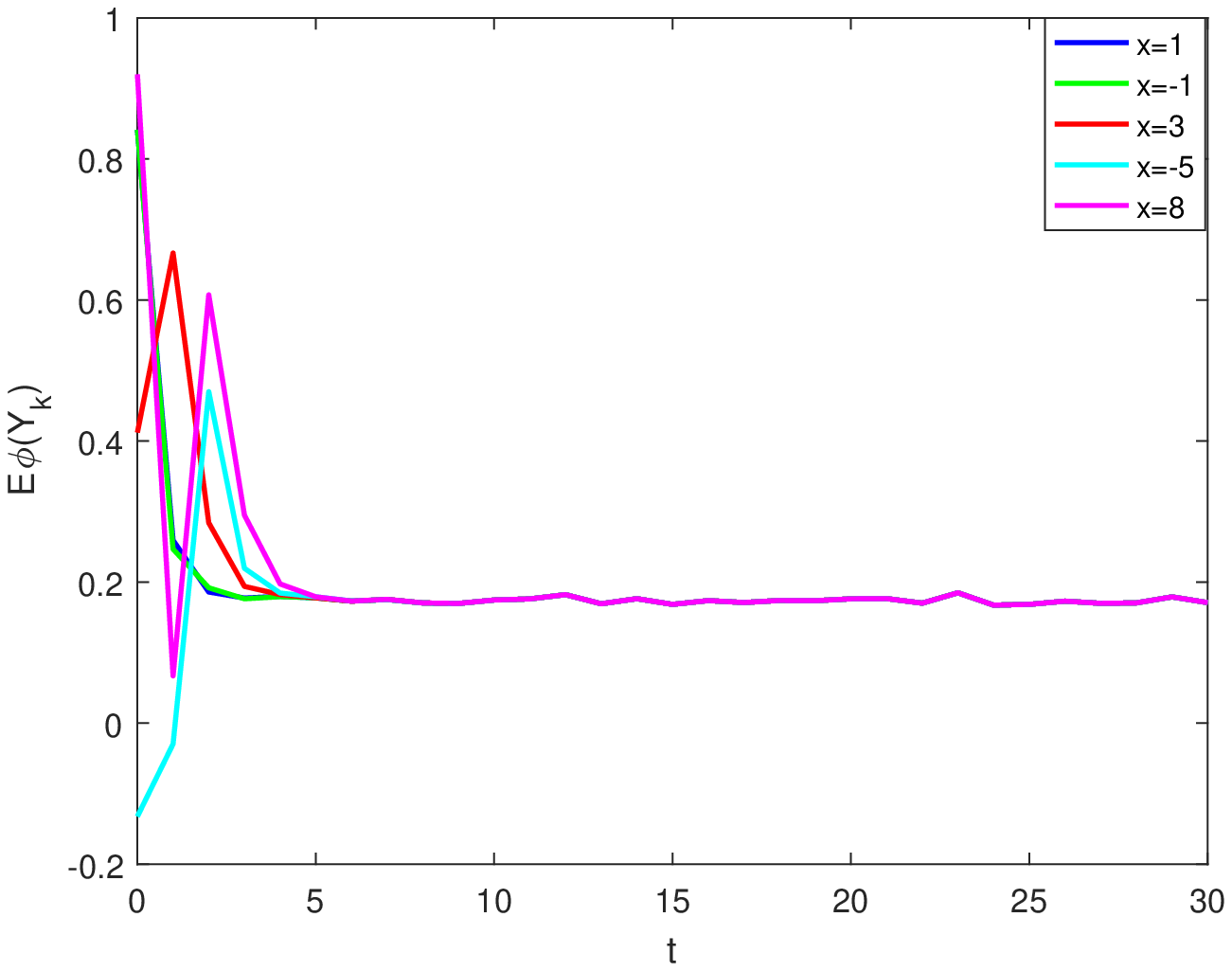}
		\end{minipage}
	}%
	
	\centering
	\caption{The mean values of $Y_k$ with different initial data }\label{mean value of phi}
\end{figure}

Then we consider the longtime behavior of the Markov chain  $\{Y_k\}_{k\in\mathbb{N}}$. Theorem 3.15 shows that $\mathbb{E}\phi(Y_k^{0,x})$ converges exponentially to the ``spatial" average of $\phi$ with different initial data, i.e $Y_k$ is strongly mixing, and this implies the ergodicity of $Y_k$. In this test, we let $\theta_1=3$ and $\theta_2=1$ and choose three test functions (a) $\phi(x)=\arctan(|x|)$, (b) $\phi(x)=\cos(|x|)$ and (c) $\phi(x)=\sin(|x|^2)$ to compute $\mathbb{E}\phi(Y_k^{0,x})$. Fig. \ref{mean value of phi} shows the mean value of $\phi(Y_k^{0,x})$ started from 5 different initial data. As can be seen from the figure, for each $\phi$, $\mathbb{E}\phi(Y_k^{0,x})$ converges exponentially to the same value. 

\textbf{Example 2.} Consider the following 1-dimensional nonlinear SDE with PCAs driven by multiplicative noise
\begin{equation}\label{NL-SDEPCAs}
\begin{cases}
dX(t)=(-X(t)^3-10X(t)+2X([t])+1)dt+(aX(t)+bX([t]))dB(t)\\
X_0=x,
\end{cases}
\end{equation}
where $x=2$ and $a$, $b$ are two parameters. Firstly, we verify the weak convergence of the BE method on a finite time interval $[0,T]$.
Let $T=6$ and we create 2000 discretized Brownian paths over $[0,T]$ with a small step-size $\bar{\delta}= 2^{-11}$. Since the exact solution can not be obtained, we use the numerical solution of the split-step backward Euler method with $\bar{\delta}= 2^{-11}$ as the ``exact solution".
We also compute the numerical solutions of the BE method using 4 different step-sizes $\delta = 2^{-6}, 2^{-7}, 2^{-8}, 2^{-9}$ on the same Brownian path. Let $X(T)$ and $Y_T$ denote the
exact and numerical solutions at the endpoint $T$, respectively.  And three sets of $a,b$ are tested. Fig. \ref{order of weak convergence multiple} plots
the weak errors $\mathbb{E}|\phi(X(T))-\phi(Y_T)|$ against $\delta$ on a log-log scale with 4 different kinds
of test functions $\phi(x)=\sin(|x|^2+\pi/2)$, $\phi(x)=\cos(|x|)$, $\phi(x)=\arctan(|x|^2)$ and $\phi(x)=e^{-|x|^2}$. The red dashed line represents a reference line with slope 1. As can be observed from Fig. \ref{order of weak convergence multiple}, the BE method converges in the weak sense with order 1.

\begin{figure}[htbp]
	\centering
	\subfigure[$a=1$, $b=0$]{
		\begin{minipage}[t]{0.31\linewidth}
			\centering
			\includegraphics[scale=0.35]{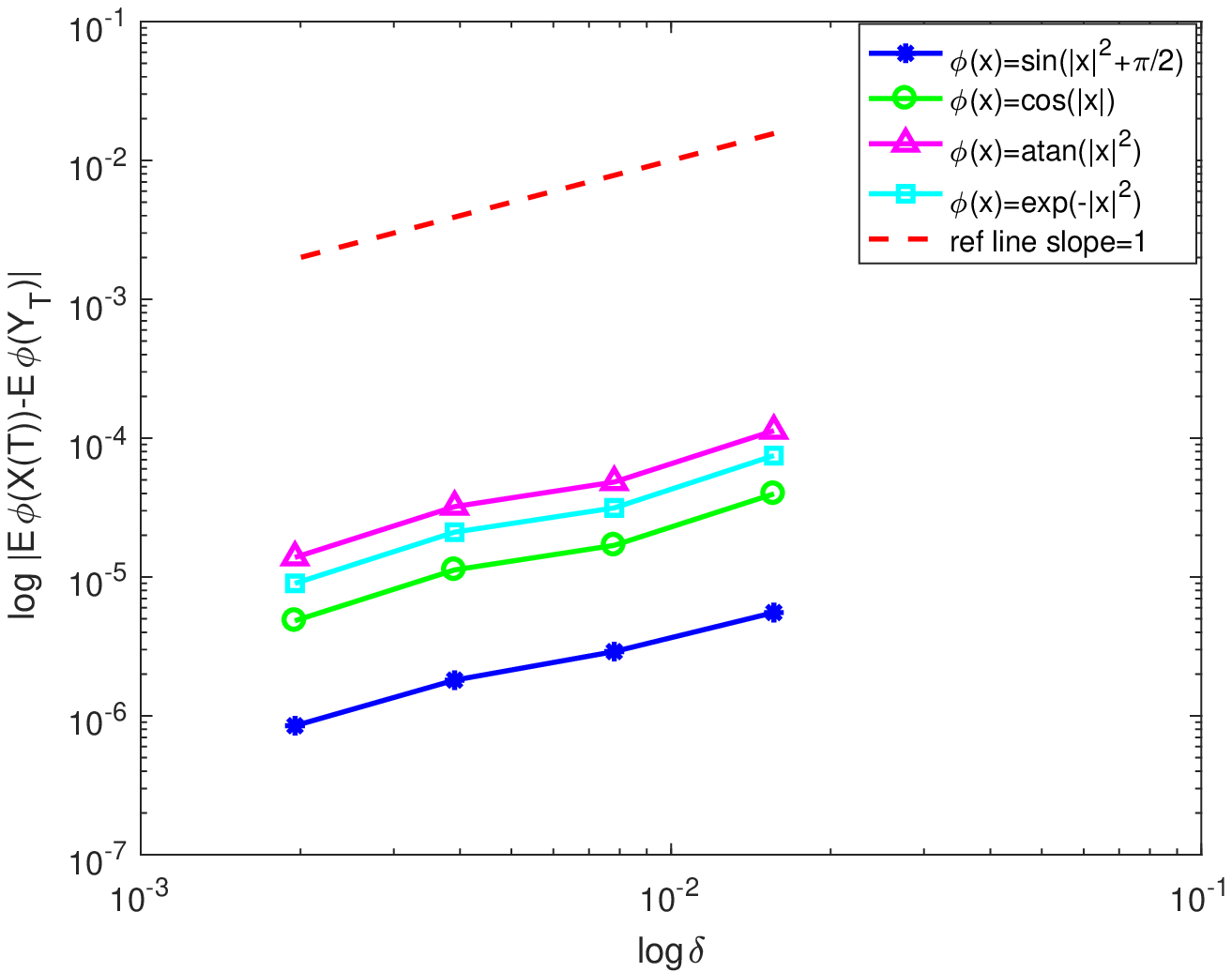}
		\end{minipage}%
	}%
	\subfigure[$a=0$, $b=1$]{
		\begin{minipage}[t]{0.31\linewidth}
			\centering
			\includegraphics[scale=0.35]{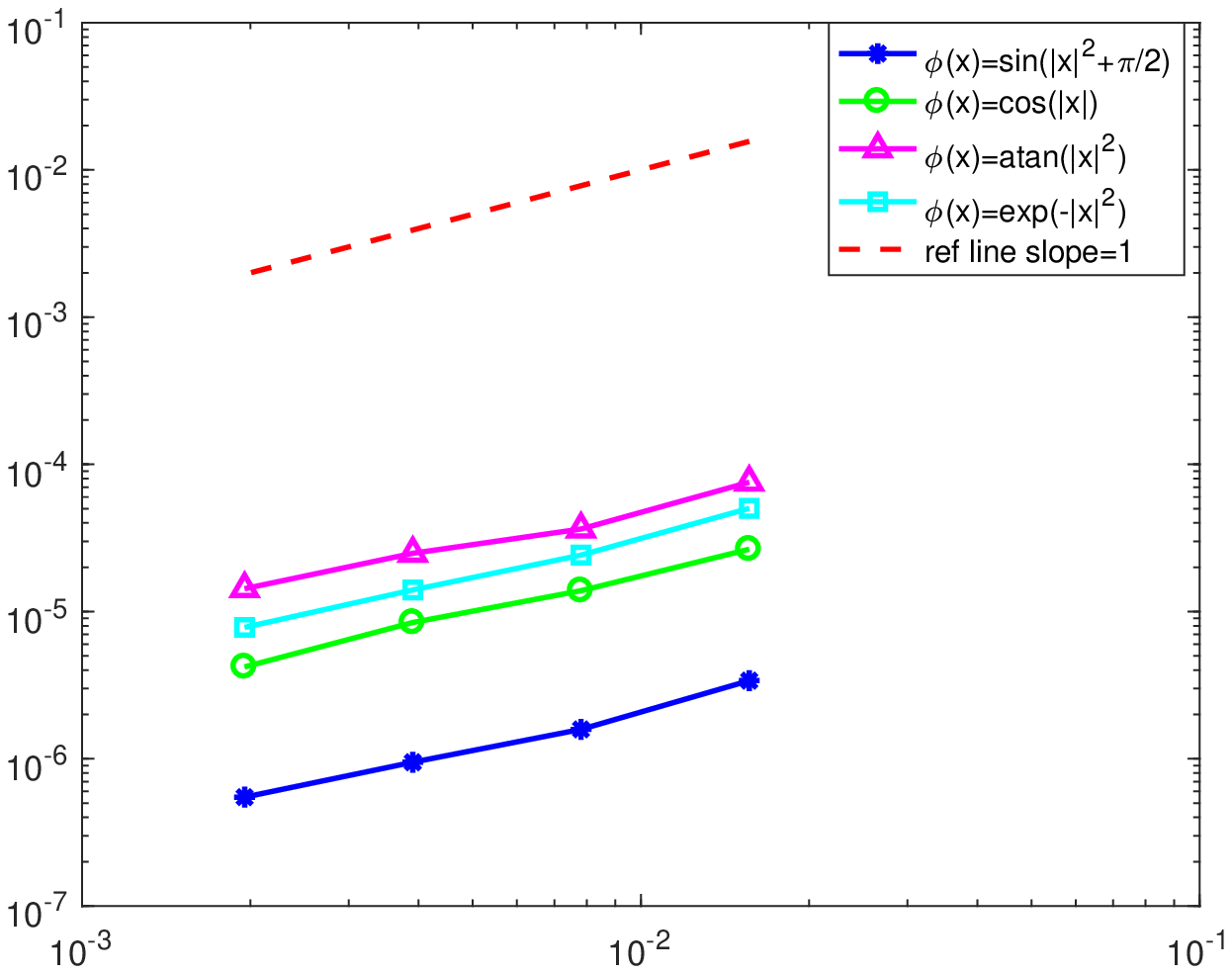}
		\end{minipage}%
	}%
	\subfigure[$a=1$, $b=1$]{
		\begin{minipage}[t]{0.31\linewidth}
			\centering
			\includegraphics[scale=0.35]{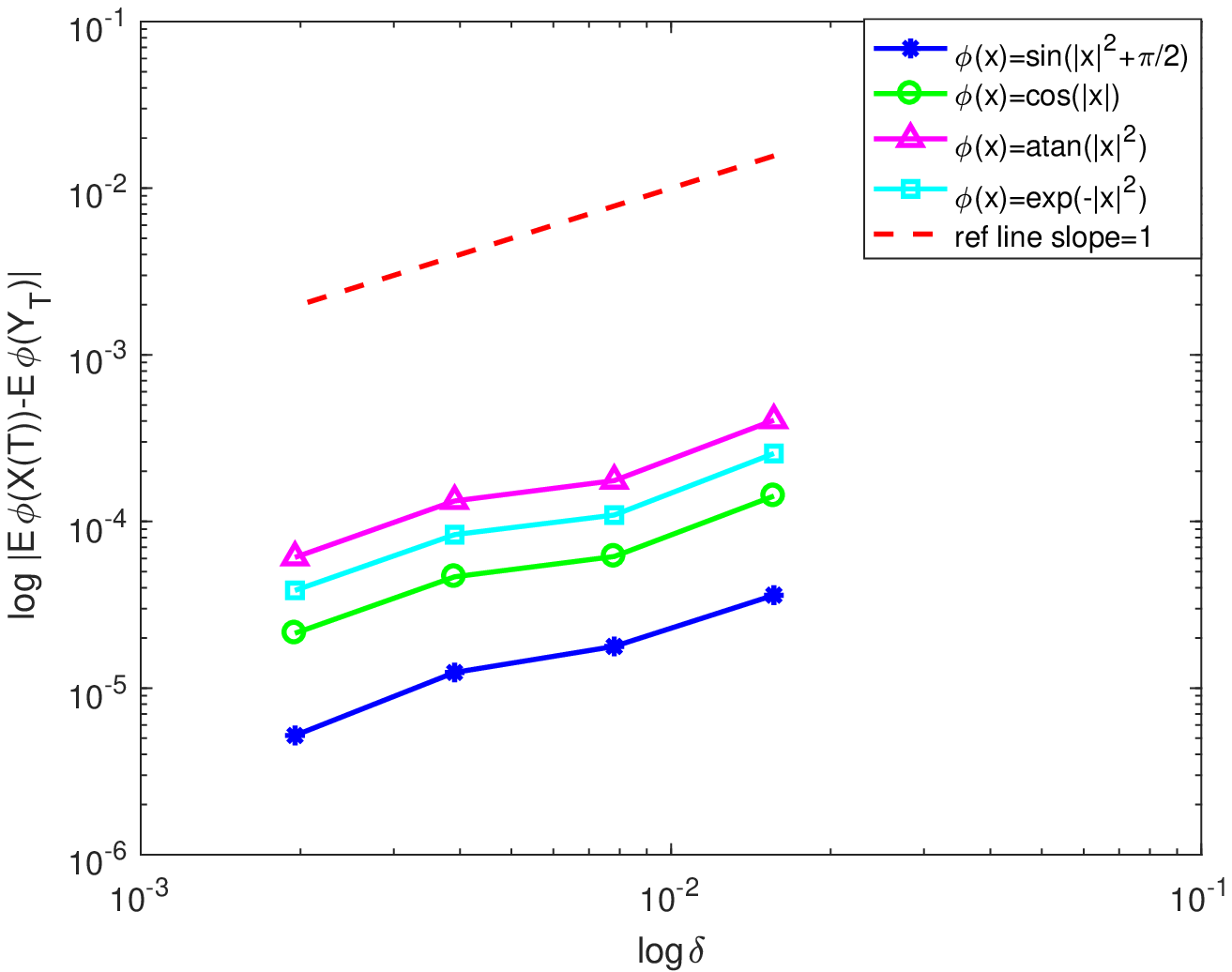}
		\end{minipage}
	}%
	
	\centering
	\caption{Order of weak convergence of BE method }\label{order of weak convergence multiple}
\end{figure}

Finally the longtime behavior of the Markov chain  $\{Y_k\}_{k\in\mathbb{N}}$ is considered. In this simulation, we take $a=1$ and $b=1$ for example and choose three  test functions (a) $\phi(x)=\arctan(|x|)$, (b) $\phi(x)=\sin(|x|^2)$ and (c) $\phi(x)=e^{-|x|^2}$. Fig. \ref{average of phi} plots the mean value of $\phi(Y_k^{0,x})$ started from 5 different initial data. It is observed that, for each $\phi$, $\mathbb{E}\phi(Y_k^{0,x})$ is exponentially convergent as $k$ tends infinity, which verifies the theoretical results.

\begin{figure}[htbp]
	\centering
	\subfigure[$\phi=\arctan(|x|)$]{
		\begin{minipage}[t]{0.31\linewidth}
			\centering
			\includegraphics[scale=0.35]{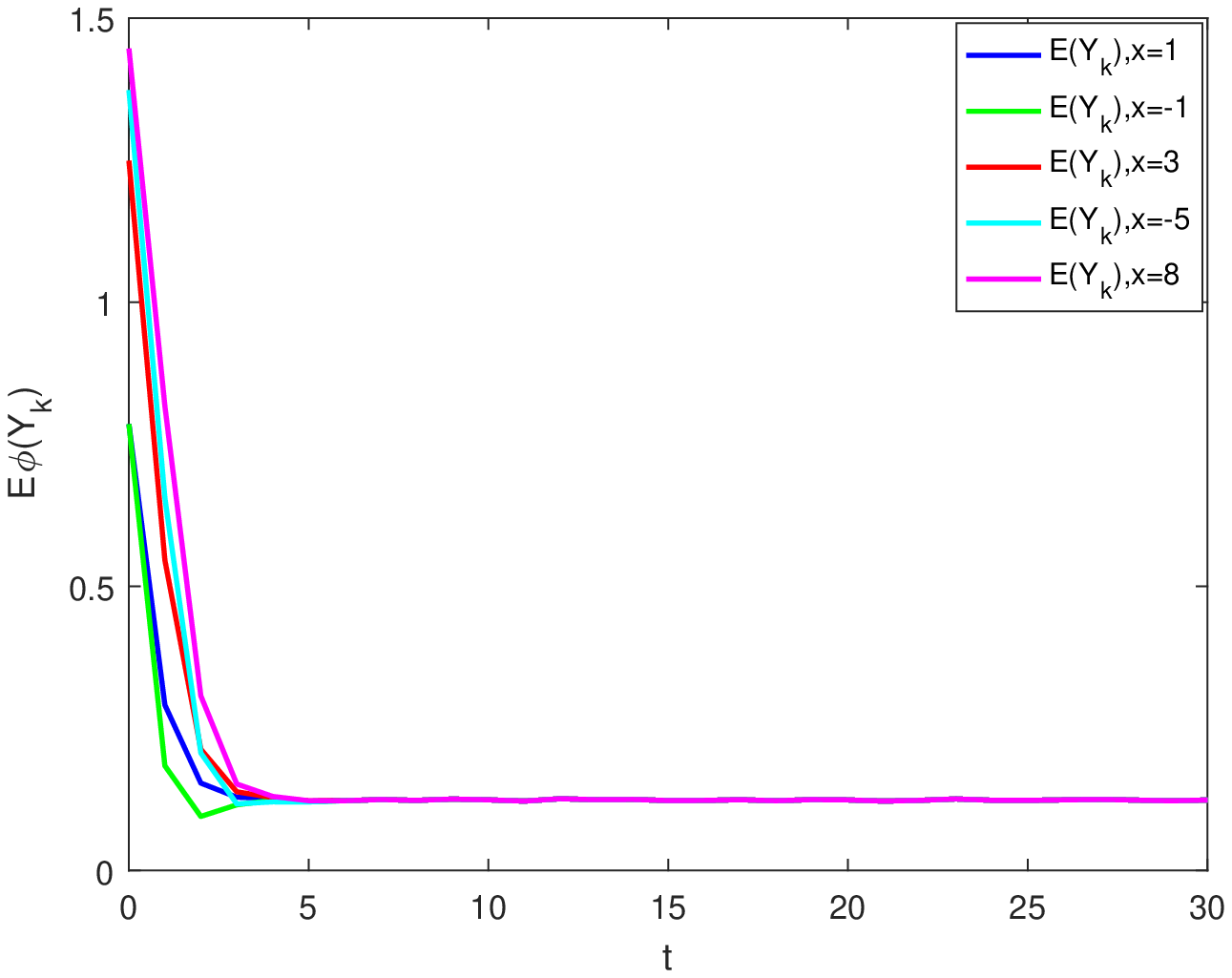}
		\end{minipage}%
	}%
	\subfigure[$\phi=\sin(|x|^2)$]{
		\begin{minipage}[t]{0.31\linewidth}
			\centering
			\includegraphics[scale=0.35]{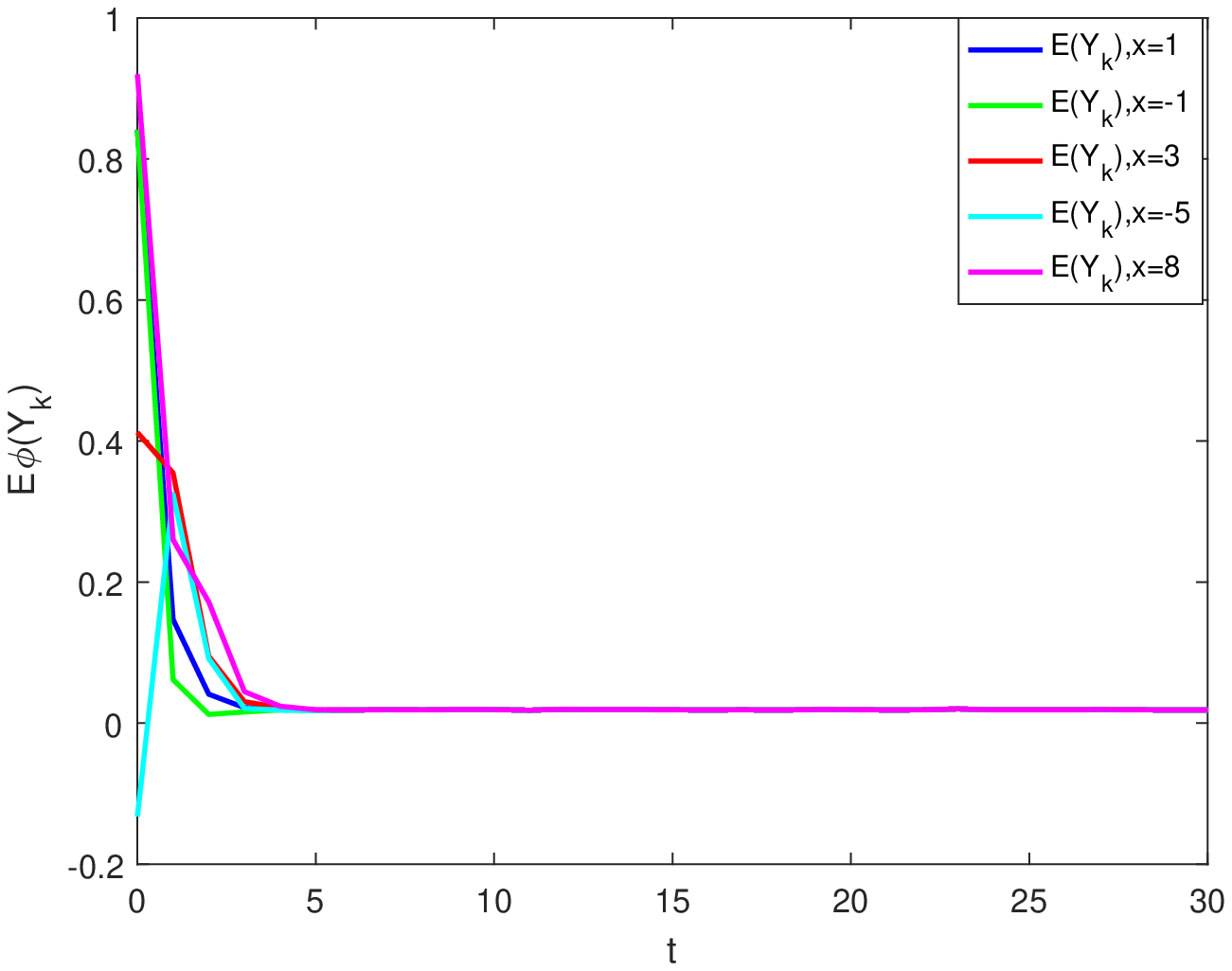}
		\end{minipage}%
	}%
	\subfigure[$\phi=e^{-|x|^2}$]{
		\begin{minipage}[t]{0.31\linewidth}
			\centering
			\includegraphics[scale=0.35]{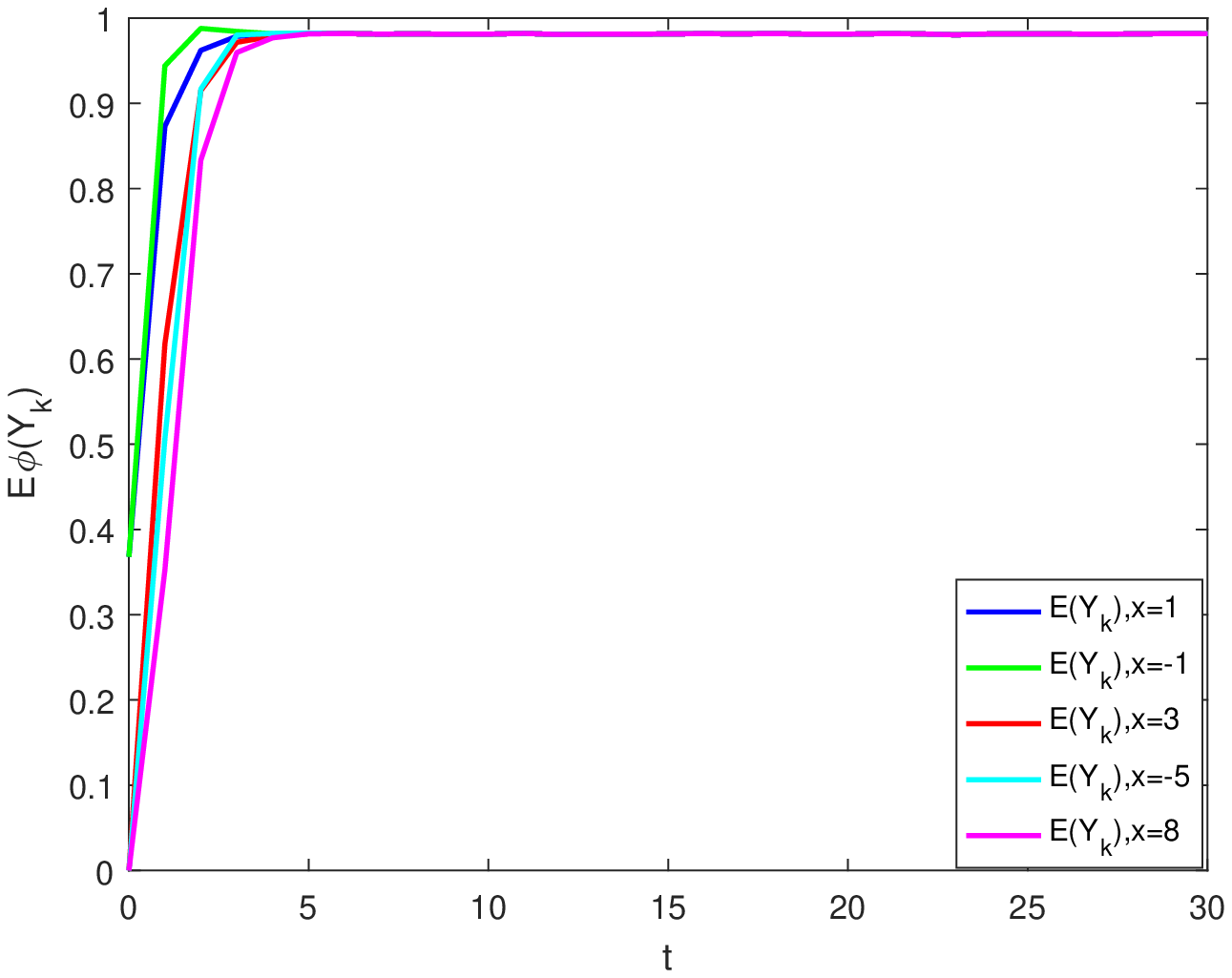}
		\end{minipage}
	}%
	
	\centering
	\caption{The mean values of $Y_k$ with different initial data }\label{average of phi}
\end{figure}


\end{document}